\documentclass[11pt]{amsart}
\usepackage{etoolbox}
\newtoggle{for_journal}
\toggletrue{for_journal}
\usepackage[utf8]{inputenc}

\usepackage[foot]{amsaddr}

\usepackage[pdftex,dvipsnames]{xcolor}
\usepackage{mathpazo}
\usepackage{hyperref}
\hypersetup{%
  colorlinks,               
  linkcolor={red!50!black}, 
  citecolor={blue!50!black},
  urlcolor={blue!80!black}  
}

\usepackage{amssymb}
\usepackage{amsmath}
\usepackage{amsfonts}
\usepackage{amsthm}
\usepackage{bbm}

\usepackage{csquotes}

\usepackage{relsize}

\usepackage{graphicx}
\usepackage{epstopdf}

\usepackage{indentfirst}
\usepackage{lipsum}

\usepackage{marvosym}
\usepackage{booktabs}
\usepackage{lscape}

\usepackage{mathtools}
\usepackage{microtype}

\usepackage{array}
\usepackage{xargs}

\usepackage{multirow}
\usepackage{multicol}

\usepackage[backgroundcolor=lime,linecolor=lime,bordercolor=white,figcolor=white]{todonotes}
\newcommandx{\todocomment}[2][1=]{\todo[linecolor=Plum,backgroundcolor=Plum!25,bordercolor=Plum,#1]{{\bf TODO (comment):} #2}}
\newcommandx{\todourgent}[2][1=]{\todo[linecolor=red,backgroundcolor=red!25,bordercolor=red,#1]{{\bf TODO (urgent):} #2}}
\newcommandx{\todohidden}[2][1=]{}

\usepackage{enumerate}
\usepackage{diagbox}
\usepackage{colortbl}

\usepackage{iftex}

\usepackage[nameinlink]{cleveref}
\usepackage{cancel}
\usepackage{diagbox}

\newcommand{\Hom}{\operatorname{Hom}}
\newcommand{\End}{\operatorname{End}}

\newcommand{\id}{{\operatorname{id}}}

\newcommand{\Sl}{\mathfrak{sl}}
\newcommand{\Gl}{\mathfrak{gl}}

\newcommand{\Z}{\mathbb{Z}}

\newcommand{\C}{\mathbb{C}}

\newcommand{\Q}{\mathbb{Q}}

\newcommand{\mathscr}{\mathcal}

\definecolor{rdylgn60}{HTML}{D73027}
\definecolor{rdylgn61}{HTML}{FC8D59}
\definecolor{rdylgn62}{HTML}{FEE08B}
\definecolor{rdylgn63}{HTML}{D9EF8B}
\definecolor{rdylgn64}{HTML}{91CF60}
\definecolor{rdylgn65}{HTML}{1A9850}
\definecolor{rdylgn66}{HTML}{006837}
\definecolor{ca1}{HTML}{72497F}
\definecolor{ca2}{HTML}{5C7B78}
\definecolor{cb1}{HTML}{0862A0}
\definecolor{cb2}{HTML}{CD9B3E}
\definecolor{cc1}{HTML}{C55824}
\definecolor{cc2}{HTML}{288338}
\definecolor{cd}{HTML}{CF838A}
\definecolor{ce}{HTML}{FFAF3A}
\definecolor{cf}{HTML}{AFBF5A}
\definecolor{cg}{HTML}{7ACF99}
\definecolor{ch}{HTML}{8AB1CF}
\definecolor{ci}{HTML}{3B58CF}
\definecolor{cj}{HTML}{736163}

\definecolor{lightgray}{HTML}{DDDDDD}

\makeatletter
\@ifclassloaded{beamer}{%
}{%
\usepackage[margin=1.3in]{geometry}
\makeatother
\theoremstyle{plain}
\newtheorem{theorem}{Theorem}[section]

\newtheorem{definition}[theorem]{Definition}
\newtheorem{example}[theorem]{Example}
\newtheorem{lemma}[theorem]{Lemma}

\newtheorem{corollary}[theorem]{Corollary}
\newtheorem{conjecture}{Conjecture}
\theoremstyle{remark}
\newtheorem{remark}{Remark}[section]
}

\numberwithin{equation}{section}

\usepackage{tikz}
\usetikzlibrary{shapes,arrows,calc}
\usetikzlibrary{fadings}

\usetikzlibrary{decorations.pathreplacing,decorations.pathmorphing,decorations.markings}

\usepackage{tikz-cd}
\ifLuaTeX%
\usetikzlibrary{graphdrawing}
\usetikzlibrary{graphdrawing.layered}
\fi

\tikzset{%
  spot/.style={color=black, thin, dashed},
  rline/.style={color=green, line width=2pt},
  sline/.style={color=blue, line width=2pt},
  tline/.style={color=red, line width=2pt},
  uline/.style={color=green, line width=2pt},
  line/.style={color=#1, line width=2pt},
  line/.default=blue,
  rdot/.style={color=green, thin, fill},
  sdot/.style={color=blue, thin, fill},
  tdot/.style={color=red, thin, fill},
  udot/.style={color=green, thin, fill},
  dot/.style={color=#1, thin, fill},
  dot/.default=blue
}

\newcommand{\smolyng}[1]{\vcenter{\hbox{\Yboxdim{4pt}\yng #1}}}
\newcommand{\tinyyng}[1]{\vcenter{\hbox{\Ylinethick{1pt}\Yboxdim{3pt}\yng #1}}}
\newcommand{\colroot}[1]{{\Ylinecolour{red!60!black}\tinyyng{#1}}}

\newcommand{\colwei}[1]{{\Ylinecolour{green!60!black}\tinyyng{#1}}}
\newcommand{\colweib}[1]{{\Ylinecolour{green!60!black}\smolyng{#1}}}

\newcommand{\egfrac}[3]{\frac{[c_{#1#2}^{\,\colwei{#3}}]}{[c_{#1#2}^{\,\colwei{#3}}-1]}}

\newcommand{\TL}{{\rm TL}}
\newcommand{\JW}{{\rm JW}}
\newcommand{\Y}{{\mathbb{Y}}}
\newcommand{\wt}{\operatorname{wt}}

\newcommand{\TLcat}{{\mathscr{T\!\!L}}}

\newcommand{\qn}[1]{{[#1]}}

\newcommand{\gaussianquant}{\genfrac{[}{]}{0pt}{}}
\newcommand{\emptypart}{\pmb{\varnothing}}

\let\emph\relax % there's no \RedeclareTextFontCommand
\DeclareTextFontCommand{\emph}{\bfseries\em}

\usepackage{cite}
\makeatletter
\newcommand{\citecomment}[2][]{\citen{#2}#1\citevar}
\newcommand{\citeone}[1]{\citecomment{#1}}
\newcommand{\citetwo}[2][]{\citecomment[,~#1]{#2}}
\newcommand{\citevar}{\@ifnextchar\bgroup{;~\citeone}{\@ifnextchar[{;~\citetwo}{]}}}
\newcommand{\citefirst}{\@ifnextchar\bgroup{\citeone}{\@ifnextchar[{\citetwo}{]}}}

\makeatother

\author{Stuart Martin and Robert A. Spencer}
\email{sm137@cam.ac.uk}
\thanks{\textit{ORCiDs:} Stuart Martin 0000-0002-7424-8397, Robert A. Spencer 0000-0003-2432-3392}
\thanks{\textit{Corresponding author:} Robert A. Spencer}
\address{Department of Pure Mathematics and Mathematical Statistics; University of Cambridge; Cambridge; CB3 0WA; UK}
\email{maths@robertandrewspencer.com}

\date{}

\usepackage{xparse}
\NewDocumentCommand{\defterm}{O{#2}m}{\textit{\textbf{#2}}\index{#1}}

\usepackage{enumitem}
\usepackage{genyoungtabtikz}
\usepackage{mathpazo}
\usepackage[super]{nth}
\usepackage{microtype}

\usepackage[framemethod=tikz,hidealllines=true,leftline=true]{mdframed}
\newcommand{\makeframedenv}[2]{%
\surroundwithmdframed[%
  linecolor=#1,%
  linewidth=2pt,%
  skipabove=5pt,%
  leftmargin=-12pt,%
  usetwoside=false,%
  innertopmargin=0pt]{#2}%
}
\makeframedenv{rdylgn60}{definition}
\makeframedenv{rdylgn62}{corollary}
\makeframedenv{rdylgn61}{proposition}
\makeframedenv{rdylgn65}{lemma}
\makeframedenv{rdylgn64}{theorem}
\makeframedenv{purple!20}{example}

\newcommand{\bmu}{{\boldsymbol{\mu}}}

\newcommand{\clasp}{clasp}
\newcommand{\clasps}{clasps}

\newcommand{\qfrac}[2]{\frac{[#1]}{[#2]}}
\newcommand{\qbinom}[2]{\genfrac{[}{]}{0pt}{}{#1}{#2}}
\title{Cell Modules for Type $A$ Webs}

\begin{document}

\begin{abstract}
  We examine the cell modules for the category of type $A_n$ webs and their natural cellular forms.
  We modify the bases of these modules, as prescribed by Elias, to obtain  an orthogonal basis of each cell module.
  Hence, we calculate the determinant of the Gram matrix with respect to such bases.
  
  These Gram determinants are given in terms of intersection forms, computed from certain traces of clasps, namely higher order Jones-Wenzl morphisms.
  Additionally, this modified basis is constructed using these clasps, and each clasp is constructed using traces of smaller clasps.
  
  In \cite{elias_2015}, Elias conjectures a value for these intersection forms and verifies it in types $A_1$, $A_2$ and $A_3$.  This paper concludes with a proof of the conjecture for all type $A_n$.

  {\vspace{1em}\noindent\textbf{Key words:} Webs, Gram determinants, cell modules, lowering and raising operators}{}
\end{abstract}

\maketitle

\section{Introduction}

The category of (tensor products of) fundamental representations of $U_q(\Sl_n)$ comes equipped with many structural features.
Indeed, it is monoidal, cellular and diagrammatic.
These three structures intersect and interact in a manner which permits us to use one to deduce facts about the others.  

The monoidal structure arises from the ability to take tensor products of representations of $U_q(\Sl_n)$. This gives us the structure of a monoid on the representations. It is associative and has an identity: the trivial representation is $\C$. However, for a general choice of $q$, it is not symmetric: $X\otimes Y$ is not isomorphic to $Y \otimes X$.

The category is cellular because it satisfies the conditions of being an OACC~\cite{soergel_bimodule_book} or, more simply, that each algebra $\End(X)$ is cellular in the sense of Graham and Lehrer\cite{graham_lehrer_1996}.

Finally, we say it is diagrammatic, due to the existence of a presentation by planar diagram generators and relations\cite{cautis_kamnitzer_morrison_2014}, known as webs or spiders.

Of the three, the cellular structure is perhaps the most interesting, as it contains the most information.
However, as we shall see, its cellular structure can be built up out of monoidal generators for the objects and diagrammatic generators for the morphisms.
This is the interplay we shall exploit.

\vspace{1em}

\noindent
Cellular categories are equipped with cell modules.
These can be thought of as the cell modules of each of the cellular endomorphism algebras, $\End(X)$.
If the category is semisimple, these cell modules are indecomposable; if
it is not semisimple, they provide us with a complete list of indecomposable objects occurring as their heads.

To be exact, each module inherits a bilinear pairing from the homomorphism space structure and the quotients of the cell modules by this radical (where the quotient is not zero) provides  a complete set of indecomposable modules.

These categories also admit an integral lattice and through this, their modular representation theory can be studied.
For example, we may take the indecomposable cell modules over characteristic zero\footnote{In fact, the characteristic should be a "semisimple mixed characteristic".} and study their decomposition in positive characteristic.
This decomposition is controlled, to the first order, by the bilinear form on the cell modules;  for example, it gives a Jantzen-like filtration.

Since the integral lattice forms a basis of these modules, even under specialisation, we can study certain properties of the bilinear form over semisimple characteristic in order to deduce properties over positive characteristic.
One such property is the determinant of the form, which this paper investigates.

\vspace{1em}

\noindent
In a semisimple cellular category, each object with a weight has an associated idempotent known as a clasp.
Indeed, if $X$ is an object of weight $\lambda$, there is a morphism $\JW_X \in \End X$ which is idempotent and which is killed by left- or right-composition by any morphism factoring through an object of weight less than $\lambda$.

In this paper we will frequently refer to \lq\lq simple\rq\rq\  examples as both inspiration and touchstones.
Our basic example will be that of the Temperley-Lieb category, $\TLcat$.
In this category (which appears in the study of  $U_q(\Sl_2)$) the clasps are known as Jones-Wenzl elements.
Such elements all exist, which is to say that the category is semisimple, if and only if  $q$ is not a root of unity.

\vspace{1em}

\noindent
The diagrammatic category equivalent to the category of fundamental representations of $U_q(\Sl_n)$ is known as the category of webs or spiders.
The concepts were introduced by Kuperberg for the rank 2 Cartan types~\cite{kuperberg_1996}.
Type $A$ webs were developed by Cautis, Kamnitzer and Morrison using a variation on Howe duality~\cite{cautis_kamnitzer_morrison_2014}.
Thereafter Bodish et. al. extended the construction to type $C_n$ \cite{bodish_elias_rose_tatham_2021}.

These presentations are in terms of generators and relations---a quotient of a certain category by a tensor ideal.
However this doesn't describe the category completely. The related problem of determining a basis for the morphism spaces was started by Elias (for type $A_n$)~\cite{elias_2015}, and continued by Bodish for type $C_2$~\cite{bodish_2020} as well as  Bodish and Wu in type $G_2$~\cite{bodish_wu_2021}.

Using a similar construction, it turns out that calculating clasps in the web categories can be done nearly concurrently to finding bases.
Indeed, we now have constructions for clasps in type $A_n$~\cite{elias_2015}, $C_2$~\cite{bodish_2021} and $G_2$~\cite{bodish_wu_2021}.

The well-known construction of Jones-Wenzl idempotents forms an important archetype for these higher order clasps.
Indeed, the Jones-Wenzl idempotents satisfy
\begin{equation}\label{eq:jones-wenzl-def}
    \vcenter{\hbox{
    \begin{tikzpicture}[scale=0.5]
      \draw (0,0) rectangle (2.5,2.5);
      \node at (1.25, 1.25) {$\JW_{n+1}$};
      \draw[very thick] (-0.5, 0.2) -- (0, 0.2);
      \draw[very thick] (-0.5, 2.3) -- (0, 2.3);
      \node at (-0.25, 1.5) {$\vdots$};
      \draw[very thick] (2.5, 0.2) -- (3, 0.2);
      \draw[very thick] (2.5, 2.3) -- (3, 2.3);
      \node at (2.75, 1.5) {$\vdots$};
    \end{tikzpicture}
    }}
    =
    \vcenter{\hbox{
    \begin{tikzpicture}[scale=0.5]
      \draw (0,0) rectangle (2,2);
      \node at (1, 1) {$\JW_{n}$};
      \draw[very thick] (-0.5, 0.2) -- (0, 0.2);
      \draw[very thick] (-0.5, 1.7) -- (0, 1.7);
      \node at (-0.25, 1.2) {$\vdots$};
      \draw[very thick] (2.5, 0.2) -- (2, 0.2);
      \draw[very thick] (2.5, 1.7) -- (2, 1.7);
      \node at (2.25, 1.2) {$\vdots$};
      \draw[very thick] (-0.5, -0.5) -- (2.5, -0.5);
    \end{tikzpicture}
    }}
    - \frac{[n]}{[n+1]}
    \vcenter{\hbox{
    \begin{tikzpicture}[scale=0.5]
      \draw (-2.75,0) rectangle (-0.75,2);
      \node at (-1.75, 1) {$\JW_n$};
      \draw (0.75,0) rectangle (2.75,2);
      \node at (1.75, 1) {$\JW_n$};
      \draw[very thick] (-3.25, -0.5) -- (-0.5, -0.5);
      \draw[very thick] (3.25, -0.5) -- (0.5, -0.5);
      \draw[very thick] (-3.25, 0.25) -- (-2.75, 0.25);
      \draw[very thick] (-3.25, 1.75) -- (-2.75, 1.75);
      \draw[very thick] (3.25, 0.25) -- (2.75, 0.25);
      \draw[very thick] (3.25, 1.75) -- (2.75, 1.75);
      \draw[very thick] (-0.75, 0.5) -- (0.75, 0.5);
      \draw[very thick] (-0.75, 1.75) -- (0.75, 1.75);
      \node at (-3.05, 1.25) {$\vdots$};
      \node at (0, 1.3) {$\vdots$};
      \node at (3.05, 1.25) {$\vdots$};
      \draw [very thick] (-0.5, -0.5) arc (-90:90:0.35) -- (-0.75, 0.2);
      \draw [very thick] (0.5, -0.5) arc (270:90:0.35) -- (0.75, 0.2);
    \end{tikzpicture}
    }}.
  \end{equation}
Two features of this construction generalise.
Firstly, clasps are inductively defined in terms of smaller clasps. This is sometimes known as the ``triple-clasp formula'' which arises from plethysm rules and can be more readily seen in the following equivalent factorisation of \cref{eq:jones-wenzl-def}.
\begin{equation}\label{eq:jones-wenzl-def2}
    \vcenter{\hbox{
    \begin{tikzpicture}[scale=0.5]
      \draw (0,0) rectangle (2.5,2.5);
      \node at (1.25, 1.25) {$\JW_{n+1}$};
      \draw[very thick] (-0.5, 0.2) -- (0, 0.2);
      \draw[very thick] (-0.5, 2.3) -- (0, 2.3);
      \node at (-0.25, 1.5) {$\vdots$};
      \draw[very thick] (2.5, 0.2) -- (3, 0.2);
      \draw[very thick] (2.5, 2.3) -- (3, 2.3);
      \node at (2.75, 1.5) {$\vdots$};
    \end{tikzpicture}
    }}
    =
    \vcenter{\hbox{
    \begin{tikzpicture}[scale=0.5]
      \draw (-3.75,0) rectangle (-1.75,2);
      \node at (-2.75, 1) {$\JW_n$};
      \draw (1.75,0) rectangle (3.75,2);
      \node at (2.75, 1) {$\JW_n$};
      \draw (-1,0) rectangle (1,2);
      \node at (0,1) {$\JW_{n}$};
      \draw[very thick] (-4.25, -0.5) -- (4.25, -0.5);

      \draw[very thick] (-4.25, 0.25) -- (-3.75, 0.25);
      \draw[very thick] (-4.25, 1.75) -- (-3.75, 1.75);
      \draw[very thick] (4.25, 0.25) -- (3.75, 0.25);
      \draw[very thick] (4.25, 1.75) -- (3.75, 1.75);
      \draw[very thick] (-1.75, 0.25) -- (-1, 0.25);
      \draw[very thick] (1.75, 0.25) -- (1, 0.25);
      \draw[very thick] (1, 1.75) -- (1.75, 1.75);
      \draw[very thick] (-1.75, 1.75) -- (-1, 1.75);
      \node at (-4.05, 1.25) {$\vdots$};
      \node at (-1.4, 1.25) {$\vdots$};
      \node at (1.4, 1.25) {$\vdots$};
      \node at (4.05, 1.25) {$\vdots$};
    \end{tikzpicture}
    }}
    - \frac{[n]}{[n+1]}
    \vcenter{\hbox{
    \begin{tikzpicture}[scale=0.5]
      \draw (-3.75,0) rectangle (-1.75,2);
      \node at (-2.75, 1) {$\JW_n$};
      \draw (1.75,0) rectangle (3.75,2);
      \node at (2.75, 1) {$\JW_n$};
      \draw (-1,0.35) rectangle (1,1.9);
      \node[scale=0.8] at (0,1.125) {$\small\JW_{n-1}$};
      \draw[very thick] (-4.25, -0.5) -- (-1.5, -0.5);
      \draw[very thick] (4.25, -0.5) -- (1.5, -0.5);
      \draw[very thick] (-4.25, 0.25) -- (-3.75, 0.25);
      \draw[very thick] (-4.25, 1.75) -- (-3.75, 1.75);
      \draw[very thick] (4.25, 0.25) -- (3.75, 0.25);
      \draw[very thick] (4.25, 1.75) -- (3.75, 1.75);
      \draw[very thick] (-1.75, 0.5) -- (-1, 0.5);
      \draw[very thick] (1.75, 0.5) -- (1, 0.5);
      \draw[very thick] (1, 1.75) -- (1.75, 1.75);
      \draw[very thick] (-1.75, 1.75) -- (-1, 1.75);
      \node at (-4.05, 1.25) {$\vdots$};
      \node at (-1.4, 1.3) {$\vdots$};
      \node at (1.4, 1.3) {$\vdots$};
      \node at (4.05, 1.25) {$\vdots$};
      \draw [very thick] (-1.5, -0.5) arc (-90:90:0.35) -- (-1.75, 0.2);
      \draw [very thick] (1.5, -0.5) arc (270:90:0.35) -- (1.75, 0.2);
    \end{tikzpicture}
    }}
    .
  \end{equation}
%SM: TO CLARIFY:  is it the first diagram in (1.1) or both?  referee thinks the first
%RS: CLARIFICATION: Its both and that's the point. I've added another equation that shows the third clasp in both terms.
The second feature that generalises is that, the factor $[n]/[n+1]$ comes from a tracing rule:
\begin{equation}
  \vcenter{\hbox{
    \begin{tikzpicture}[scale=0.5]
      \draw[very thick] (-3.25, 0.5) -- (-0.25, 0.5);
      \draw[very thick] (-3.25, 2.05) -- (-0.25, 2.05);
      \node at (-3.05, 1.5) {$\vdots$};
      \node at (-0.4, 1.5) {$\vdots$};
      \draw [very thick] (-3., -0.5) arc (270:90:0.35) -- (-0.5, 0.2);
      \draw [very thick] (-3., -0.5) -- (-0.5, -0.5) arc (-90:90:0.35);
      
      \draw[fill=white] (-2.75,0) rectangle (-0.75,2.3);
      \node at (-1.75, 1.15) {$\JW_n$};
    \end{tikzpicture}
  }}
  =
  \frac{[n+1]}{[n]}
      \vcenter{\hbox{
    \begin{tikzpicture}[scale=0.5]
      \node at (-0.25, 1.2) {$\vdots$};
      \draw[very thick] (-.5, 0.2) -- (3, 0.2);
      \draw[very thick] (-.5, 1.7) -- (3, 1.7);
      \node at (2.75, 1.2) {$\vdots$};
      \draw[fill=white] (0,0) rectangle (2.5,1.9);
      \node at (1.25, 0.95) {$\JW_{n-1}$};
    \end{tikzpicture}
    }}.
\end{equation}
This tracing rule for $\JW_n$ can be computed inductively using \cref{eq:jones-wenzl-def} which in turn uses the tracing rule for $\JW_{n-1}$ and so on.
Thus determining the inductive form of the Jones-Wenzl clasps is computationally equivalent to computing the traces of these clasps.
In the more general case, this trace is called an intersection form and calculating their numerical values can be difficult.
Though these have been calculated for types $A_n$ for $n \le 4$, $C_2$ and $G_2$, in general we must revert to a conjecture due to Elias for these constants.

\vspace{1em}

\noindent
In this paper we draw out the connection between these intersection forms and the Gram determinants of cell modules in the category.
In particular we give expressions for all the cell Gram determinants in terms of these intersection forms.
We then go on to show that Elias' conjecture holds in type $A$, by finding pre-images of two distinguished bases across the quantum skew Howe duality

While this paper concerns type $A$ webs, it should be read with a more general aspect.
Some results, such as \cref{lem:same_cells} hold in more general ``object adapted cellular categories''\cite{spencer_phd}.
These include web categories for other Coxeter types, and diagrammatic Soergel bimodule categories.
If such a category has a monoidal structure which suitably interfaces with the category structure, then further constructions (such as those of the ladder basis and triple-clasps) should generalise.
However, a discussion of monoidal, cellular categories is beyond the scope of this paper.

\vspace{1em}

\noindent
The remainder of the paper is laid out as follows.
In \cref{sec:webs} we review the definition the web category and its intersection forms, and introduce Elias' conjecture for such  forms.
Then, in \cref{sec:cells} we examine the cell modules of this category and recount the ladder basis.
We then derive new bases, which give recurrent solutions to the determinant of the cell module and also determine an orthogonal basis.
Finally, \cref{sec:howe} shows that a previously-known basis of $U_q(\Gl_m)$ is a pre-image of this orthogonal basis.
Using this we are able to confirm Elias' conjectured formula for the intersection forms.

\section{The Category of Type $A$ Webs}\label{sec:webs}
We recall the definition of type $A$ webs.
This work was initiated by Kuperberg~\cite{kuperberg_1996} who derived the diagrammatic form for all rank two Lie type, and then completed by Cautis, Kamnitzer and Morrison who gave diagrammatics for general type $A_n$~\cite{cautis_kamnitzer_morrison_2014}.

The aim is to study the representation theory of $U_q(\Sl_{n})$ over $\Q(q)$.
To do this, it will be sufficient to study the category of (tensor products of) \emph{fundamental representations} of $U_q(\Sl_n)$, denoted $\text{\textbf{Fund}}_n$.
This category encodes the full category of finite dimensional representations through its Karoubi envelope.
The reader is directed to~\cite[\S 3]{elias_2015} for a full description of the appropriate idempotents, which are termed \emph{clasps}.

\begin{remark}
  The reader may wish to overlook the quantum deformation on a first reading.  The appropriate mental notational substitutions are
  \begin{equation}
    U_q(\Sl_n) \mapsto \Sl_n, \quad\quad\quad q \mapsto 1, \quad\quad\quad [n]\mapsto n.
  \end{equation}
  The results that hold for undeformed $\Sl_n$ often also hold for $U_q(\Sl_n)$ under the above notational substitution (see for example the definition of a Jones-Wenzl idempotent using quantum numbers).  However, the derivation of this is not always straightforward.
  For example, the representation theory of $\Sl_n$ occurs as a \textit{limit} of the theory of $U_q(\Sl_n)$ as $q\to 1$.
  In this paper, we make a point of keeping ourselves in the quantum case (i.e. $U_q(\Sl_n)$, $q$ and $[n]$) but the reader may have more intuition if they employ the above mental substitutions.
  %SM: TO DO:  make precise this last statement.
  %RS: How about the new formulation? I'm not sure how to be more explicit.
\end{remark}

Recall that $U_q(\Sl_n)$ has $n-1$ fundamental representations, namely $\{V_i := \bigwedge^i \C^n\}_{i = 1}^{n-1}$.
Here $\C^n$ is given the standard action of $U_q(\Sl_n)$.
In this way we can also consider $V_0 \cong V_n \cong \C$.
If $i + j = k$ then there are natural $U_q(\Sl_n)$-maps $V_i \otimes V_j \to V_k$ and $V_k \to V_i \otimes V_j$.  Note that these maps are not inverse.

The objects in the category $\text{\textbf{Fund}}_n$ are of the form $V_{i_1} \otimes V_{i_2} \otimes \cdots \otimes V_{i_t}$ and the morphisms are generated (through composition and tensor product) by the aforementioned maps between $V_i \otimes V_j$ and $V_k$.
If $\underline{i} = \underline{i_1i_2\cdots i_t}$ is a sequence of elements of $\{1,\ldots,n\}$ we will write $V_{\underline{i}}$ for $V_{i_1} \otimes V_{i_2} \otimes \cdots \otimes V_{i_t}$.
%SM: k changed to t in this paragraph

\begin{definition}\label{def:iota}
    There is an important anti-involution on the category, denoted $\iota$.
    This sends the aforementioned generating maps $V_{\underline{ij}} \to V_k$ and $V_k \to V_{\underline{ij}}$ to each other.
\end{definition}

\begin{remark}\label{remark:no_tags}
  Strictly speaking, we should define the objects to also permit tensor factors of the duals of the $V_i$, but since $V_i^* \cong V_{n-i}$ what we have constructed is sufficient.
  In \cite{elias_2015} our category is denoted by $\text{\textbf{Fund}}^+_n$. In the language of the oriented spider category, all our orientations run from left to right.

  In a similar simplification, we have omitted \textit{tags}, preferring to view these as generating morphisms with $k = n$ as in \cite[\S 2.1]{cautis_kamnitzer_morrison_2014}.
\end{remark}

The composition of the map from $V_k$ to $V_{\underline{ij}}$ and the map from $V_{\underline{ij}}$ to $V_k$ again is a multiple of the identity map on $V_k$.
Indeed, let $\qn n = (q^{n} - q^{-n})/(q-q^{-1})$ be the $n$-th quantum number.
Then if $\qn n ! = \qn n \qn {n-1} \cdots \qn 1$ and $\gaussianquant n r = \qn n! / (\qn r ! \qn {n-r}!)$, the multiple is $\gaussianquant k j =\gaussianquant k i$\cite[eq. 2.4]{cautis_kamnitzer_morrison_2014}.

\label{sec:quantum}
We will often use the notation $\delta = [2]$ so that, for example, $[3] = \delta^2-1$ and $[4] = \delta^3-2\delta$ in order to make complex polynomials of quantum numbers more readable. As such, the quantum numbers actually are members of $\Z[\delta] = \Z[q+q^{-1}]$ and it is a remarkable fact that the quantum binomial coefficients are too.

\begin{remark}\label{rem:any_semisimple}
  Our analysis will actually hold over any pointed ring wherein $q$ is not a root of unity.
  In these characteristics, all the quantum numbers are invertible and so the necessary objects will be well defined.
  When $q$ is specialised to a root of unity, certain quantum numbers and quantum binomial coefficients will vanish, and the required clasps will not be well defined.
  In fact, specialisation of $q$ to a root of unity corresponds to the specialisation of $\delta$ to a root of some quantum number.
  
  See~\cite{burrull_libedinsky_sentinelli_2019,martin_spencer_2021} for the correct idempotents in the $A_1$ case over positive characteristic with $q=1$, and $q$ a nontrivial root of unity, respectively.
\end{remark}
% RS: Todo include a section on quantum numbers and $\delta$

To provide a convenient category in which to work, we introduce the category of \emph{webs} or \emph{spiders}.
This category denoted $\text{\textbf{Sp}}_n$ is diagrammatic and is equivalent as a category to $\text{\textbf{Fund}}_n$ and then by isomorphism to $\text{\textbf{Fund}}^+_n$.
Thus we can think of these interchangeably.

The objects are taken from the set of finite words in letters $\{1, \ldots, n-1\}$ and the object $\underline{i}$ in $\text{\textbf{Sp}}_n$ corresponds to the object $V_{\underline i}$ in $\text{\textbf{Fund}}_n^+$.
As such, the tensor product on words is simply concatenation.

Morphisms are linear combinations of \emph{diagrams}\footnote{Technically we are working with the \emph{ladder} incarnation of diagrams from \cite{cautis_kamnitzer_morrison_2014}, however no generality is lost.}.
A diagram is an oriented, decorated, planar graph with edges labelled by the letters $\{1, \ldots, n-1\}$. 
These graphs are taken up to isotopy.
Each non-trivial vertex is trivalent and if, when read left-to-right, one edge splits into two, the sum of the outgoing labels equal that of the incoming.
A similar result holds for the reverse.

For example, the two classes of generating morphisms are represented as below.

\begin{equation}
\vcenter{\hbox{
      \begin{tikzpicture}[scale=2]
        \draw[very thick] (0.2, 0) -- (0.5, 0);
        \draw[very thick] (0.75, 0.25) -- (1.0, 0.25);
        \draw[very thick] (0.75, -0.25) -- (1.0, -0.25);
        \draw[very thick] (0.5, 0) -- (0.75, -0.25);
        \draw[very thick] (0.5, 0) -- (0.75, 0.25);

        \node at (0, 0) {$i+j$};
        
        \node at (1.2, 0.25) {$i$};
        \node at (1.2, -.25) {$j$};
        
      \end{tikzpicture}
    }}\quad\quad\quad\quad\quad\quad
    \vcenter{\hbox{
      \begin{tikzpicture}[scale=2]
        \draw[very thick] (-0.2, 0) -- (-0.5, 0);
        \draw[very thick] (-0.75, 0.25) -- (-1.0, 0.25);
        \draw[very thick] (-0.75, -0.25) -- (-1.0, -0.25);
        \draw[very thick] (-0.5, 0) -- (-0.75, -0.25);
        \draw[very thick] (-0.5, 0) -- (-0.75, 0.25);

        \node at (0, 0) {$i+j$};
        
        \node at (-1.2, -0.25) {$i$};
        \node at (-1.2, 0.25) {$j$};
        
      \end{tikzpicture}
    }}
\end{equation}

We adopt the convention of removing all edges that are labelled with $n$ or $0$.
Indeed, since tensoring with $\C$ is a null operation, we lose no information on the object level.
On the morphism level, vertices with an $n$ or $0$ are turned into bivalent vertices (simply edges).
%REFEREE IDENTIFIES A PROBLEM OMITTING $n$.  
\begin{example}
 When $n = 2$ so we are examining the representation theory of $U_q(\Sl_2)$, and the well known category to consider is the Temperley-Lieb category where circles resolve to $-[2]$. One can define a functor from the Temperley-Lieb category to $\text{\textbf{Fund}}^+_2$, an isomorphism onto its essential image which is the full subcategory with object set $\{1\}^{*} \simeq \{\bullet\}^* \simeq N$.

 To explain this functor, we temporarily undo the convention of \cref{remark:no_tags}. The generating object $V_2$ of $\text{\textbf{Fund}}^+_2$ is isomorphic to the base field and can be replaced by the monoidal identity, $\mathbb 1$. When doing this, the merge map $V_{\underline{11}} \to V_{\underline 2}$ and the split map $V_{\underline 2} \to V_{\underline{11}}$ can be interpreted (using tags) as morphisms $V_{\underline{11}} \leftrightarrow \mathbb 1$, as in \cite[\S 2.1]{cautis_kamnitzer_morrison_2014}. The functor sends the Temperley-Lieb ``cap'' to the merge map $V_{\underline{11}} \to \mathbb 1$ and the ``cup'' to minus the split map $\mathbb 1 \to V_{\underline{11}}$
\end{example}
%  Example 2.1  needs editing as per the referee.
\begin{example}
  When considering $U_q(\Sl_3)$ there are two objects which are dual to each other.  We will often consider the set $\{\texttt{+}, \texttt{-}\}^*$ instead of $\{1,2\}^*$.
  Instead of labeling the edges with $\{1,2\}$ we will orient them with the understanding that trivalent vertices are all sources or sinks.
  
  An example diagram from \underline{\texttt{+++}} on the right to \underline{\texttt{++---+}} on the left might be
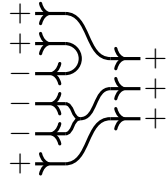
\begin{equation}
    \vcenter{\hbox{
  \begin{tikzpicture}[scale=0.4]
\begin{scope}[very thick,decoration={
    markings,
    mark=at position 0.5 with {\arrow{>}}}
    ] 
    \draw[very thick, postaction=decorate] (0.5, 4) -- (1.5,4);
    \draw[very thick, postaction=decorate] (1.5, 3) -- (0.5, 3);
    \draw[very thick] (1.5,4) arc (90:-90:0.5);
    
    \draw[very thick, postaction=decorate] (0.5, 5) -- (1.5,5) edge[out=0,in=180] (3,3.5);
    \draw[very thick, postaction=decorate] (3,3.5) -- (4,3.5);
    \draw[very thick, postaction=decorate] (0.5, 0) -- (1.5,0) edge[out=0,in=180] (3,1.5);
    \draw[very thick, postaction=decorate] (3,1.5) -- (4,1.5);
    
    \draw[very thick] (2,1.5)edge[out=180,in=0] (1.5,1);
    \draw[very thick] (2,1.5)edge[out=180,in=0] (1.5,2);
    \draw[very thick, postaction=decorate] (1.5,1) --(0.5, 1);
    \draw[very thick, postaction=decorate] (1.5,2) --(0.5, 2);

    \draw[very thick] (2,1.5) edge[out=0,in=180] (3,2.5);
    \draw[very thick, postaction=decorate] (3,2.5) -- (4,2.5);
    
  \end{scope}  
    \node at (0,0) {\texttt+};
    \node at (0,1) {\texttt-};
    \node at (0,2) {\texttt-};
    \node at (0,3) {\texttt-};
    \node at (0,4) {\texttt+};
    \node at (0,5) {\texttt+};
    
    \node at (4.5,1.5) {\texttt+};
    \node at (4.5,2.5) {\texttt+};
    \node at (4.5,3.5) {\texttt+};
  \end{tikzpicture}
}}.
\end{equation}
\end{example}
%  Referee does not like various conventions on diagrams in Ex 2.2....
There are relations on the morphisms in this category, 
%as suggested by the comment after \cref{rem:any_semisimple}, 
but their exact forms are not important for our study.
They can be found in~\cite{cautis_kamnitzer_morrison_2014}.

\subsection{Weights}\label{subsec:weights}
Each fundamental representation $V_i$ has highest weight a fundamental dominant weight $\varpi_i$.
Viewed as a partition, $\varpi_i = (1^i, 0^{n-i})$.
For any object $\underline x$ in $\textbf{Fund}$, let $\wt(\underline x)$ be $\sum_i \varpi_{x_i}$.
If this is viewed as a $n+1$-part partition $(\lambda_1, \cdots, \lambda_n, 0)$ of $\sum_i x_i$ then $\underline x$ contains $\lambda_{i} - \lambda_{i+1}$ copies of $i$.
The set of all weights, denoted $\Lambda$, is equipped with the dominance partial order, $\prec$.
We will abuse nation slightly and write $\underline{x} \prec \underline{y}$ for sequences $\underline{x}$ and $\underline{y}$ if $\wt(\underline x) \prec \wt(\underline y)$.

\subsection{Compatible Families}
Let $\lambda \in \Lambda$ and suppose $\varphi^\lambda =\{\varphi_{\underline{x},\underline{y}} : \underline{x} \to \underline{y}\}$ is a set of morphisms, where $\underline{x}$ and $\underline{y}$ range over all objects of weight $\lambda$.
% referee wants all or some to decorate the \lambda.  Is "given" better?
Let $J_{< \lambda}$ be the ideal spanned by morphisms that factor through objects of weight less than $\lambda$.
We say that $\varphi^\lambda$ is a \emph{compatible family} if
\begin{itemize}
    \item for each pair of objects $\underline x$ and $\underline y$ of weight $\lambda$, there is a unique $\varphi_{\underline{x},\underline{y}} : \underline{x} \to \underline{y}$ in $\varphi^\lambda$,
    \item $\varphi_{\underline{x},\underline{y}} \circ \varphi_{\underline{z},\underline{x}} = \varphi_{\underline{z},\underline{y}}$ modulo $J_{< \lambda}$ for any three composable elements of $\varphi^\lambda$, and
    \item $\varphi_{\underline{x},\underline{x}} = \id_{\underline{x}}$ modulo $J_{< \lambda}$.
\end{itemize}
The two examples of compatible families we will meet are \clasps{} and neutral ladders.

\begin{lemma}\label{lem:neutrals}
There is a compatible family of morphisms such that $\varphi_{\underline{x},\underline{x}} = \id_{\underline{x}}$ exactly.
\end{lemma}
These morphisms are known as (families of) \emph{neutral ladders}.
\begin{proof}
  We present the construction here and signpost the reader to the proof that it suffices.

  Set $\varphi_{\underline{x},\underline{x}} = \id_{\underline{x}}$.
  Now, suppose $\underline x$ and $\underline y$ differ in only two adjacent values, which is to say
  \begin{equation}
  \underline{x} = \underline{x_{i_1}\cdots x_{i_{t-1}}x_{i_t}x_{i_{t+1}}x_{i_{t+2}}\cdots x_{i_r}}
  \end{equation}
  and
  \begin{equation}
  \underline{y} = \underline{x_{i_1}\cdots x_{i_{t-1}}x_{i_{t+1}}x_{i_t}x_{i_{t+2}}\cdots x_{i_r}}
  \end{equation}
  Firstly, assume that  $x_{i_{t+1}} > x_{i_t}$.   
  %(noting that if the reverse holds we simply apply $\iota$ to the following construction).
  Then we construct the morphism $\varphi_{\underline{x},\underline{y}}$, which we term a simple morphism, as follows: 
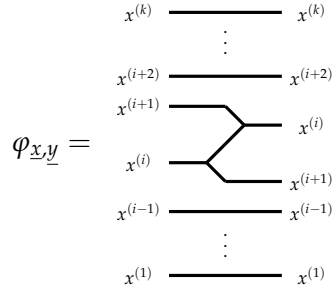
  \begin{equation}
  \varphi_{\underline{x},\underline{y}} =
  \vcenter{\hbox{
      \begin{tikzpicture}[scale=1]
        \draw[very thick] (0, 0) -- (0.5, 0);
        \draw[very thick] (1.0, 0.5) -- (0.5, 0);
        \draw[very thick] (1.0, 0.5) -- (1.5, 0.5);
        \draw[very thick] (0.5, 0) -- (0.75, -0.25);
        \draw[very thick] (1.5, -.25) -- (0.75,-0.25);
        \draw[very thick] (1.0, 0.5) -- (0.75, 0.75);
        \draw[very thick] (0.75, 0.75) -- (0, 0.75);
        
        \node at (-.4, 2) {\small $x_{i_r}$};
        \node at (-.4, 1.15) {\small $x_{i_{t+2}}$};
        \node at (-.4, 0.75) {\small $x_{i_t}$};
        \node at (-.4, 0) {\small $x_{i_{t+1}}$};
        \node at (-.4, -.65) {\small $x_{i_{t-1}}$};
        \node at (-.4, -1.5) {\small $x_{i_1}$};
        
        \node at (1.9, 2) {\small $x_{i_r}$};
        \node at (1.9, 1.15) {\small $x_{i_{t+2}}$};
        \node at (1.9, 0.5) {\small $x_{i_{t+1}}$};
        \node at (1.9, -.25) {\small $x_{i_t}$};
        \node at (1.9, -.65) {\small $x_{i_{t-1}}$};
        \node at (1.9, -1.5) {\small $x_{i_1}$};
        
        \draw[very thick] (0, -.65) -- (1.5, -.65);
        \draw[very thick] (0, 1.15) -- (1.5, 1.15);
        \draw[very thick] (0, 2.0) -- (1.5, 2.0);
        \draw[very thick] (0, -1.5) -- (1.5, -1.5);
        \node at (0.75, -1) {\tiny$\vdots$};
        \node at (0.75, 1.7) {\tiny$\vdots$};
      \end{tikzpicture}
    }}
  \end{equation}
  where the middle rung is labelled $x_{i_{t+1}} - x_{i_t}$.
  In case $x_{i_t}<x_{i_{t+1}}$, we set $\varphi_{\underline{x},\underline{y}}=\iota(\varphi_{\underline{y},\underline{x}})$, where $\iota $ is the anti-involution on the category as in \cref{def:iota}. %RS: I've now mentioned this earlier
  
  Then construct the family by composing these maps as appropriate. 
  These are known as {\itshape neutral ladders}~\cite[\S2.3]{elias_2016}. While there might be many ways to compose simple morphisms to obtain a morphism between two given objects, these are all equal modulo the ideal $J_{<\lambda}$ by ~\cite[Corollary 2.24]{elias_2016}.
\end{proof}
% RS Just removed it: makes more confusion than it removes
%\begin{remark}   Elias calls simple morphisms \emph{neutral rungs}.  
%  The objects of weight $\lambda\vdash k$ are in bijection with the elements of $\mathfrak{S}_k / \mathfrak{S}^\lambda$ where $\mathfrak{S}^\lambda$ is the Young subgroup of shape $\lambda$.
%  Under this bijection, the simple morphisms correspond to the (images of) simple reflections in $\mathfrak{S}_k / \mathfrak{S}^\lambda$, and by picking reduced words we obtain that the neutral ladders can be given by {\itshape Young diagrams}.  % perhaps add the word "Young"  ?    See further comments from the referee asking for clarity.
%\end{remark}

% referee wants "grey boxes with a 0" moved to juts aftre (3.3).  Maybe keep these examples here though.
%When used in an equation, we will denote neutral ladders by grey boxes with a 0.
%They are predominantly used to move between objects of the same weight in our category.

\begin{example}
  When $n = 2$, the weights and objects can both be identified with $\mathbb{N}$.  The dominance ordering is generated by $n \succ n-2$ and so odd numbers cannot be compared to even numbers.
  Since there are unique objects of a given weight, all neutral ladders are identities.
\end{example}
\begin{example}
  
  When $n = 3$, objects are simply elements of $\{\texttt{+}, \texttt{-}\}^*$ and neutral ladders are
  compositions of transpositions:
  \begin{equation} s_i =
    \vcenter
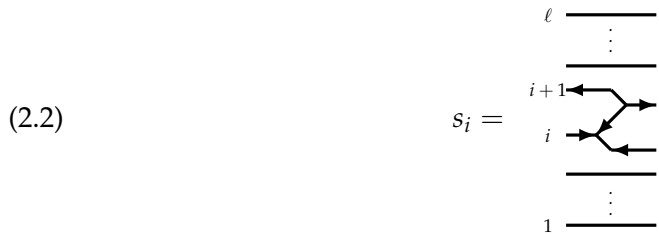
{\hbox{
      \begin{tikzpicture}[scale=0.8,decoration = {markings,
              mark = at position 1 with {\arrow{latex}}
                }]
        \draw[very thick, postaction=decorate] (0, 0) -- (0.5, 0);
        \draw[very thick, postaction=decorate] (1.0, 0.5) -- (0.5, 0);
        \draw[very thick, postaction=decorate] (1.0, 0.5) -- (1.5, 0.5);
        \draw[very thick] (0.5, 0) -- (0.75, -0.25);
        \draw[very thick, postaction=decorate] (1.5, -.25) -- (0.75,-0.25);
        \draw[very thick] (1.0, 0.5) -- (0.75, 0.75);
        \draw[very thick, postaction=decorate] (0.75, 0.75) -- (0, 0.75);
        
        \node at (-.3, 0) {\tiny $i$};
        \node at (-.3, 0.75) {\tiny $i+1$};
        \node at (-.3, 2) {\tiny $\ell$};
        \node at (-.3, -1.5) {\tiny $1$};
        
        \draw[very thick] (0, -.65) -- (1.5, -.65);
        \draw[very thick] (0, 1.15) -- (1.5, 1.15);
        \draw[very thick] (0, 2.0) -- (1.5, 2.0);
        \draw[very thick] (0, -1.5) -- (1.5, -1.5);
        \node at (0.75, -1) {\tiny$\vdots$};
        \node at (0.75, 1.7) {\tiny$\vdots$};
      \end{tikzpicture}
    }}
  \end{equation}
  The relations on the category ensure that the braid relations are held by the $s_i$, up to terms factoring through objects of lower weight.
  %nets need defining
\end{example}

The second critical family is that of \emph{clasps}. For the rest of the paper, fix a family of neutral ladders $\varphi^\lambda$ for each weight $\lambda$.

\begin{theorem}\label{thm:jw_exists_unique}
  If $\underline{x}$ and $\underline{y}$ are both of weight $\lambda$, then there is a unique non-zero morphism $\JW_{\underline{x},\underline{y}}:\underline{x} \to \underline{y}$ which is killed by all left (or right) compositions with morphisms that factor through objects of weight less than $\lambda$.
  Further, $\varphi_{\underline{x},\underline{y}}=\JW_{\underline{x},\underline{y}}$ modulo $J_{<\lambda}$.
\end{theorem}
%next sentence change "them"  
When the objects $\underline{x}$ and $\underline{y}$ are clear or do not matter, we will simply refer to the clasps as $\JW_\lambda$.
Note that $\JW_{\underline x, \underline x}$ is an idempotent for each $\underline x$.
\begin{example}
  When $n =2$ these are the celebrated Jones-Wenzl idempotents.  %give a reference for them
\end{example}

\begin{lemma}\label{lem:murderous_jw}
  If $f : \underline{x} \to \underline{y}$, then
  \begin{equation}
      \JW_{\underline{y}, \underline {y'}} \circ f \circ \JW_{\underline{x'}, \underline{x}}
      = \begin{cases}
      a_f\JW_{\underline{x'}, \underline{y'}}\underline&\wt(\underline x) = \wt(\underline y)\\
      0&\text{otherwise}
      \end{cases}
  \end{equation}
  for any $\underline{y'}$ of the same weight as $\underline{y}$, and $\underline{x'}$ of the same weight as $\underline{x}$, where $f = a_f \varphi_{\underline{x},\underline{y}}$ modulo $J_{<\lambda}$.
\end{lemma}

% SM: I agree with referee that this is false (but I think it's fixable)
% RS: Yeah, not sure what the point of this was. Remove?
% RS: Aha! It is used in the first part of the Gram Determinants chapte (but was mis referenced). Fixing it to work there by adding another clasp.
\begin{corollary} \label{cor:murderous_jw}
  Any morphism factoring through both a clasp of weight $\mu$ and a clasp of weight $\lambda$ vanishes unless $\mu = \lambda$, in which case it is a multiple of a clasp.
\end{corollary}

%   needs rewriting to satisfy the referee.  Rename elem light ladder as a leaf (with a specific morphism); singlestep light ladder for the map; light ladder for the whole lot.  
%Probably would benefit from an example.

Finally, we will introduce tiers and elementary light ladders.
Recall that if $\varpi_\ell$ is a fundamental weight, then there are $\binom n \ell$ weights appearing in the fundamental representation $V_\ell$.
To each of these weights $\mu$, we associate an \emph{elementary light ladder} $E_\mu$, as will be shown explicitly in \cref{eq:e_mu}. This morphism, when tensored with the identity and pre- or post-composed by neutral ladders, gives morphisms from objects of weight $\lambda + \mu$ to $\lambda + \varpi_i$ called \emph{tiers}.
If $\lambda+\mu$ is not dominant, the tier associated to $\mu$ is not defined for that $\lambda$.

The precise definition appears in~\cite[\S 2.5]{elias_2015} and we will give a brief overview in the next section. Note that our definition of tiers and light ladders are flipped right-to-left from Elias'. That is, what we call $E_\mu$ is called $\iota E_\mu$ in \cite{elias_2015} and vice-versa.

\begin{remark}
  Though we do not need to deal explicitly with the elementary light ladders in this paper, we will note later that they are the images of the monomial basis under the skew quantum Howe duality at play.
\end{remark}
\subsection{Partitions and Computation of Intersection forms}
The computation of the clasps for $A_n$ webs hinges on the calculation of so-called intersection forms.
\begin{definition}\label{def:inter_form}
  Let $\lambda \in \Lambda$ and suppose that $a$ is a fundamental dominant weight.  Let $\Omega(a)$ be the set of weights appearing in the representation $V_a$. Then the intersection form $\kappa_\lambda^{\lambda+ \mu}$ for $\mu \in \Omega(a)$ is the coefficient of the identity diagram in
  \begin{equation}
    \vcenter{\hbox{
    \begin{tikzpicture}
      \draw[very thick, fill=purple!30!white] (-.5,-1) rectangle (.5,1);
      \draw[very thick] (.5,-1.4) rectangle (1.2,0.1);
      \draw[very thick] (-.5,-1.4) rectangle (-1.2,0.1);
      \draw[line width=2pt] (-.5, -1.2) -- (0.5, -1.2);
      
      \draw[line width=2pt] (1.2, -1.2) -- (1.4, -1.2);
      \draw[line width=2pt] (1.2, -1) -- (1.4, -1);
      \draw[line width=2pt] (1.2, -0.8) -- (1.4, -0.8);
      \draw[line width=2pt] (1.2, -0.6) -- (1.4, -0.6);
      \draw[line width=2pt] (1.2, -0.4) -- (1.4, -0.4);
      \draw[line width=2pt] (1.2, -0.2) -- (1.4, -0.2);
      \draw[line width=2pt] (1.2, -0) -- (1.4, -0);
      \draw[line width=2pt] (0.5, 0.2) -- (1.4, 0.2);
      \draw[line width=2pt] (0.5, 0.4) -- (1.4, 0.4);
      \draw[line width=2pt] (0.5, 0.6) -- (1.4, 0.6);
      \draw[line width=2pt] (0.5, 0.8) -- (1.4, 0.8);

      \draw[line width=2pt] (-1.2, -1.2) -- (-1.4, -1.2);
      \draw[line width=2pt] (-1.2, -1) -- (-1.4, -1);
      \draw[line width=2pt] (-1.2, -0.8) -- (-1.4, -0.8);
      \draw[line width=2pt] (-1.2, -0.6) -- (-1.4, -0.6);
      \draw[line width=2pt] (-1.2, -0.4) -- (-1.4, -0.4);
      \draw[line width=2pt] (-1.2, -0.2) -- (-1.4, -0.2);
      \draw[line width=2pt] (-1.2, -0) -- (-1.4, -0);
      \draw[line width=2pt] (-0.5, 0.2) -- (-1.4, 0.2);
      \draw[line width=2pt] (-0.5, 0.4) -- (-1.4, 0.4);
      \draw[line width=2pt] (-0.5, 0.6) -- (-1.4, 0.6);
      \draw[line width=2pt] (-0.5, 0.8) -- (-1.4, 0.8);

      \node at (0, 0) {$\lambda$};
      \node at (0.85, -.65) {$E_\mu$};
      \node at (-0.85, -.65) {$\iota E_\mu$};
      \node at (0, -1.4) {\small$a$};
    \end{tikzpicture}}}.
  \end{equation}
  Note this is a morphism from an object of weight $\lambda + \mu$ to itself and so the intersection form is only defined when $\lambda + \mu$ is dominant.
\end{definition}
In~\cite{elias_2015} these constants are denoted by $\kappa_{\lambda,\mu}$.
\begin{remark}
Note that the intersection form $\kappa_\lambda^{\lambda+\mu}$ requires us to know the clasp on $\lambda$ to calculate. However, in~\cite{elias_2015}, the clasp on $\lambda$ is given in terms of intersection forms, much like the form found for generalised Jones-Wenzl elements in~\cite{martin_spencer_2021}.

This is not circular logic, however, but part of the inductive definition of the clasp elements. Indeed, to define the clasp on $\lambda$ we need only determine  all intersection forms $\kappa_{\lambda-\mu}^\lambda$.
\end{remark}

There is a conjectured form of these intersection forms in \cite[Conjecture 3.16]{elias_2015}.  This conjecture is known to hold for $n \le 4$.
In \cref{sec:howe} we will prove this conjecture for all $n$.
\begin{conjecture}[Elias]
\label{conj:elias}
Let $\Phi^+$ and $\Phi^-$ be the sets of positive and negative roots respectively, $\rho$ the half sum of the positive roots, and let $\lambda$ be a dominant weight. Suppose that $a$ is a fundamental dominant weight and $\mu \in \Omega(a)$. Then if $\lambda + \mu$ is also dominant,
\begin{equation}
    \kappa_{\lambda}^{\lambda+\mu} = \prod_{\alpha \in \Phi(\mu)} \frac{[\langle\lambda+\rho, \alpha\rangle]}{[\langle\lambda+\rho, \alpha\rangle-1]}
\end{equation}
where $\Phi(\mu) = \{\alpha \in \Phi^+ \;:\; w_\mu^{-1}(\alpha) \in \Phi^-\}$ and $w_\mu$ is the minimal element of $S_n$ taking $\mu$ to a dominant weight.
\end{conjecture}

% SM: Note comment on \underline notation of x and i
% RS: Changed $i$ to $j$ to help
To compute these conjectured intersection forms it will help to introduce some notation.
Let $\varpi_j$ denote the $j$-th fundamental weight so that $V_j$ has highest weight $\varpi_j$.
We identify $\varpi_j$ with the $n$-part partition (composition) of $n$ given by $(1^j0^{n-j})$.
For brevity in equations, we may denote this similarly to the below, for example.
\begin{equation}
    \varpi_4 = \colwei{(1,1,1,1,0,0,0)}
\end{equation}
Note that the highest weight of $V_{\underline{i}}$ is $\varpi_{i_1} + \cdots \varpi_{i_s}$, which is a partition of $\sum_s i_s$ into $n$ parts.

In this notation and basis, the roots of type $A_n$ are of the form $\alpha_{j,k} = -\varpi_{j-1}+\varpi_{j}+\varpi_{k-1} - \varpi_{k}$ for distinct $j$ and $k$ or, more succinctly,
\begin{equation}
    \alpha_{j, k} = \begin{cases}
    \vspace{0.2em}\colroot{(0,1,0,0,-1,0,0)}&j < k\\
    \colroot{(0,-1,0,0,1,0,0)}&  j > k
    \end{cases}
\end{equation}
where there is a positive box in the $j$-th row and a negative in the $k$-th.
We use red to visually distinguish roots from weights.
The inner product is the usual inner product on $n$-part tuples.

Now, given a weight $\mu$ (a composition in $n$ parts) appearing in some fundamental representation $V_\ell$, observe that it consists of a sequence of $1$s and $0$s. It is not hard to see then that $\Phi(\mu)$ contains all $\alpha_{j,k}$ where $j<k$ and $(\mu_j, \mu_k) = (0,1)$.
%perhaps we do want a notation for the set of such weights?

\begin{example}
We calculate some of the sets $\Phi$ for strategically chosen weights to demonstrate this calculation. Here, $n = 5$ so all compositions are in five parts.
\begin{center}
\begin{minipage}{0.30\textwidth}
\begin{align*}
\Phi\left(\colwei{(1,0,0,0,0)}\right) &= \{\}\\
\Phi\left(\colwei{(0,1,0,0,0)}\right) &= \left\{\colroot{(1,-1,0,0,0)}\right\}\\
\Phi\left(\colwei{(0,0,1,0,0)}\right) &= \left\{\colroot{(0,1,-1,0,0)},\colroot{(1,0,-1,0,0)}\right\}\\
\Phi\left(\colwei{(0,0,0,1,0)}\right) &= \left\{\colroot{(0,0,1,-1,0)},\colroot{(0,1,0,-1,0)},\colroot{(1,0,0,-1,0)}\right\}
\end{align*}
\end{minipage}
\begin{minipage}{0.30\textwidth}
\begin{align*}
\Phi\left(\colwei{(1,1,1,0,0)}\right) &= \{\}\\
\Phi\left(\colwei{(1,1,0,1,0)}\right) &= \left\{\colroot{(0,0,1,-1,0)}\right\}\\
\Phi\left(\colwei{(1,0,1,1,0)}\right) &= \left\{\colroot{(0,1,-1,0,0)}, \colroot{(0,1,0,-1,0)}\right\}
\end{align*}
\end{minipage}
\begin{minipage}{0.30\textwidth}
\begin{align*}
\Phi\left(\colwei{(1,1,0,0,0)}\right) &= \{\}\\
\Phi\left(\colwei{(1,0,0,1,0)}\right) &= \left\{\colroot{(0,0,1,-1,0)}, \colroot{(0,1,0,-1,0)}\right\}\\
\Phi\left(\colwei{(0,1,1,0,0)}\right) &= \left\{\colroot{(1,-1,0,0,0)}, \colroot{(1,0,-1,0,0)}\right\}\\
\Phi\left(\colwei{(0,1,0,1,0)}\right) &= \left\{\colroot{(1,-1,0,0,0)}, \colroot{(1,0,0,-1,0)}, \colroot{(0,0,1,-1,0)}\right\}
\end{align*}
\end{minipage}
\end{center}
\end{example}

It is immediate that if $\mu$ is dominant (i.e. a partition of $\ell$ into at most $n$ parts) the set is empty, so 
that $\mu=w_\ell$, and $E_{\mu}$ is the identity map.  Consequently one expects
$\kappa_{\lambda}^{\lambda + \mu} = 1$.
This corresponds to adding a single strand labelled $\mu$ in the construction of the clasps and so it is expected that the intersection form is trivial.

With this notation of weights $\mu$ as sequences of 1s and 0s, we can take a brief aside to define the maps $E_\mu$ explicitly as per \cite[Eq. 2.17]{elias_2015}. They are of the form
\begin{equation}\label{eq:e_mu}
  \vcenter{\hbox{\begin{tikzpicture}[scale=0.5]
    \draw[very thick] (0, 4) -- (3, 4);
    \draw[very thick] (0, 3) -- (3, 3);
    \draw[very thick] (0, 2) -- (3, 2);
    \draw[very thick] (0, 1) -- (3, 1);
    \draw[very thick] (0, 0) -- (3, 0);
    \draw[very thick] (1, 4) -- (2, 3);
    \draw[very thick] (1, 3) -- (2, 2);
    \draw[very thick] (1, 2) -- (2, 1);
    \draw[very thick] (1, 1) -- (2, 0);
    \node at (-.6, 4) {$x_1$};
    \node at (-.6, 3) {$x_2$};
    \node at (-.6, 2.2) {$\vdots$};
    \node at (-.6, 1) {$x_k$};
    \node at (-.6, 0) {$\ell$};

    \node at (3.6, 4) {$y_1$};
    \node at (3.6, 3) {$y_2$};
    \node at (3.6, 2.2) {$\vdots$};
    \node at (3.6, 1) {$y_k$};
    \node at (3.7, 0) {$y_{k+1}$};
  \end{tikzpicture}}}
\end{equation}
where $0 \le y_1 < x_1<y_2<x_2<\cdots<x_k<y_{k+1}\le n$ are derived as follows. Split the sequence of 1s and 0s into maximal consecutive subsequences of only 1s and only 0s. Then $y_i$ is the last index in the $i$-th $1$-string (where $y_1$ is zero if the whole sequence starts with a 0) and $x_i$ is the last index in the $i$-th $0$-string. Note that $k$ is the number of $0$-strings excluding the final one, and there are $k+1$ horizontal strands.

Next let us give some combinatorial flavour to the inner products appearing in \cref{conj:elias}.  We have that $\langle \lambda + \rho, \alpha_{i,j}\rangle = \lambda_i - \lambda_j + \rho_i - \rho_j = \lambda_i - \lambda_j -i + j$. This value is known by another name:
\begin{definition}\label{def:axial}
  The \emph{axial distance} between $j$ and $k$ in partition $\lambda$ is
  \begin{equation}
      c_{jk}^\lambda = \lambda_j - \lambda_k + k - j.
  \end{equation}
\end{definition}

These axial distances occur frequently in the literature of the representation theory of $\mathrm{GL}_n$ and related algebras. They can be interpreted as the differences between \emph{contents} in the last boxes of row $j$ and $k$, where the contents of a box with coordinates $(x,y)$ is $x-y$.
Axial distances are also intimately related to the ratio of (quantum) hook lengths\cite{kosuda_1990}. Note that, for example, $c_{jk}^\lambda + c_{k\ell}^\lambda = c_{j\ell}^\lambda$ and that $c_{jk}^\lambda > 0$ if $k > j$.

\begin{example}\label{eg:all_kappa}
We continue our computation of strategically chosen constants and present some conjectured intersection forms.

The easiest to exhibit are those with empty $\Phi$ sets:
\begin{equation*}
  \kappa_{\,\tinyyng{(1,1,1)}}^{\,\tinyyng{(2,2,1)}}\;=
  \;
  \kappa_{\,\tinyyng{(2,1,1,1)}}^{\,\tinyyng{(3,2,2,1)}}\;=
  \;
  \kappa_{\,\tinyyng{(2,2,1)}}^{\,\tinyyng{(3,3,2)}}\;=
  \;
  \kappa_{\,\tinyyng{(3,3,2,1)}}^{\,\tinyyng{(4,4,2,1)}}\;=
  \;
  \kappa_{\,\tinyyng{(3,3,1,1)}}^{\,\tinyyng{(4,3,1,1)}}\;=
  \;
  \kappa_{\,\tinyyng{(3,3,2)}}^{\,\tinyyng{(4,3,2)}}\;=\;1.
\end{equation*}
In each of the above, the boxes added to the bottom partition to form the top partition (i.e. $\mu$) are themselves a partition (i.e. $\mu$ is dominant).

The rest is simply a matter of working through the calculations.
\begin{align*}
  \kappa_{\tinyyng{(1,1,1)}}^{\tinyyng{(2,1,1,1)}} =
  \kappa_{\tinyyng{(1,1,1)}}^{\colwei{(1,1,1,0,0)}+\colwei{(1,0,0,1,0)}}
  &=
  \egfrac{3}{4}{(1,1,1,0,0)}
  \egfrac{2}{4}{(1,1,1,0,0)}
  = \frac{[2]}{[1]}\frac{[3]}{[2]} = [3]\\
  \kappa_{\tinyyng{(2,2,1)}}^{\tinyyng{(3,3,1,1)}}
  =\kappa_{\tinyyng{(2,2,1)}}^{\colwei{(2,2,1,0,0)}+\colwei{(1,1,0,1,0)}}
  &=
  \egfrac{3}{4}{(2,2,1,0,0)} = [2]
\end{align*}
Note that these intersection forms evaluate to a quantum integer, or more loosely, a polynomial in $[2]$ instead of a rational function. However, this is not always the case:
\begin{align*}
\kappa_{\tinyyng{(2,2,1)}}^{\tinyyng{(3,2,2,1)}}&=
\egfrac{2}{3}{(2,2,1,0,0)}
\egfrac{2}{4}{(2,2,1,0,0)}
= \frac{[2]}{[1]}\frac{[4]}{[3]} = \frac{[2][4]}{[3]}.
\end{align*}
We finish this example with five further intersection forms that will be useful in later calculations.

\begin{align*}
\kappa_{\tinyyng{(4,3,1,1)}}^{\tinyyng{(4,4,2,1)}}&=
\egfrac{1}{2}{(4,3,1,1,0)}
\egfrac{1}{3}{(4,3,1,1,0)}
=\frac{[2][5]}{[4]}
\\
\kappa_{\tinyyng{(4,3,2)}}^{\tinyyng{(4,4,2,1)}}&=
\egfrac{1}{2}{(4,3,2,0,0)}
\egfrac{1}{4}{(4,3,2,0,0)}
\egfrac{3}{4}{(4,3,2,0,0)}=
\frac{[3][7]}{[6]}\\
\kappa_{\tinyyng{(3,2,2,1)}}^{\tinyyng{(3,3,2,1)}}&=
\egfrac{1}{2}{(3,2,2,1,0)}
=[2]
\\
\kappa_{\tinyyng{(3,3,1,1)}}^{\tinyyng{(3,3,2,1)}}&=
\egfrac{2}{3}{(3,3,1,1,0)}
\egfrac{1}{3}{(3,3,1,1,0)}
 = \frac{[4]}{[2]}
\\
\kappa_{\tinyyng{(3,3,2)}}^{\tinyyng{(3,3,2,1)}}&=
\egfrac{3}{4}{(3,3,2,0,0)}
\egfrac{2}{4}{(3,3,2,0,0)}
\egfrac{1}{4}{(3,3,2,0,0)}
= \frac{[3][6]}{[2][4]}
\end{align*}
\end{example}

We can take inspiration from the above to rephrase \cref{conj:elias} in terms of partitions:
\begin{conjecture}[Elias]
\label{conj:elias2}
Let $\lambda'\subset \lambda$ be partitions (dominant weights), such that $\lambda\setminus\lambda'$ has no parts larger than 1.  Then
\begin{equation}
    \kappa_{\lambda'}^{\lambda} = \prod_{\substack{j < k\\\lambda'_j = \lambda_j\\\lambda'_k = \lambda_k+1}} \frac{[c^\lambda_{jk}]}{[c^\lambda_{jk} - 1]}.
\end{equation}
\end{conjecture}

\section{Cell Modules}\label{sec:cells}
We will now show how to use semistandard tableaux to enumerate bases of cell modules. For general references on this method see \cite{littelmann_1995} and the appendix of \cite{green_1980}.

%need to define \succ
Let $\underline{x}$ and $\underline{y}$ be sequences and suppose (using the notation of \cref{subsec:weights}) that $\underline x \succ \underline y$. Then the \emph{cell module $S(\underline x, \underline y)$} is the vector space of all morphisms $\underline y \to \underline x$ modulo those that factor through objects of weight less than $\underline y$.

\begin{lemma}\label{lem:same_cells}
  If $\wt(\underline{y_1}) = \wt(\underline{y_2})$ then $S(\underline x, \underline{y_1}) \simeq S(\underline x, \underline{y_2})$ as $\text{\bf Fund}$-modules.
\end{lemma}
\begin{proof}
  By \cref{lem:neutrals} there is a complete set of neutral ladders between all objects of weight $\wt(\underline{y_1})$.
  Due to the fact that these are a compatible family and so form inverses to eachother (up to $J_{<\wt(\underline y_1)}$ they form isomorphisms between the cell modules as vector spaces.
  Further, as right-composition on a module of morphisms on which $\text{\bf Fund}$ acts on the left, these are actually $\text{\bf Fund}$-morphisms.
\end{proof}

% SM: why the lambdas are now bold?  And need a definition of \iota    
% RS: lambdas disemboldened
For this reason, we may refer to $S(\underline{x}, \lambda)$ where $\lambda$ is a weight.
Here, one can pick whichever object of weight $\lambda$ one prefers as the source for the morphisms in $S(\underline{x}, \lambda)$.

We recall that $\text{\bf Fund}$ is equipped with an anti-involution $\iota$.
In particular, this endows the cell modules with canonical forms $\langle - , -\rangle$ given by (for any pair $m_1,m_2 \in S(\underline x, \underline y)$),
\begin{align}\label{eq:usual_form}
     (\iota m_1)\circ m_2 &\equiv \langle m_1, m_2 \rangle \id_{\underline{y}} \\
\intertext{modulo $J_{<\wt(\underline y)}$. In the language of clasps, we can write this as}
    \JW \circ (\iota m_1)\circ m_2 \circ \JW &= \langle m_1, m_2 \rangle \JW.
\end{align}

\begin{remark}
  This form is the crux of the cellular category and underpins many of the structure constants that appear.
  For example, by comparing with \cref{def:inter_form}, we see that the intersection form $\kappa_\lambda^{\lambda+\mu}$ is simply the inner product
  \begin{equation}
  \Big\langle (\JW_\lambda\otimes\id_a) \circ (\id\otimes E_\mu),\;(\JW_\lambda\otimes\id_a) \circ (\id\otimes E_\mu) \Big\rangle
  \end{equation}
  in $S(\underline{xa}, \lambda + \mu)$ if $\underline{x}$ is the object we have chosen for $\JW_\lambda$.
  % SM: next sentence needs altering 
  % RS: Made some adjustments (completely rewrote it)
  In fact, we could consider the \emph{intersection} form as the \emph{cellular} form restricted to the (in this case one-dimensional) space spanned by morphisms of the form $(\JW_\lambda\otimes\id_{a})\circ (\id\otimes E_{\mu})$.
\end{remark}

% RS: We need to revisit this remark in light of Charles' work.
\begin{remark}
  When one has an integral lattice, as we do with the diagram basis we are about to develop, one can consider reductions mod $p$.
  In particular we construct $\mathfrak p = (p)$ and consider the modules $M / \mathfrak p M$.
  
  If there is a contravariant bilinear form with suitable properties, one achieves a \emph{Jantzen filtration}. 
  This is a filtration of the modular cell module into semisimple layers, each occuring as a reduction of the characteristic zero module $\mathfrak p^nM / \mathfrak p^{n+1}M$\cite[chapter 8]{jantzen_1987}.

  %\delta needs to be defined as [2].  
  In our case, however, we will want to reduce modulo $p$ and some minimal polynomial for $\delta$, say $m_\delta$.
  This corresponds to the ordinary and quantum characteristics respectively.
  We thus have a maximal ideal $\mathfrak p = (p, m_\delta)$ in $\Z[\delta]$.
  
  However, we now no longer have a chain of nested ideals, but rather a set of ideals of the form $(p^i, m_\delta^j)$.
  This forms a lattice of ideals under inclusion.
  The layers of our filtration are now indexed by pairs of ideals (one covering the other) in this lattice.

  Studying these layers in mixed characteristic, even for simple examples such as the Temperley-Lieb algebras, promises to be a rich line of future inquiry.
  %I wonder if worth a reference to you,me and Charles' work?
\end{remark}

\subsection{A basis}\label{subsec:basis}
We recall the construction of a basis for the cell module set out in~\cite{elias_2015}.
This basis is built inductively on the length of the source object.

Suppose we wish to build a basis for cell module $S(\underline{x}, \underline{y})$ and we can write $\underline x = \underline{\hat{x} x_0}$ where $x_0 \in \{1,\ldots,n-1\}$.
We build this out of basis elements for $S(\underline{\hat x}, \underline {z})$ for various $\underline z$.

To be exact, let $\mu \in \Omega(x_0)$ and pick an object $\underline{z_\mu}$ of weight $\wt(\underline y) - \mu$.
Suppose that $L_{\mathbf{t}}$ is a basis element for $S(\underline{\hat{x}}, \underline{z_\mu})$ where the index $\mathbf t$ is a sequence of weights.
We construct basis element $L_{\mathbf{t}a}$ for $S(\underline{x}, \underline{y})$ diagrammatically as follows:
\begin{equation}\label{eq:inductive_diagram}
\vcenter{\hbox{
  \begin{tikzpicture}
    \draw[thick] (0,0) -- (0, 2) -- (1,1.5) -- (1,0)-- cycle;
    \draw[very thick] (-0.2, -.2) -- (1.5,-.2); \node at (0.75,-.5) {\tiny$a$};
    \draw[very thick] (-0.2, 0.1) -- (0,0.1);
    \draw[very thick] (-0.2, 0.4) -- (0,0.4);
    \draw[very thick] (-0.2, 0.7) -- (0,0.7);
    \draw[very thick] (-0.2, 1.0) -- (0,1.0);
    \draw[very thick] (-0.2, 1.3) -- (0,1.3);
    \draw[very thick] (-0.2, 1.6) -- (0,1.6);
    \draw[very thick] (-0.2, 1.9) -- (0,1.9);
    
    \node at (0.5, 0.75) {$L_{\mathbf t}$};
    
    \draw[very thick] (1.5, 1.35) -- (2.1,1.35);
    \draw[very thick] (1.5, 1.1) -- (2.1,1.1);
    \draw[very thick] (1.5, 0.85) -- (2.1,0.85);

    \draw[thick, fill=black!20!white] (2.1,-.4) rectangle(2.6, 1.5);
    \node at (2.35, 0.6) {$0$};
    
    \draw[thick, fill=black!20!white] (1.0,0) rectangle(1.5, 1.5);
    \node at (1.25, 0.75) {$0$};
    
    \draw[thick] (1.5,-0.4) rectangle(2.1, 0.6);
    \node at (1.8, 0.1) {$E_\mu$};
    
    \draw[very thick] (2.6, -.2) -- (2.8, -.2);
    \draw[very thick] (2.6, 0.1) -- (2.8, 0.1);
    \draw[very thick] (2.6, 0.4) -- (2.8, 0.4);
    \draw[very thick] (2.6, 0.7) -- (2.8, 0.7);
    \draw[very thick] (2.6, 1.0) -- (2.8, 1.0);
    \draw[very thick] (2.6, 1.3) -- (2.8, 1.3);
  \end{tikzpicture}
}}
\end{equation}
Here, grey boxes with zeros denote neutral ladders that do not change the weight. The exact neutral ladders used are not important, but must ensure that the output of $L_{\mathbf t}$, i.e. $\underline{z_\mu}$ is post-composable by $E_\mu$ in the manner shown above, and also to obtain the correct object as the target.
Then the set of morphisms $L_{\mathbf t a}$, over all the $\mu$ and $L_{\mathbf{t}}$ for each $\mu$, form a basis for $S(\underline{x}, \underline{y})$.

%see referee's comments to expand this example:  ...
\begin{example}\label{eg:tl_basis_construct}
 If $n = 2$ we are in the Temperley-Lieb case; the only option for $a$ is $\bullet$.
 The two maps $E_\mu$ corresponding to the weights in $V_1$ are
 \begin{equation*}
 E_+ = 
  \vcenter{\hbox{
    \begin{tikzpicture}[scale=0.6]
      \draw[very thick] (-.2,0) -- (0,0) edge[out=0,in=0] (0,0.5);
      \draw[very thick] (0,0.5) -- (-.2,0.5);
    \end{tikzpicture}
  }}\quad\quad\quad\quad E_- = 
  \vcenter{\hbox{
    \begin{tikzpicture}[scale=0.6]
      \draw[very thick] (0,0) edge[out=0,in=180] (1,0);
    \end{tikzpicture}
  }}.
 \end{equation*}
 There are no neutral ladders.
 The basis inductively described then is 
 \begin{equation*}
 \left\{
    \vcenter
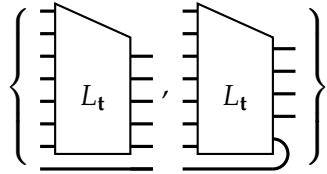
{\hbox{
      \begin{tikzpicture}
        \draw[thick] (0,0) -- (0, 2) -- (1,1.55) -- (1,0)-- cycle;
        \draw[very thick] (-0.2, -.2) -- (1.3,-.2);

        \draw[very thick] (-0.2, 0.1) -- (0,0.1);
        \draw[very thick] (-0.2, 0.4) -- (0,0.4);
        \draw[very thick] (-0.2, 0.7) -- (0,0.7);
        \draw[very thick] (-0.2, 1.0) -- (0,1.0);
        \draw[very thick] (-0.2, 1.3) -- (0,1.3);
        \draw[very thick] (-0.2, 1.6) -- (0,1.6);
        \draw[very thick] (-0.2, 1.9) -- (0,1.9);
        
        \node at (0.5, 0.75) {${L}_{\mathbf{t}}$};
        
        \draw[very thick] (1.0, -.2) -- (1.3, -.2);
        \draw[very thick] (1.0, 0.1) -- (1.3, 0.1);
        \draw[very thick] (1.0, 0.4) -- (1.3, 0.4);
        \draw[very thick] (1.0, 0.7) -- (1.3, 0.7);
        \draw[very thick] (1.0, 1.0) -- (1.3, 1.0);
        \draw[very thick] (1.0, 1.3) -- (1.3, 1.3);
      \end{tikzpicture}
    }}
    ,
    \vcenter{\hbox{
      \begin{tikzpicture}
        \draw[thick] (0,0) -- (0, 2) -- (1,1.55) -- (1,0)-- cycle;
        \draw[very thick] (-0.2, -.2) -- (1,-.2);
        \draw[very thick] (1,-.2) arc (-90:90:0.2);
        
        \draw[very thick] (-0.2, 0.1) -- (0,0.1);
        \draw[very thick] (-0.2, 0.4) -- (0,0.4);
        \draw[very thick] (-0.2, 0.7) -- (0,0.7);
        \draw[very thick] (-0.2, 1.0) -- (0,1.0);
        \draw[very thick] (-0.2, 1.3) -- (0,1.3);
        \draw[very thick] (-0.2, 1.6) -- (0,1.6);
        \draw[very thick] (-0.2, 1.9) -- (0,1.9);
        
        \node at (0.5, 0.75) {${L}_{\mathbf{t}}$};
        
        \draw[very thick] (1.0, 0.5) -- (1.3, 0.5);
        \draw[very thick] (1.0, 0.8) -- (1.3, 0.8);
        \draw[very thick] (1.0, 1.1) -- (1.3, 1.1);
        \draw[very thick] (1.0, 1.4) -- (1.3, 1.4);
      \end{tikzpicture}
    }}
    \right\},
\end{equation*}
which is exactly the usual inductive construction for the diagram basis for the Temperley-Lieb cell modules.
\end{example}

Now for $\mu \in \Omega(a)$, the morphism $E_\mu$ increases the weight by $\mu - a$.
Thus in each step we increase the weight by $a$ (by adding a strand labelled $a$) and then again by $\mu - a$ through an application of $E_\mu$.

We thus obtain a labelling of the basis by \emph{miniscule Littelmann paths}, $E(\underline{x}, w(\underline y))$.
This is the set of sequences of weights $\underline{\mu_1\mu_2\cdots\mu_k}$ with each $\mu_i\in\Omega(x_i))$ such that the partial sum $\mu_1 + \cdots + \mu_t$ is always dominant and the total sum $\sum_i \mu_i = \wt(\underline y)$.

\begin{lemma}\label{lem:standardtableaux}
  The set $E(\underline{1^a}, \lambda)$ is in natural bijection with standard tableaux of shape $\lambda$.
\end{lemma}
\begin{proof}
  The weights of $V_{1^a}$ are of the form
  \begin{equation*}
      \Omega(1) = \left\{
      \colwei{(1,0,0,0,0)}\;,\;
      \colwei{(0,1,0,0,0)}\;,\;
      \colwei{(0,0,1,0,0)}\;,\;
      \colwei{(0,0,0,1,0)}\;,\;
      \colwei{(0,0,0,0,1)}
      \right\}
  \end{equation*}
  and so a step in a miniscule Littelmann path corresponds to adding a box to a partition without making it non-dominant (i.e. no longer a partition).
\end{proof}

\begin{corollary}
When $n = 1$, the dimension of $S(\underline{1^a}, \lambda)$ is
\begin{equation}\label{eq:temperley_lieb_dim}
  d_\lambda = \binom{a}{\lambda_2} - \binom{a}{\lambda_2 - 1}.
\end{equation}
When $n = 2$, The dimension of $S(\underline{1^a}, \lambda)$ is
\begin{equation}
  d_\lambda = \frac{a! (\lambda_1 - \lambda_2 + 1)(\lambda_1 - \lambda_3 + 2)(\lambda_2 - \lambda_3+1)}{(\lambda_1 + 2)!(\lambda_2 + 1)!\lambda_3!}.
\end{equation}
\end{corollary}
\begin{proof}
This is essentially a restatement of the familiar hook length-formula.
\end{proof}
Note that \cref{eq:temperley_lieb_dim} is the dimensions of the cell modules for the Temperley-Lieb algebra on $a$ strands, as expected.

\begin{example}
  In the same vein as \cref{eg:tl_basis_construct}, let us construct the basis for $S(\underline{+^a}, \lambda)$ for $n=2$.
  The three elements of $\Omega(+)$ correspond to
\begin{equation*}
 E_\colwei{(1,0,0)} = 
  \vcenter{\hbox{
    \begin{tikzpicture}[scale=0.6,decoration = {markings,
              mark = at position 0.5 with {\arrow{latex}}
                }]
      \draw[very thick] (-.5,0) -- (0,0);
      \draw[very thick,postaction=decorate] (0,0) -- (.5,0);
    \end{tikzpicture}
  }}
  \quad\quad\quad\quad E_\colwei{(0,1,0)} = 
  \vcenter{\hbox{
    \begin{tikzpicture}[scale=0.6,decoration = {markings,
              mark = at position 0.5 with {\arrow{latex}}
                }]
      \draw[very thick] (-.55,0.25) -- (-0.1,.25);
      \draw[very thick] (-.55,0.75) -- (-0.1,.75);
      \draw[very thick, postaction=decorate] (-.1,.25) -- (0,.25);
      \draw[very thick, postaction=decorate] (-.1,.75) -- (0,.75);
      \draw[very thick] (0,0.25) -- (0.3,0.5);
      \draw[very thick] (0,0.75) -- (0.3,0.5);
      \draw[very thick,postaction=decorate] (0.55,0.5) -- (0.3,0.5);
      \draw[very thick] (1,0.5) -- (0.55,0.5);
    \end{tikzpicture}
  }}
  \quad\quad\quad\quad E_\colwei{(0,0,1)} = 
  \vcenter{\hbox{
    \begin{tikzpicture}[scale=0.6,decoration = {markings,
              mark = at position 0.5 with {\arrow{latex}}
                }]
      \draw[very thick] (-.2,0) -- (0,0);
      \draw[very thick] (0,0) -- (0.3,0);
      \draw[very thick,postaction=decorate] (0.3,0) edge[out=0,in=0] (0.3,0.5);
      \draw[very thick,postaction=decorate] (0,0.5) -- (-.2,0.5);
      \draw[very thick] (0,0.5) -- (0.3,0.5);
    \end{tikzpicture}
  }}
  .
 \end{equation*}  
 Note that we would not be able to insert $E_{\colwei{(0,1,0)}}$ if $L_{\mathbf{t}}$ didn't have a rightward facing strand at its target, and we wouldn't be able to use  $E_{\colwei{(0,1,0)}}$  unless it had a leftward facing one. To circumvent this issue we can apply $E_{\colwei{(0,1,0)}}$  or $E_{\colwei{(0,1,0)}}$ by using a neutral ladder to bring a rightward, respectively leftward facing strand to the bottom, which is always possible as long as such a strand exists. This is exactly the statement that $\lambda + \mu$ is dominant.
 %Check I got this right above....
 
 Now, since we are looking for a basis of $S(\underline{+}^a, \lambda)$, we consider standard tableaux of shape $\lambda$.
 Given such a standard tableaux $s \in T(\lambda)$, we will denote its basis element by $L_s$.
  If the box with the largest number is removed from $s$, we obtain $t \in T(\lambda_i^-)$ for some $i$. Here $\lambda_i^-$ denotes the partition obtained from $\lambda$ by deleting the rightmost box on the $i$-th row.
  Now, depending on the value of $i$, we set $\mathbb{s}$ to be one of the three options below.
    \begin{equation}
    \substack{\vcenter
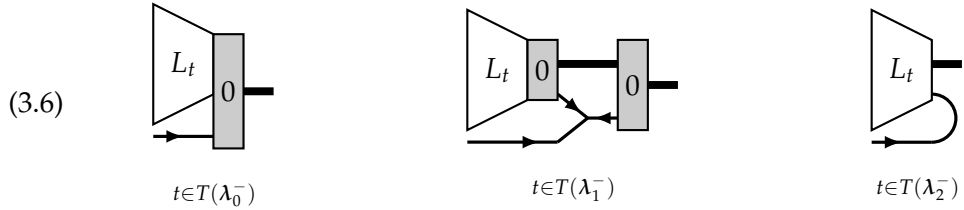
{\hbox{
      \begin{tikzpicture}[scale=0.8,decoration = {markings,
              mark = at position 0.5 with {\arrow{latex}}
                }]
        \draw[thick] (0,0) -- (0, 2) -- (1,1.5) -- (1,0.5)-- cycle;
        \draw[very thick,postaction=decorate] (0, -.2) -- (1,-.2);
        \node at (0.5, 1) {${L}_t$};
        \draw[thick, fill=black!20!white] (1,-.4) rectangle(1.5, 1.5);
        \node at (1.25, 0.55) {$0$};
        \draw[line width=3pt] (1.5,0.55) -- (2.0,0.55);
      \end{tikzpicture}
    }}\vspace{1em}\\t \in T(\lambda_0^-)}\quad\quad\quad\quad\quad\quad
    \substack{\vcenter{\hbox{
      \begin{tikzpicture}[scale=0.8,decoration = {markings,
              mark = at position 0.8 with {\arrow{latex}}
                }]
        \draw[thick] (0,0) -- (0, 2) -- (1,1.5) -- (1,0.5)-- cycle;
        \draw[very thick,postaction=decorate] (0, -.2) -- (1.5,-.2);
        \draw[very thick] (1.5,-.2) -- (2.0, 0.2);
        \draw[very thick,postaction=decorate] (1.5, 0.6) -- (2.0,0.2);
        \draw[very thick,postaction=decorate](2.5,0.2) -- (2.0, 0.2);
        \node at (0.5, 1) {${L}_t$};
        \draw[line width=3pt] (1.5,1.1) -- (2.5,1.1);
        \draw[thick, fill=black!20!white] (2.5,0) rectangle(3.0, 1.5);
        \node at (2.75, 0.75) {$0$};
        \draw[thick, fill=black!20!white] (1.0,0.5) rectangle(1.5, 1.5);
        \node at (1.25, 1.) {$0$};
        \draw[line width=3pt] (3.0,0.75) -- (3.5,0.75);
      \end{tikzpicture}
    }}\vspace{1em}\\t \in T(\lambda_1^-)}\quad\quad\quad\quad\quad\quad
    \substack{\vcenter{\hbox{
      \begin{tikzpicture}[scale=0.8,decoration = {markings,
              mark = at position 0.5 with {\arrow{latex}}
                }]
        \draw[thick] (0,0) -- (0, 2) -- (1,1.5) -- (1,0.5)-- cycle;ways
        \draw[very thick,postaction=decorate] (0, -.2) -- (1,-.2);
        \draw[very thick] (1,-.2) arc  (-90:90:0.4);
        \node at (0.5, 1) {${L}_t$};
        \draw[line width=3pt] (1.0,1.1) -- (1.5,1.1);
    %    \draw[thick, fill=black!20!white] (1.5,0.6) rectangle(2.0, 1.5);
    %    \node at (1.75, 1.0) {$0$};
      \end{tikzpicture}
    }}\vspace{1em}\\t \in T(\lambda_2^-)}
    \end{equation}
    Note that ${L}_t$ is a map from $(+^n,0)$ to some object of weight $w(\lambda_i^-)$.
    The grey box labelled 0 is a neutral map that simply rearranges the output suitably.
\end{example}

There is a natural extension of \cref{lem:standardtableaux}.
\begin{lemma}
  The set $E(\underline{x}, \lambda)$ is in natural bijection with row-semistandard tableaux of shape $\lambda$ and weight $\underline{x}$.
\end{lemma}
Here a row-semistandard tableaux is one in which the elements increase strictly along rows and weakly along columns.
A (column) semistandard tableaux is strictly increasing along columns and weakly along rows.
We will denote the set of all row-semistandard tableaux of shape $\lambda$ and weight $\underline x$ as $\operatorname{r-SSYT}(\lambda, \underline x)$, and the set of all semistandard tableaux of that shape and weight by $\operatorname{SSYT}(\lambda, \underline x)$.

A useful construction is the \emph{reduced Young poset} $\Y^n_{\underline x}$ \cite{ceccherini-silberstein_scarabotti_tolli_2010}\footnote{Our convention for row-semistandard tableaux results in the graph of transposed partitions to those in the literature.}.
This consists of all the partitions
\begin{equation}
    \bigcup_{i = 0}^{\operatorname{length} \lambda} \left\{
    \bmu \text{ a partition of at most $n$ parts} \;:\; \operatorname{r-SSYT}(\bmu, \underline{x_1\cdots x_i}) \neq \emptyset\right\}
\end{equation}
arranged in natural layers, with edges between a partition in layer $i$ and one in layer $i+1$ if $x_i$ boxes can be added to obtain it.

It is  clear an element of $E(\underline{x}, \lambda)$ or equivalently $\operatorname{r-SSYT}(\lambda, \underline{x})$ is equivalent to a path from the empty partition $\pmb{\varnothing}$ to $\lambda$ in $\Y_{\underline{x}}^n$. Further, in the case $\underline{x} = \underline{1^n}$, we recover exactly the Young graph $\Y_n$.

\begin{example}
We present the graph $\Y_{\underline{3231}}^5$.
  
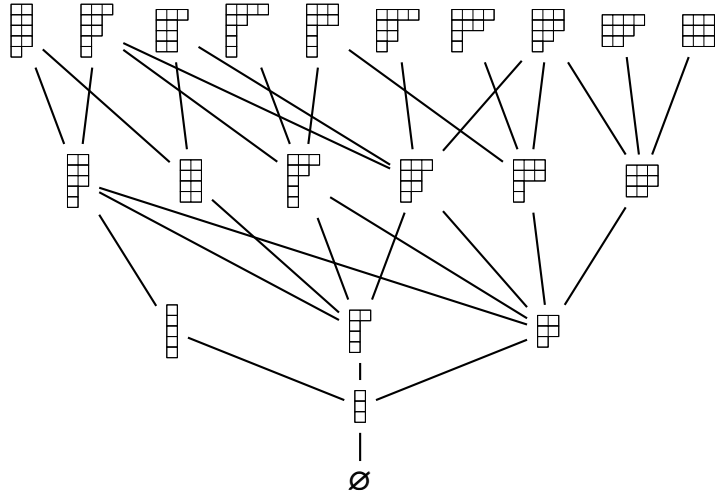
\begin{center}
    \begin{tikzpicture}[scale=0.5]
    
\node (d) at (0,0) {$\emptypart$};

\node (111) at (0,2) {$\smolyng{(1,1,1)}$};

\node (11111) at (-5,4) {$\smolyng{(1,1,1,1,1)}$};
\node (2111) at (0,4) {$\smolyng{(2,1,1,1)}$};
\node (221) at (5,4) {$\smolyng{(2,2,1)}$};

\node (22211) at (-7.5,8) {$\smolyng{(2,2,2,1,1)}$};
\node (2222) at (-4.5,8) {$\smolyng{(2,2,2,2)}$};
\node (32111) at (-1.5,8) {$\smolyng{(3,2,1,1,1)}$};
\node (3221) at (1.5,8) {$\smolyng{(3,2,2,1)}$};
\node (3311) at (4.5,8) {$\smolyng{(3,3,1,1)}$};
\node (332) at (7.5,8) {$\smolyng{(3,3,2)}$};

\node (22221) at (-9,12) {$\smolyng{(2,2,2,2,1)}$};
\node (32211) at (-7,12) {$\smolyng{(3,2,2,1,1)}$};
\node (42111) at (-3,12) {$\smolyng{(4,2,1,1,1)}$};
\node (3222) at (-5,12) {$\smolyng{(3,2,2,2)}$};
\node (33111) at (-1,12) {$\smolyng{(3,3,1,1,1)}$};
\node (4221) at (1,12) {$\smolyng{(4,2,2,1)}$};
\node (4311) at (3,12) {$\smolyng{(4,3,1,1)}$};
\node (3321) at (5,12) {$\smolyng{(3,3,2,1)}$};
\node (432) at (7,12) {$\smolyng{(4,3,2)}$};
\node (333) at (9,12) {$\smolyng{(3,3,3)}$};

\draw[thick] (d) -- (111);
\draw[thick] (111) -- (11111);
\draw[thick] (111) -- (2111);
\draw[thick] (111) -- (221);
\draw[thick] (11111) -- (22211);
\draw[thick] (2111) -- (3221);
\draw[thick] (2111) -- (22211);
\draw[thick] (2111) -- (32111);
\draw[thick] (2111) -- (2222);
\draw[thick] (221) -- (332);
\draw[thick] (221) -- (3221);
\draw[thick] (221) -- (22211);
\draw[thick] (221) -- (32111);
\draw[thick] (221) -- (3311);
\draw[thick] (332) -- (3321);
\draw[thick] (332) -- (432);
\draw[thick] (332) -- (333);
\draw[thick] (3221) -- (32211);
\draw[thick] (3221) -- (3321);
\draw[thick] (3221) -- (4221);
\draw[thick] (3221) -- (3222);
\draw[thick] (22211) -- (22221);
\draw[thick] (22211) -- (32211);
\draw[thick] (32111) -- (42111);
\draw[thick] (32111) -- (32211);
\draw[thick] (32111) -- (33111);
\draw[thick] (2222) -- (22221);
\draw[thick] (2222) -- (3222);
\draw[thick] (3311) -- (3321);
\draw[thick] (3311) -- (4311);
\draw[thick] (3311) -- (33111);
    \end{tikzpicture}
\end{center}
\end{example}

% refere wants a reference for the corespondence between minuscule Littelmann paths and SST and somewhere (near the start of Ch3 perhaps) say this is standard material. 

\section{Gram Determinants}\label{sec:determinants}
Our main result of this section concerns the Gram determinants of the 
%state or give reference to what form we're considering.  Also need to recall d_{\lambda} and what w(\lambda_i} is.  Referee asks why it is needed....and I certainly can't see it immediately useful. 
% RS: I think this was useful once upon a time when I thought that you needed to have the JW elements comparable for the inner product thing to work. That is too strong a condition, so I can remove it without
% explanation I think
form on $S(\underline x, \lambda)$ as defined in \cref{eq:usual_form}.

Our observation is a change of basis from the construction in \cref{eq:inductive_diagram}.
In this, we simply replace one of the neutral ladders with the appropriate clasp.
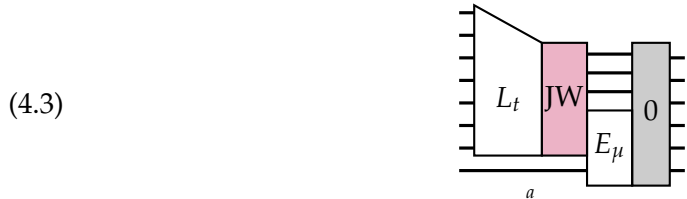
\begin{equation}\label{eq:inductive_diagram_2}
\vcenter{\hbox{
  \begin{tikzpicture}
    \draw[thick] (0,0) -- (0, 2) -- (0.9,1.5) -- (0.9,0)-- cycle;
    \draw[very thick] (-0.2, -.2) -- (1.5,-.2); \node at (0.75,-.5) {\tiny$a$};
    \draw[very thick] (-0.2, 0.1) -- (0,0.1);
    \draw[very thick] (-0.2, 0.4) -- (0,0.4);
    \draw[very thick] (-0.2, 0.7) -- (0,0.7);
    \draw[very thick] (-0.2, 1.0) -- (0,1.0);
    \draw[very thick] (-0.2, 1.3) -- (0,1.3);
    \draw[very thick] (-0.2, 1.6) -- (0,1.6);
    \draw[very thick] (-0.2, 1.9) -- (0,1.9);
    
    \node at (0.45, 0.75) {${L}_{\mathbf t}$};
    
    \draw[very thick] (1.5, 1.35) -- (2.1,1.35);
    \draw[very thick] (1.5, 1.1) -- (2.1,1.1);
    \draw[very thick] (1.5, 0.85) -- (2.1,0.85);

    \draw[thick, fill=black!20!white] (2.1,-.4) rectangle(2.6, 1.5);
    \node at (2.35, 0.6) {$0$};
    
    \draw[thick, fill=purple!30!white] (0.9,0) rectangle(1.5, 1.5);
    \node at (1.2, 0.75) {$\JW$};
    
    \draw[thick] (1.5,-0.4) rectangle(2.1, 0.6);
    \node at (1.8, 0.1) {$E_\mu$};
    
    \draw[very thick] (2.6, -.2) -- (2.8, -.2);
    \draw[very thick] (2.6, 0.1) -- (2.8, 0.1);
    \draw[very thick] (2.6, 0.4) -- (2.8, 0.4);
    \draw[very thick] (2.6, 0.7) -- (2.8, 0.7);
    \draw[very thick] (2.6, 1.0) -- (2.8, 1.0);
    \draw[very thick] (2.6, 1.3) -- (2.8, 1.3);
  \end{tikzpicture}
}}
\end{equation}
This represents  an upper unitriangular change of basis and so the Gram matrix for this basis has the same determinant as that for the original basis.

However, if we evaluate an inner product,
\begin{equation}
\left \langle
\vcenter{\hbox{
  \begin{tikzpicture}
    \draw[thick] (0,0) -- (0, 2) -- (0.9,1.5) -- (0.9,0)-- cycle;
    \draw[very thick] (-0.2, -.2) -- (1.5,-.2); \node at (0.75,-.5) {\tiny$a$};
    \draw[very thick] (-0.2, 0.1) -- (0,0.1);
    \draw[very thick] (-0.2, 0.4) -- (0,0.4);
    \draw[very thick] (-0.2, 0.7) -- (0,0.7);
    \draw[very thick] (-0.2, 1.0) -- (0,1.0);
    \draw[very thick] (-0.2, 1.3) -- (0,1.3);
    \draw[very thick] (-0.2, 1.6) -- (0,1.6);
    \draw[very thick] (-0.2, 1.9) -- (0,1.9);
    
    \node at (0.45, 0.75) {${L}_{\mathbf{t}}$};
    
    \draw[very thick] (1.5, 1.35) -- (2.1,1.35);
    \draw[very thick] (1.5, 1.1) -- (2.1,1.1);
    \draw[very thick] (1.5, 0.85) -- (2.1,0.85);

    \draw[thick, fill=black!20!white] (2.1,-.4) rectangle(2.6, 1.5);
    \node at (2.35, 0.6) {$0$};
    
    \draw[thick, fill=purple!30!white] (0.9,0) rectangle(1.5, 1.5);
    \node at (1.2, 0.75) {$\JW$};
    
    \draw[thick] (1.5,-0.4) rectangle(2.1, 0.6);
    \node at (1.8, 0.1) {$E_\mu$};
    
    \draw[very thick] (2.6, -.2) -- (2.8, -.2);
    \draw[very thick] (2.6, 0.1) -- (2.8, 0.1);
    \draw[very thick] (2.6, 0.4) -- (2.8, 0.4);
    \draw[very thick] (2.6, 0.7) -- (2.8, 0.7);
    \draw[very thick] (2.6, 1.0) -- (2.8, 1.0);
    \draw[very thick] (2.6, 1.3) -- (2.8, 1.3);
  \end{tikzpicture}
}},
\vcenter{\hbox{
  \begin{tikzpicture}
    \draw[thick] (0,0) -- (0, 2) -- (0.9,1.5) -- (0.9,0)-- cycle;
    \draw[very thick] (-0.2, -.2) -- (1.5,-.2); \node at (0.75,-.5) {\tiny$a$};
    \draw[very thick] (-0.2, 0.1) -- (0,0.1);
    \draw[very thick] (-0.2, 0.4) -- (0,0.4);
    \draw[very thick] (-0.2, 0.7) -- (0,0.7);
    \draw[very thick] (-0.2, 1.0) -- (0,1.0);
    \draw[very thick] (-0.2, 1.3) -- (0,1.3);
    \draw[very thick] (-0.2, 1.6) -- (0,1.6);
    \draw[very thick] (-0.2, 1.9) -- (0,1.9);
    
    \node at (0.45, 0.75) {${L}_{\mathbf{s}}$};
    
    \draw[very thick] (1.5, 1.35) -- (2.1,1.35);
    \draw[very thick] (1.5, 1.1) -- (2.1,1.1);
    \draw[very thick] (1.5, 0.85) -- (2.1,0.85);

    \draw[thick, fill=black!20!white] (2.1,-.4) rectangle(2.6, 1.5);
    \node at (2.35, 0.6) {$0$};
    
    \draw[thick, fill=purple!30!white] (0.9,0) rectangle(1.5, 1.5);
    \node at (1.2, 0.75) {$\JW$};
    
    \draw[thick] (1.5,-0.4) rectangle(2.1, 0.6);
    \node at (1.8, 0.1) {$E_\nu$};
    
    \draw[very thick] (2.6, -.2) -- (2.8, -.2);
    \draw[very thick] (2.6, 0.1) -- (2.8, 0.1);
    \draw[very thick] (2.6, 0.4) -- (2.8, 0.4);
    \draw[very thick] (2.6, 0.7) -- (2.8, 0.7);
    \draw[very thick] (2.6, 1.0) -- (2.8, 1.0);
    \draw[very thick] (2.6, 1.3) -- (2.8, 1.3);
  \end{tikzpicture}
}}
\right \rangle
\end{equation}
we must compute
\begin{equation}
\vcenter{\hbox{
  \begin{tikzpicture}
    \draw[thick] (0,0) -- (0, 2) -- (0.9,1.5) -- (0.9,0)-- cycle;
    \draw[very thick] (0, -.2) -- (1.5,-.2); \node at (0,-.5) {\tiny$a$};

    \node at (0.45, 0.75) {${L}_{\mathbf{t}}$};
    
    \draw[very thick] (1.5, 1.35) -- (2.1,1.35);
    \draw[very thick] (1.5, 1.1) -- (2.1,1.1);
    \draw[very thick] (1.5, 0.85) -- (2.1,0.85);

    \draw[thick, fill=black!20!white] (2.1,-.4) rectangle(2.6, 1.5);
    \node at (2.35, 0.6) {$0$};
    
    \draw[thick, fill=purple!30!white] (0.9,0) rectangle(1.5, 1.5);
    \node at (1.2, 0.75) {$\JW$};
    
    \draw[thick] (1.5,-0.4) rectangle(2.1, 0.6);
    \node at (1.8, 0.1) {$E_\mu$};
    
    \draw[very thick] (2.6, -.2) -- (2.8, -.2);
    \draw[very thick] (2.6, 0.1) -- (2.8, 0.1);
    \draw[very thick] (2.6, 0.4) -- (2.8, 0.4);
    \draw[very thick] (2.6, 0.7) -- (2.8, 0.7);
    \draw[very thick] (2.6, 1.0) -- (2.8, 1.0);
    \draw[very thick] (2.6, 1.3) -- (2.8, 1.3);

\begin{scope}[yscale=1,xscale=-1]
    \draw[thick] (0,0) -- (0, 2) -- (0.9,1.5) -- (0.9,0)-- cycle;
    \draw[very thick] (0, -.2) -- (1.5,-.2);

    \node at (0.45, 0.75) {${L}_{\mathbf{s}}$};
    
    \draw[very thick] (1.5, 1.35) -- (2.1,1.35);
    \draw[very thick] (1.5, 1.1) -- (2.1,1.1);
    \draw[very thick] (1.5, 0.85) -- (2.1,0.85);

    \draw[thick, fill=black!20!white] (2.1,-.4) rectangle(2.6, 1.5);
    \node at (2.35, 0.6) {$0$};
    
    \draw[thick, fill=purple!30!white] (0.9,0) rectangle(1.5, 1.5);
    \node at (1.2, 0.75) {$\JW$};
    
    \draw[thick] (1.5,-0.4) rectangle(2.1, 0.6);
    \node at (1.8, 0.1) {$\iota E_\nu$};
    
    \draw[very thick] (2.6, -.2) -- (2.8, -.2);
    \draw[very thick] (2.6, 0.1) -- (2.8, 0.1);
    \draw[very thick] (2.6, 0.4) -- (2.8, 0.4);
    \draw[very thick] (2.6, 0.7) -- (2.8, 0.7);
    \draw[very thick] (2.6, 1.0) -- (2.8, 1.0);
    \draw[very thick] (2.6, 1.3) -- (2.8, 1.3);
\end{scope}
  \end{tikzpicture}
}}
\end{equation}
% SM: referee makes the comment that L_s and L_t have different targets???
% RS: um, isn't that the point?
However, invoking \cref{cor:murderous_jw}, we see that this vanishes {\itshape identically} (that is not only up to some quotient) if ${L}_{\mathbf{s}}$ and ${L}_{\mathbf{t}}$ have different targets.
But if they have the same targets and this morphism is between two objects of the same weight, we must have $\mu = \nu$.
If they have the same target, of weight $\lambda' = \lambda - \mu$ then we are left with 
\begin{equation}
\left\langle {L}_{\mathbf s},{L}_{\mathbf t}\right\rangle \times 
\vcenter{\hbox{
  \begin{tikzpicture}
    \draw[very thick] (-0.3, -.2) -- (0.3,-.2); \node at (0,-.5) {\tiny$a$};

    \draw[very thick] (0.3, 1.35) -- (0.9,1.35);
    \draw[very thick] (0.3, 1.1) -- (0.9,1.1);
    \draw[very thick] (0.3, 0.85) -- (0.9,0.85);
    
    \draw[thick, fill=black!20!white] (0.9,-.4) rectangle(1.4, 1.5);
    \node at (1.15, 0.6) {$0$};
    
    \draw[thick, fill=purple!30!white] (-0.3,0) rectangle(0.3, 1.5);
    \node at (0, 0.75) {$\JW$};
    
    \draw[thick] (0.3,-0.4) rectangle(0.9, 0.6);
    \node at (0.6, 0.1) {$E_\mu$};
    
    \draw[very thick] (1.4, -.2) -- (1.6, -.2);
    \draw[very thick] (1.4, 0.1) -- (1.6, 0.1);
    \draw[very thick] (1.4, 0.4) -- (1.6, 0.4);
    \draw[very thick] (1.4, 0.7) -- (1.6, 0.7);
    \draw[very thick] (1.4, 1.0) -- (1.6, 1.0);
    \draw[very thick] (1.4, 1.3) -- (1.6, 1.3);

\begin{scope}[yscale=1,xscale=-1]
    \draw[very thick] (0.3, 1.35) -- (0.9,1.35);
    \draw[very thick] (0.3, 1.1) -- (0.9,1.1);
    \draw[very thick] (0.3, 0.85) -- (0.9,0.85);

    \draw[thick, fill=black!20!white] (0.9,-.4) rectangle(1.4, 1.5);
    \node at (1.15, 0.6) {$0$};

    \draw[thick] (0.3,-0.4) rectangle(0.9, 0.6);
    \node at (0.6, 0.1) {$\iota E_\mu$};
    
    \draw[very thick] (1.4, -.2) -- (1.6, -.2);
    \draw[very thick] (1.4, 0.1) -- (1.6, 0.1);
    \draw[very thick] (1.4, 0.4) -- (1.6, 0.4);
    \draw[very thick] (1.4, 0.7) -- (1.6, 0.7);
    \draw[very thick] (1.4, 1.0) -- (1.6, 1.0);
    \draw[very thick] (1.4, 1.3) -- (1.6, 1.3);
\end{scope}
  \end{tikzpicture}
}}=
\left\langle {L}_{\mathbf s},{L}_{\mathbf t}\right\rangle
\kappa_{\lambda'}^{\lambda} \id \mod J_{< \lambda}
\end{equation}
%RHS is scalar times identity map modulo lower terms stuff,,,
We have thus shown that in this basis, the Gram matrix is block diagonal, with one block for each weight $\lambda'$ one level lower down in the Young graph.
Each block is the Gram matrix for the corresponding cell module $S(\underline{\hat x}, \lambda')$ multiplied by the constant $\kappa_{\lambda'}^\lambda$. We have thus shown
\begin{equation}\label{eq:recursive-det}
\det G(\underline x,\lambda) = \prod_{
\substack{
    \lambda' \rightarrow \lambda\\
    \in \Y_{\underline{x}}^n}}
    (\kappa_{\lambda'}^\lambda)^{d(\underline x,\lambda')}
    \det G(\underline{\hat x},\lambda')
    .
\end{equation}
%  referee wants more prominence given to this construction and to call it something like "clasped light ladders".....
  \begin{example}
    For the Temperley-Lieb category, this change of basis and subsequent calculation is a standard method of computing the Gram determinant.  We have adopted the convention that the circle resolves to $+[2]$ for these examples, in order to make the equations cleaner. Recall however, that the webs for $U_q(\Sl_2)$ require the circle to resolve to $-[2]$. The substitution is straightforward.
    
    Consider the module $S(6,2)$ where usual basis is
    \begin{equation*}
    \substack{\vcenter{\hbox{
      \begin{tikzpicture}[scale=0.5]
        \draw[very thick] (0,0) edge[in=180,out=0] (1,1);
        \draw[very thick] (0,0.5) edge[in=180,out=0] (1,1.5);
        \draw[very thick] (0,1) edge[in=0,out=0] (0,1.5);
        \draw[very thick] (0,2) edge[in=0,out=0] (0,2.5);
      \end{tikzpicture}}}\vspace{.5em}\\
      {\Yboxdim{9pt}\footnotesize\young(1356,24)}
      }\;
      \substack{\vcenter{\hbox{
      \begin{tikzpicture}[scale=0.5]
        \draw[very thick] (0,0) edge[in=180,out=0] (1,1);
        \draw[very thick] (0,0.5) edge[in=180,out=0] (1,1.5);
        \draw[very thick] (0,1.5) edge[in=0,out=0] (0,2);
        \draw[very thick] (0,1) edge[in=0,out=0] (0,2.5);
      \end{tikzpicture}}}\vspace{.5em}\\
      {\Yboxdim{9pt}\footnotesize\young(1256,34)}
      }\;
      \substack{\vcenter{\hbox{
      \begin{tikzpicture}[scale=0.5]
        \draw[very thick] (0,0) edge[in=180,out=0] (1,1);
        \draw[very thick] (0,1.5) edge[in=180,out=0] (1,1.5);
        \draw[very thick] (0,.5) edge[in=0,out=0] (0,1);
        \draw[very thick] (0,2) edge[in=0,out=0] (0,2.5);
      \end{tikzpicture}}}\vspace{.5em}\\
      {\Yboxdim{9pt}\footnotesize\young(1346,25)}
      }\;
      \substack{\vcenter{\hbox{
      \begin{tikzpicture}[scale=0.5]
        \draw[very thick] (0,0) edge[in=180,out=0] (1,1);
        \draw[very thick] (0,2.5) edge[in=180,out=0] (1,1.5);
        \draw[very thick] (0,.5) edge[in=0,out=0] (0,1);
        \draw[very thick] (0,1.5) edge[in=0,out=0] (0,2);
      \end{tikzpicture}}}\vspace{.5em}\\
      {\Yboxdim{9pt}\footnotesize\young(1246,35)}
      }\;
      \substack{\vcenter{\hbox{
      \begin{tikzpicture}[scale=0.5]
        \draw[very thick] (0,0) edge[in=180,out=0] (1,1);
        \draw[very thick] (0,2.5) edge[in=180,out=0] (1,1.5);
        \draw[very thick] (0,.5) edge[in=0,out=0] (0,2);
        \draw[very thick] (0,1) edge[in=0,out=0] (0,1.5);
      \end{tikzpicture}}}\vspace{.5em}\\
      {\Yboxdim{9pt}\footnotesize\young(1236,45)}
      }\;
      \vcenter{\hbox{
      \begin{tikzpicture}[scale=0.5]
        \draw(0,-1) -- (0,3.5);
      \end{tikzpicture}}}\;
      \substack{\vcenter{\hbox{
      \begin{tikzpicture}[scale=0.5]
        \draw[very thick] (0,1) edge[in=180,out=0] (1,1);
        \draw[very thick] (0,1.5) edge[in=180,out=0] (1,1.5);
        \draw[very thick] (0,0) edge[in=0,out=0] (0,0.5);
        \draw[very thick] (0,2) edge[in=0,out=0] (0,2.5);
      \end{tikzpicture}}}\vspace{.5em}\\
      {\Yboxdim{9pt}\footnotesize\young(1345,26)}
      }\;
      \substack{\vcenter{\hbox{
      \begin{tikzpicture}[scale=0.5]
        \draw[very thick] (0,1) edge[in=180,out=0] (1,1);
        \draw[very thick] (0,2.5) edge[in=180,out=0] (1,1.5);
        \draw[very thick] (0,0) edge[in=0,out=0] (0,0.5);
        \draw[very thick] (0,1.5) edge[in=0,out=0] (0,2);
      \end{tikzpicture}}}\vspace{.5em}\\
      {\Yboxdim{9pt}\footnotesize\young(1245,36)}
      }\;
      \substack{\vcenter{\hbox{
      \begin{tikzpicture}[scale=0.5]
        \draw[very thick] (0,2) edge[in=180,out=0] (1,1);
        \draw[very thick] (0,2.5) edge[in=180,out=0] (1,1.5);
        \draw[very thick] (0,0) edge[in=0,out=0] (0,0.5);
        \draw[very thick] (0,1) edge[in=0,out=0] (0,1.5);
      \end{tikzpicture}}}\vspace{.5em}\\
      {\Yboxdim{9pt}\footnotesize\young(1235,46)}
      }\;
      \substack{\vcenter{\hbox{
      \begin{tikzpicture}[scale=0.5]
        \draw[very thick] (0,2) edge[in=180,out=0] (1,1);
        \draw[very thick] (0,2.5) edge[in=180,out=0] (1,1.5);
        \draw[very thick] (0,0) edge[in=0,out=0] (0,1.5);
        \draw[very thick] (0,0.5) edge[in=0,out=0] (0,1);
      \end{tikzpicture}}}\vspace{.5em}\\
      {\Yboxdim{9pt}\footnotesize\young(1234,56)}
      }.
    \end{equation*}
    Here we have drawn both the diagram basis and the corresponding tableaux.
    Note that we have partitioned the basis by the row in which the box labelled 6 occurs.

    The Gram matrix in this basis is
    \begin{equation}
        G(6,2)=
        \begin{pmatrix}
            \delta^2&\delta&\delta&1&\delta&0&0&0&0\\
            \delta&\delta^2&1&\delta&1&0&0&0&0\\
            \delta&1&\delta^2&\delta&1&\delta&1&0&0\\
            1&\delta&\delta&\delta^2&\delta&1&\delta&1&\delta\\
            \delta&1&1&\delta&\delta^2&0&1&\delta&1\\
            0&0&\delta&1&0&\delta^2&\delta&0&0\\
            0&0&1&\delta&1&\delta&\delta^2&\delta&1\\
            0&0&0&1&\delta&0&\delta&\delta^2&\delta\\
            0&1&0&\delta&1&0&1&\delta&\delta^2
        \end{pmatrix}.
    \end{equation}
    Our construction gives a new basis in terms of the idempotent $\JW_3$ and $\JW_2$.  In this case the $\JW_1$ is ``absorbed'' completely by the condition on elements of cell modules.
    \begin{equation*}
    \vcenter
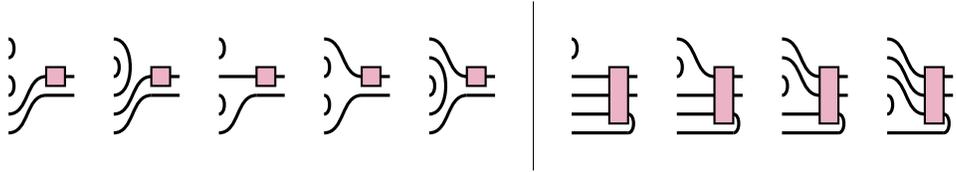
{\hbox{
      \begin{tikzpicture}[scale=0.5]
        \draw[very thick] (0,0) edge[in=180,out=0] (1,1);
        \draw[very thick] (0,0.5) edge[in=180,out=0] (1,1.5);
        \draw[very thick] (0,1) edge[in=0,out=0] (0,1.5);
        \draw[very thick] (0,2) edge[in=0,out=0] (0,2.5);
        \draw[thick,fill=purple!30!white] (1,1.25) rectangle (1.5, 1.75);
        \draw[very thick] (1,1) edge (1.75, 1);
        \draw[very thick] (1.5,1.5) -- (1.75,1.5);
      \end{tikzpicture}}}\quad
      \vcenter{\hbox{
      \begin{tikzpicture}[scale=0.5]
        \draw[very thick] (0,0) edge[in=180,out=0] (1,1);
        \draw[very thick] (0,0.5) edge[in=180,out=0] (1,1.5);
        \draw[very thick] (0,1.5) edge[in=0,out=0] (0,2);
        \draw[very thick] (0,1) edge[in=0,out=0] (0,2.5);
        \draw[thick,fill=purple!30!white] (1,1.25) rectangle (1.5, 1.75);
        \draw[very thick] (1,1) edge (1.75, 1);
        \draw[very thick] (1.5,1.5) -- (1.75,1.5);
      \end{tikzpicture}}}\quad
      \vcenter{\hbox{
      \begin{tikzpicture}[scale=0.5]
        \draw[very thick] (0,0) edge[in=180,out=0] (1,1);
        \draw[very thick] (0,1.5) edge[in=180,out=0] (1,1.5);
        \draw[very thick] (0,.5) edge[in=0,out=0] (0,1);
        \draw[very thick] (0,2) edge[in=0,out=0] (0,2.5);
        \draw[thick,fill=purple!30!white] (1,1.25) rectangle (1.5, 1.75);
        \draw[very thick] (1,1) edge (1.75, 1);
        \draw[very thick] (1.5,1.5) -- (1.75,1.5);
      \end{tikzpicture}}}\quad
      \vcenter{\hbox{
      \begin{tikzpicture}[scale=0.5]
        \draw[very thick] (0,0) edge[in=180,out=0] (1,1);
        \draw[very thick] (0,2.5) edge[in=180,out=0] (1,1.5);
        \draw[very thick] (0,.5) edge[in=0,out=0] (0,1);
        \draw[very thick] (0,1.5) edge[in=0,out=0] (0,2);
        \draw[thick,fill=purple!30!white] (1,1.25) rectangle (1.5, 1.75);
        \draw[very thick] (1,1) edge (1.75, 1);
        \draw[very thick] (1.5,1.5) -- (1.75,1.5);
      \end{tikzpicture}}}\quad
      \vcenter{\hbox{
      \begin{tikzpicture}[scale=0.5]
        \draw[very thick] (0,0) edge[in=180,out=0] (1,1);
        \draw[very thick] (0,2.5) edge[in=180,out=0] (1,1.5);
        \draw[very thick] (0,.5) edge[in=0,out=0] (0,2);
        \draw[very thick] (0,1) edge[in=0,out=0] (0,1.5);
        \draw[thick,fill=purple!30!white] (1,1.25) rectangle (1.5, 1.75);
        \draw[very thick] (1,1) edge (1.75, 1);
        \draw[very thick] (1.5,1.5) -- (1.75,1.5);
      \end{tikzpicture}}}\quad
      \vcenter{\hbox{
      \begin{tikzpicture}[scale=0.5]
        \draw(0,-1) -- (0,3.5);
      \end{tikzpicture}}}\quad
      \vcenter{\hbox{
      \begin{tikzpicture}[scale=0.5]
        \draw[very thick] (0,1) edge[in=180,out=0] (1,1);
        \draw[very thick] (0,1.5) edge[in=180,out=0] (1,1.5);
        \draw[very thick] (0,0) -- (1.5,0);
        \draw[very thick] (1.5,0) edge[in=0,out=0] (1.5,0.5);
        \draw[very thick] (0,0.5) -- (1.25,0.5);
        \draw[very thick] (0,2) edge[in=0,out=0] (0,2.5);
        \draw[thick,fill=purple!30!white] (1,0.25) rectangle (1.5, 1.75);
        \draw[very thick] (1.5,1) edge (1.75, 1);
        \draw[very thick] (1.5,1.5) -- (1.75,1.5);
      \end{tikzpicture}}}\quad
      \vcenter{\hbox{
      \begin{tikzpicture}[scale=0.5]
        \draw[very thick] (0,1) edge[in=180,out=0] (1,1);
        \draw[very thick] (0,2.5) edge[in=180,out=0] (1,1.5);
        \draw[very thick] (0,0) -- (1.5,0);
        \draw[very thick] (0,0.5) -- (1.25,0.5);
        \draw[very thick] (1.5,0) edge[in=0,out=0] (1.5,0.5);
        \draw[very thick] (0,1.5) edge[in=0,out=0] (0,2);
        \draw[thick,fill=purple!30!white] (1,0.25) rectangle (1.5, 1.75);
        \draw[very thick] (1.5,1) edge (1.75, 1);
        \draw[very thick] (1.5,1.5) -- (1.75,1.5);
      \end{tikzpicture}}}\quad
      \vcenter{\hbox{
      \begin{tikzpicture}[scale=0.5]
        \draw[very thick] (0,2) edge[in=180,out=0] (1,1);
        \draw[very thick] (0,2.5) edge[in=180,out=0] (1,1.5);
        \draw[very thick] (0,0) -- (1.5,0);
        \draw[very thick] (0,0.5) -- (1.25,0.5);
        \draw[very thick] (1.5,0) edge[in=0,out=0] (1.5,0.5);
        \draw[very thick] (0,1) edge[in=0,out=0] (0,1.5);
        \draw[thick,fill=purple!30!white] (1,0.25) rectangle (1.5, 1.75);
        \draw[very thick] (1.5,1) edge (1.75, 1);
        \draw[very thick] (1.5,1.5) -- (1.75,1.5);
      \end{tikzpicture}}}\quad
      \vcenter{\hbox{
      \begin{tikzpicture}[scale=0.5]
        \draw[very thick] (0,2) edge[in=180,out=0] (1,1);
        \draw[very thick] (0,2.5) edge[in=180,out=0] (1,1.5);
        \draw[very thick] (0,0) -- (1.5,0);
        \draw[very thick] (0,1.5) edge[in=180,out=0] (1,0.5);
        \draw[very thick] (1.5,0) edge[in=0,out=0] (1.5,0.5);
        \draw[very thick] (0,0.5) edge[in=0,out=0] (0,1);
        \draw[thick,fill=purple!30!white] (1,0.25) rectangle (1.5, 1.75);
        \draw[very thick] (1.5,1) edge (1.75, 1);
        \draw[very thick] (1.5,1.5) -- (1.75,1.5);
      \end{tikzpicture}}}.
    \end{equation*}
    Now, in this new basis, 
    \begin{align}
        G'(6,2) &=
        \begin{pmatrix}
            \delta^2&\delta  &\delta  &1       &\delta  &0       &0       &0       &0       \\
            \delta  &\delta^2&1       &\delta  &1       &0       &0       &0       &0       \\
            \delta  &1       &\delta^2&\delta  &1       &0       &0       &0       &0       \\
            1       &\delta  &\delta  &\delta^2&\delta  &0       &0       &0       &0       \\
            \delta  &1       &1       &\delta  &\delta^2&0       &0       &0       &0       \\
            0       &0       &0       &0       &0       &\frac{\delta^4-2\delta^2}{\delta^2 - 1}&\frac{\delta^3 - 2\delta}{\delta^2-1}  &0       &0       \\
            0       &0       &0       &0       &0       &\frac{\delta^3 - 2\delta}{\delta^2-1}  &\frac{\delta^4-2\delta^2}{\delta^2 - 1}&\frac{\delta^3 - 2\delta}{\delta^2-1}  &0       \\
            0       &0       &0       &0       &0       &0       &\frac{\delta^3 - 2\delta}{\delta^2-1}  &\frac{\delta^4-2\delta^2}{\delta^2 - 1}&\frac{\delta^3 - 2\delta}{\delta^2-1}  \\
            0       &0       &0       &0       &0       &0       &0       &\frac{\delta^3 - 2\delta}{\delta^2-1} &\frac{\delta^4-2\delta^2}{\delta^2 - 1}
        \end{pmatrix}\vspace{1em}\\
        &=
        \left(
            \begin{array}{c|c}
            G(5,1) & 0 \\
            \hline
            0 & \frac{\delta^3-2\delta}{\delta^2-1}G(5,3)
        \end{array}
        \right).
    \end{align}
    The factor $(\delta^3 - 2\delta)/(\delta^2 - 1)$ comes from the trace of the Jones-Wenzl idempotent, $\JW_3$.
  \end{example}

We could in fact bubble this construction down and use clasps instead of the neutral ladders in \cref{eq:inductive_diagram} always. This would give us an orthogonal basis of the cell modules and would preserve the Gram determinant.

\begin{definition}
    Consider the construction of a basis as in \cref{subsec:basis} where all neutral ladders are replaced by clasps. This defines a new basis denoted $\mathbb{L}_{\mathbf t}$ called the \emph{clasped light ladders}.
\end{definition}

\begin{example}
    Let us continue our example above and write out the clasped light ladder basis. We present the elements after simplification by absorbing all Jones-Wenzl projectors into each other.
    \begin{equation*}
    \vcenter{\hbox{
      \begin{tikzpicture}[scale=0.5]
        \draw[very thick] (0,0) edge[in=180,out=0] (1,1);
        \draw[very thick] (0,0.5) edge[in=180,out=0] (1,1.5);
        \draw[very thick] (0,1) edge[in=0,out=0] (0,1.5);
        \draw[very thick] (0,2) edge[in=0,out=0] (0,2.5);

      \end{tikzpicture}}}\quad
      \vcenter{\hbox{
      \begin{tikzpicture}[scale=0.5]
        \draw[very thick] (0,0) edge[in=180,out=0] (1,1);
        \draw[very thick] (0,0.5) edge[in=180,out=0] (1,1.5);
        \draw[very thick] (0,1.5) edge[in=0,out=0] (0,2);
        \draw[very thick] (0,1) edge[in=0,out=0] (0,2.5);
        \draw[thick,fill=purple!30!white] (-.5,1.75) rectangle (0, 2.75);
        \draw[very thick] (-.75, 0) edge (0,0);
        \draw[very thick] (-.75, 0.5) edge (0,0.5);
        \draw[very thick] (-.75, 1) edge (0,1);
        \draw[very thick] (-.75, 1.5) edge (0,1.5);
        \draw[very thick] (-.75, 2) edge (-.5,2);
        \draw[very thick] (-.75, 2.5) edge (-.5,2.5);
      \end{tikzpicture}}}\quad
      \vcenter{\hbox{
      \begin{tikzpicture}[scale=0.5]
        \draw[very thick] (0,0) edge[in=180,out=0] (1,1);
        \draw[very thick] (0,1.5) edge[in=180,out=0] (1,1.5);
        \draw[very thick] (0,.5) edge[in=0,out=0] (0,1);
        \draw[very thick] (-0.75,2) edge[in=0,out=0] (-.75,2.5);
        \draw[thick,fill=purple!30!white] (-.5,0.75) rectangle (0, 1.75);

        \draw[very thick] (-.75, 0) edge (0,0);
        \draw[very thick] (-.75, 0.5) edge (0,0.5);
        \draw[very thick] (-.75, 1) edge (-.5,1);
        \draw[very thick] (-.75, 1.5) edge (-.5,1.5);
      \end{tikzpicture}}}\quad
      \vcenter{\hbox{
      \begin{tikzpicture}[scale=0.5]
        \draw[very thick] (0,0) edge[in=180,out=0] (1,1);
        \draw[very thick] (0,2) edge[in=180,out=0] (1,1.5);
        \draw[very thick] (0,.5) edge[in=0,out=0] (0,1);
        \draw[very thick] (-1,1.5) edge[in=0,out=0] (-1,2);
        \draw[very thick] (-1.75,1.5) edge (-1,1.5);
        \draw[thick,fill=purple!30!white] (-1.5,1.75) rectangle (-1, 2.75);
        \draw[thick,fill=purple!30!white] (-.5,0.75) rectangle (0, 2.25);
        \draw[very thick] (-1.75,0) edge (0, 0);
        \draw[very thick] (-1.75,0.5) edge (0, 0.5);
        \draw[very thick] (-1.75, 2) edge (-1.5,2);
        \draw[very thick] (-1.75, 2.5) edge (-1.5,2.5);
        \draw[very thick] (-1.75, 1) edge (-.5,1);
        \draw[very thick] (-1, 2.5) edge[out=180, out=0] (-.5,2);
      \end{tikzpicture}}}\quad
      \vcenter{\hbox{
      \begin{tikzpicture}[scale=0.5]
        \draw[very thick] (0,0) edge[in=180,out=0] (1,1);
        \draw[very thick] (0,2.5) edge[in=180,out=0] (1,1.5);
        \draw[very thick] (0,.5) edge[in=0,out=0] (0,2);
        \draw[very thick] (0,1) edge[in=0,out=0] (0,1.5);
        \draw[very thick] (-.75, 0) edge (0,0);
        \draw[very thick] (-.75, 0.5) edge (0,0.5);
        \draw[very thick] (-.75, 1) edge (0,1);
        \draw[very thick] (-.75, 1.5) edge (0,1.5);
        \draw[very thick] (-.75, 2) edge (-.5,2);
        \draw[very thick] (-.75, 2.5) edge (-.5,2.5);
        \draw[thick,fill=purple!30!white] (-.5,1.25) rectangle (0, 2.75);
      \end{tikzpicture}}}\quad
      \vcenter{\hbox{
      \begin{tikzpicture}[scale=0.5]
        \draw(0,-1) -- (0,3.5);
      \end{tikzpicture}}}\quad
      \vcenter{\hbox{
      \begin{tikzpicture}[scale=0.5]
        \draw[very thick] (0,1) edge[in=180,out=0] (1,1);
        \draw[very thick] (0,1.5) edge[in=180,out=0] (1,1.5);
        \draw[very thick] (0,0) -- (.75,0);
        \draw[very thick] (.75,0) edge[in=0,out=0] (.75,0.5);
        \draw[very thick] (0,0.5) -- (.75,0.5);
        \draw[very thick] (0,2) edge[in=0,out=0] (0,2.5);
        \draw[thick,fill=purple!30!white] (.25,0.25) rectangle (.75, 1.75);
      \end{tikzpicture}}}\quad
      \vcenter{\hbox{
      \begin{tikzpicture}[scale=0.5]
        \draw[very thick] (0,1) edge[in=180,out=0] (1,1);
        \draw[very thick] (0,2.5) edge[in=180,out=0] (1,1.5);
        \draw[very thick] (0,0) -- (1.5,0);
        \draw[very thick] (0,0.5) -- (1.25,0.5);
        \draw[very thick] (1.5,0) edge[in=0,out=0] (1.5,0.5);
        \draw[very thick] (0,1.5) edge[in=0,out=0] (0,2);
        \draw[thick,fill=purple!30!white] (1,0.25) rectangle (1.5, 1.75);
        \draw[very thick] (1.5,1) edge (1.75, 1);
        \draw[very thick] (1.5,1.5) -- (1.75,1.5);
        \draw[thick,fill=purple!30!white] (-.5,1.75) rectangle (0,2.75);
        \draw[very thick] (-.75, 0) edge (0,0);
        \draw[very thick] (-.75, 0.5) edge (0,0.5);
        \draw[very thick] (-.75, 1) edge (0,1);
        \draw[very thick] (-.75, 1.5) edge (0,1.5);
        \draw[very thick] (-.75, 2) edge (-.5,2);
        \draw[very thick] (-.75, 2.5) edge (-.5,2.5);
      \end{tikzpicture}}}\quad
      \vcenter{\hbox{
      \begin{tikzpicture}[scale=0.5]
        \draw[very thick] (0,2) edge[in=180,out=0] (1,1);
        \draw[very thick] (0,2.5) edge[in=180,out=0] (1,1.5);
        \draw[very thick] (0,0) -- (1.5,0);
        \draw[very thick] (0,0.5) -- (1.25,0.5);
        \draw[very thick] (1.5,0) edge[in=0,out=0] (1.5,0.5);
        \draw[very thick] (0,1) edge[in=0,out=0] (0,1.5);
        \draw[thick,fill=purple!30!white] (1,0.25) rectangle (1.5, 1.75);
        \draw[very thick] (1.5,1) edge (1.75, 1);
        \draw[very thick] (1.5,1.5) -- (1.75,1.5);
        \draw[thick,fill=purple!30!white] (-.5,1.25) rectangle (0,2.75);
        \draw[very thick] (-.75, 0) edge (0,0);
        \draw[very thick] (-.75, 0.5) edge (0,0.5);
        \draw[very thick] (-.75, 1) edge (0,1);
        \draw[very thick] (-.75, 1.5) edge (-.5,1.5);
        \draw[very thick] (-.75, 2) edge (-.5,2);
        \draw[very thick] (-.75, 2.5) edge (-.5,2.5);
      \end{tikzpicture}}}\quad
      \vcenter{\hbox{
      \begin{tikzpicture}[scale=0.5]
        \draw[very thick] (0,2) edge[in=180,out=0] (1,1);
        \draw[very thick] (0,2.5) edge[in=180,out=0] (1,1.5);
        \draw[very thick] (0,1.5) edge[in=0,out=0] (0,0);
        \draw[very thick] (0,0.5) edge[in=0,out=0] (0,1);
        \draw[very thick] (-.75, 0) edge (0,0);
        \draw[very thick] (-.75, 0.5) edge (0,0.5);
        \draw[very thick] (-.75, 1) edge (-.5,1);
        \draw[very thick] (-.75, 1.5) edge (-.5,1.5);
        \draw[very thick] (-.75, 2) edge (-.5,2);
        \draw[very thick] (-.75, 2.5) edge (-.5,2.5);
        \draw[thick,fill=purple!30!white] (-.5,0.75) rectangle (0,2.75);
      \end{tikzpicture}}}.
    \end{equation*}
    In this basis, the Gram matrix has form
    \begin{equation}
        G''(6,2) = \operatorname{diag}\left(\delta^2, \delta^2 - 1, \delta^2 - 1, \frac{(\delta^2-1)^2}{\delta^2}, \delta^3-2\delta, 
        \frac{\delta^3-2\delta}{\delta^2-1}, {\delta^2-2}, \frac{(\delta^3 - 2\delta)^2}{(\delta^2-1)^2}, \frac{\delta^4-3\delta^2 + 1}{\delta^2-1}
        \right)
    \end{equation}
\end{example}

\begin{remark}
  The importance of preserving the Gram determinant comes from having a view to the modular theory of webs.
  The relations that webs satisfy are given in terms of polynomials\footnote{Indeed, despite being written as rational functions they are in fact polynomials in $\delta$. This is because the relations all involve factors of the form $\gaussianquant{a}{b}$ which are all polynomials in $\delta$ (see \cref{sec:quantum}). This means the relations are well defined in any given field.} in $\delta$.
  Thus webs form a $\Z[\delta]$-spanning set for the category, and the basis constructed by Elias gives a $\Z[\delta]$ form.
  As such, given a suitable $p$-modular system (or rather, an $(\ell, p)$-modular system as in~\cite{martin_spencer_2021}), we can consider the modular representation theory of this category.
  
  Here, the cell modules are almost never simple, and the Gram determinant identifies this, vanishing when the cell module is not.
  Moreover there are filtrations of the cell modules that can be derived from the elementary divisors of the Gram matrix.
\end{remark}

Of course \cref{eq:recursive-det} is recursive in nature.  It behoves us to write down a ``closed form''.
  To do this, we must simply enumerate all the paths from $\pmb{\varnothing}$ to $\lambda$ in $\Y_{\underline{x}}^n$.
  Let this set be denoted $T(\underline{x}, \lambda)$ which has cardinality $d(\underline x, \lambda) = \dim S(\underline x, \lambda)$.  
  Then
  \begin{equation}
    \det G(\underline x, \lambda) = \prod_{t \in T(\underline x,\lambda)}\;\prod_{\substack{e\,:\, \mu_1 \to \mu_2\\ \in t}} (\kappa_{\mu_1}^{\mu_2})^{d(\underline x,\mu_1)}.
  \end{equation}
  We can write this in a different way:
  instead of counting edges by paths, we can count paths by edges.  Let $d(\underline x, \lambda\setminus\mu)$ be the number of paths in $\Y_{\underline x}^n$ from $\mu$ to $\lambda$.
  Then
  \begin{equation}\label{eq:final_form}
    \det G(\underline x,\lambda) = \prod_{\substack{\mu_1 \to \mu_2\\e \in \mathbb{Y}_n}}
    (\kappa_{\mu_1}^{\mu_2})^{d(\underline x,\mu_1) d(\underline x,\lambda \setminus \mu_2)}
  \end{equation}
  We note that $d(\underline x,\mu_1) d(\underline x,\lambda \setminus \mu_2)$ is exactly the number of paths from $\emptypart$ to $\lambda$ passing through both $\mu_1$ and (immediately subsequently) $\mu_2$. That is, it is the number of paths that include the edge $\mu_1 \to \mu_2$.
  
  \begin{example}
    If $n = 2$ so we are in the Temperley-Lieb case, then we must have $\underline{x} = \underline{\bullet^n}$ and the truncated Young poset $\Y_{\underline{\bullet^n}}$ is simply the branching graph
    
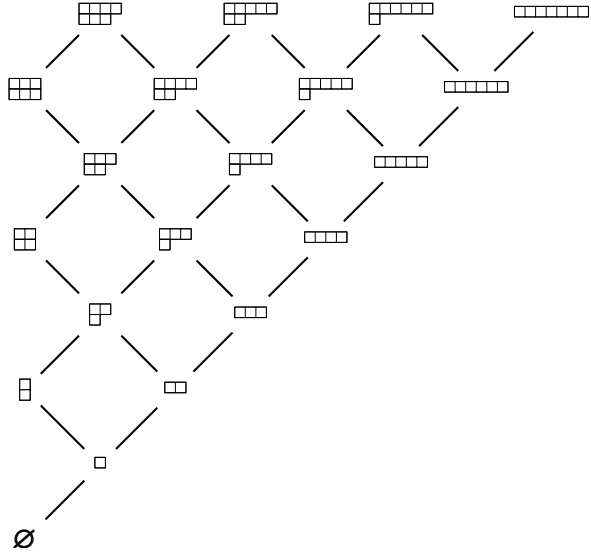
\begin{center}
        \begin{tikzpicture}[xscale=-2]
          \node (empt) at (0,0) {$\emptypart$};
          
          \node (1) at (-.5,1) {$\smolyng{(1)}$};
          
          \node (2) at (-1,2) {$\smolyng{(2)}$};
          \node (11) at (0,2) {$\smolyng{(1,1)}$};
          
          \node (3) at (-1.5,3) {$\smolyng{(3)}$};
          \node (21) at (-.5,3) {$\smolyng{(2,1)}$};
                    
          \node (4) at (-2,4) {$\smolyng{(4)}$};
          \node (31) at (-1,4) {$\smolyng{(3,1)}$};
          \node (22) at (0,4) {$\smolyng{(2,2)}$};
          
          \node (5) at (-2.5,5) {$\smolyng{(5)}$};
          \node (41) at (-1.5,5) {$\smolyng{(4,1)}$};
          \node (32) at (-.5,5) {$\smolyng{(3,2)}$};
          
          \node (6) at (-3,6) {$\smolyng{(6)}$};
          \node (51) at (-2,6) {$\smolyng{(5,1)}$};
          \node (42) at (-1,6) {$\smolyng{(4,2)}$};
          \node (33) at (0,6) {$\smolyng{(3,3)}$};
          
          \node (7) at (-3.5,7) {$\smolyng{(7)}$};
          \node (61) at (-2.5,7) {$\smolyng{(6,1)}$};
          \node (52) at (-1.5,7) {$\smolyng{(5,2)}$};
          \node (43) at (-.5,7) {$\smolyng{(4,3)}$};
          
          \draw[thick] (empt) -- (1);
          \draw[thick] (1) -- (2);
          \draw[thick] (2) -- (3);
          \draw[thick] (3) -- (4);
          \draw[thick] (4) -- (5);
          \draw[thick] (5) -- (6);          
          \draw[thick] (6) -- (7);
          
          \draw[thick] (1) -- (11);
          \draw[thick] (2) -- (21);
          \draw[thick] (3) -- (31);
          \draw[thick] (4) -- (41);
          \draw[thick] (5) -- (51);
          \draw[thick] (6) -- (61);
          
          \draw[thick] (21) -- (22);
          \draw[thick] (31) -- (32);
          \draw[thick] (41) -- (42);
          \draw[thick] (51) -- (52);

          \draw[thick] (32) -- (33);
          \draw[thick] (42) -- (43);
          
          \draw[thick] (11) -- (21);
          \draw[thick] (21) -- (31);
          \draw[thick] (31) -- (41);
          \draw[thick] (41) -- (51);
          \draw[thick] (51) -- (61);

          \draw[thick] (22) -- (32);
          \draw[thick] (32) -- (42);
          \draw[thick] (42) -- (52);

          \draw[thick] (33) -- (43);
        \end{tikzpicture}
    \end{center}
%    If we label the partition $(a,b)$ by $(n,m) = (a+b, (a-b)/2)$ so that it is in the $n$th row and is $m$th column, then we know that $d(n, m)$ is the number of Dyck paths from $\emptypart$ to $(n,m)$ which is given by $\binom{n}{(n-m)/2} - \binom{n}{(n-m)/2-1}$.
%    On the other hand, the number of paths from $(n_1, m_1)$ to $(n_2, m_2)$ is given by
%    $$
%    \binom{n_2-n_1}{\frac{1}{2}(n_2-n_1 - m_2+m_1)}-
%    \binom{n_2-n_1}{\frac{1}{2}(n_2-n_1 + m_2+m_1)}
%    $$
    
    The intersection forms can be evaluated by tracing the Jones-Wenzl idempotents and are known to be
    \begin{equation*}
      \kappa_{(n,m)}^{(n+1,m+1)} = 1 \quad \quad \text{and} \quad\quad
      \kappa_{(n,m)}^{(n+1,m-1)} = \frac{[m+1]}{[m]}.
    \end{equation*}
%    Hence, by \cref{eq:final_form},
%    \begin{align*}
%        \det G(x,\lambda) &=
%        \prod_{\substack{(n_1,m_1) \to (n_2,m_2)\\\in \mathbb{Y}_n}}
%        \left(\kappa_{(n_1,m_1)}^{(n_2,m_2)}\right)^{d(n_1, m_1) d(\underline x,\lambda \setminus (n_2,m_2))} \\
%        &=
%        \prod_{\substack{(n_1,m_1) \to (n_1+1,m_1-1)\\\in \mathbb{Y}_n}}
%        \left(\frac{[m_1 + 1]}{[m_1]}\right)^{d(n_1, m_1) d(\underline x,\lambda \setminus (n+1,m_1-1))} \\
%        &=
%        \prod_{n = 0}^{x}\prod_{m=1,2}^n
%        \left(\frac{[m + 1]}{[m]}\right)^{d(n, m) d(\underline x,\lambda \setminus (n,m-1))} \\
%    \end{align*}
\end{example}
  
  \begin{example}
    More generally, if $\underline x = \underline{1^n}$, so that paths are standard tableaux,
    $d(\underline x,\mu_1) d(\underline x,\lambda \setminus \mu_2)$ counts the number of standard tablueax of shape $\lambda$ with a $|\mu_2|$ in the box $\mu_2 \setminus \mu_1$.
    
    For a concrete example, 
    if $\mu_1 = (7,6,4,3,1)$ and $\mu_2 = (7,6,5,3,1)$ for $\lambda = (8,7,6,6,3,2,1)$, then we are counting all standard tableaux of the form below, where the numbers $1$ through $21 = |\mu_1|$ are placed in the pink boxes, 22 is as indicated in the green box and the remaining 23 through 33 are in the blue boxes.
    
    
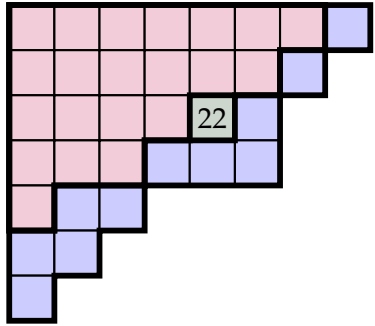
\vspace{1em}
    \begin{center}
        \begin{tikzpicture}[scale=0.6]
            \foreach \l/\r in {7/1,6/2,4/3,3/4,1/5} {
              \foreach \i in {1,...,\l} {
                \draw[thick,fill=purple!20!white] (\i,-\r) rectangle (\i+1,-\r-1);
              }
            }
            \foreach \s/\l/\r in {8/8/1,7/7/2,6/6/3,4/6/4,2/3/5,1/2/6,1/1/7} {
              \foreach \i in {\s,...,\l} {
                \draw[thick,fill=blue!20!white] (\i,-\r) rectangle (\i+1,-\r-1);
              }
            }
            \foreach \s/\l/\r in {5/5/3} {
              \foreach \i in {\s,...,\l} {
                \draw[thick,fill=green!20!black!20!white] (\i,-\r) rectangle (\i+1,-\r-1);
              }
            }
            \draw[line width=0.2em] (1,-1) -- (8,-1) -- (8,-2) -- (7,-2) --(7,-3) -- (6,-3) -- (5,-3) -- (5,-4)--(4,-4) -- (4,-5)--(2,-5)--(2,-6) -- (1,-6) -- cycle;
            \draw[line width=0.2em] (1,-1) -- (8,-1) -- (8,-2) -- (7,-2) --(7,-3) -- (6,-3) -- (6,-4) -- (5,-4)--(4,-4) -- (4,-5)--(2,-5)--(2,-6) -- (1,-6) -- cycle;
            \draw[line width=0.2em] (1,-1) -- (9,-1) -- (9,-2) -- (8,-2)  --(8,-3) -- (7,-3) -- (7,-5) -- (6,-5)--(6,-5) -- (4,-5)--(4,-6)--(3,-6)--(3,-7)--(2,-7)--(2,-8) -- (1,-8) -- cycle;
            \node at (5.5,-3.5) {$22$};
        \end{tikzpicture}
    \end{center}
  \end{example}
  
  \begin{remark}
  This process should work for other categories.
  In particular, we only needed
  \begin{enumerate}
    \item a category with a suitable poset of weights;
    \item  neutral ladders (which are often not even necessary because the object set and the weight set coincide);
    \item \clasp s, which are often inductively definable, %what paper?
    \item a reasonable way of constructing bases for cell modules which looks like the ladder construction (often a byproduct of the construction of \clasp s being constructed by an iterative ladder).
  \end{enumerate}
  In particular, we should be able to calculate the determinants for categories such as the planar rook monoid category (where we expect the answer to be ``not zero'', since all cell modules are irreducible over arbitrary characteristic) and coloured $\TL$ algebras.
\end{remark}

  As a final remark, we note that in any path through the Young graph, the number of parts of the partition is non-decreasing.  Hence we can remove the restriction $e \in \mathbb{Y}^n_{\underline x}$ in \cref{eq:final_form} and replace it with $e \in \mathbb{Y}_{\underline{x}}$, eliminating the dependence on $n$.
  
  %error:  two different edges , both labelled [2] whihc have mult 2.
\begin{example}\label{eg:big}
  We complete the example that we have been building up to through the paper so far.
  
  Let us compute the Gram determinant for $S(\underline{32312}, \underline{4232})$ for $n=5$.
  That is, in maximal verbosity, we are considering the space
  \begin{equation}
      \Hom_{U_q(\mathfrak{sl}_5)}\left(
      V_3\otimes
      V_2\otimes
      V_3\otimes
      V_1\otimes
      V_2,
      V_4\otimes V_2\otimes V_3\otimes V_2\right)
  \end{equation}
  modulo morphisms factoring through representations of weight lower than the source,
  where $V_i = \bigwedge^i \C^5$.
  
  The weight of the source is
  \begin{equation}
      \lambda = \wt(\underline{4232}) = \colwei{(4,4,2,1)}
  \end{equation}
  and so we should consider the truncated Young poset $\Y^5_{\underline{32312}}$ of weights at most $\wt(\underline{4232})$.
  
  We will draw this poset, where each node has the number of paths from $\emptypart$ to that node annotated in blue and the number of paths to $\lambda$ in red. Each edge is labelled with the appropriate value of $\kappa$, which we can read from \cref{eg:all_kappa}.
  \begin{center}
    \begin{tikzpicture}[scale=0.7]
    \node (empt) at (0,-4) {${\emptypart}_{\color{blue!80!black}\mathbf{1}}^{\color{red!80!black}\mathbf{6}}$};
    
      \node (111) at (0,-1) {$\colweib{(1,1,1,0)}_{\color{blue!80!black}\mathbf{1}}^{\color{red!80!black}\mathbf{6}}$};
      
      \node (2111) at (-2,1) {$\colweib{(2,1,1,1)}_{\color{blue!80!black}\mathbf{1}}^{\color{red!80!black}\mathbf{1}}$};
      \node (221) at (2,1) {$\colweib{(2,2,1,0)}_{\color{blue!80!black}\mathbf{1}}^{\color{red!80!black}\mathbf{5}}$};
      
      \node (3221) at (-4,5) {$\colweib{(3,2,2,1)}_{\color{blue!80!black}\mathbf{2}}^{\color{red!80!black}\mathbf{1}}$};
      \node (3311) at (4,5) {$\colweib{(3,3,1,1)}_{\color{blue!80!black}\mathbf{1}}^{\color{red!80!black}\mathbf{2}}$};
      \node (332) at (0,5) {$\colweib{(3,3,2,0)}_{\color{blue!80!black}\mathbf{1}}^{\color{red!80!black}\mathbf{2}}$};
      
      \node (432) at (-4,8) {$\colweib{(4,3,2,0)}_{\color{blue!80!black}\mathbf{1}}^{\color{red!80!black}\mathbf{1}}$};
      \node (3321) at (0,8) {$\colweib{(3,3,2,1)}_{\color{blue!80!black}\mathbf{4}}^{\color{red!80!black}\mathbf{1}}$};
      \node (4311) at (4,8) {$\colweib{(4,3,1,1)}_{\color{blue!80!black}\mathbf{1}}^{\color{red!80!black}\mathbf{1}}$};

      \node (4421) at (0,11) {$\colweib{(4,4,2,1)}_{\color{blue!80!black}\mathbf{6}}^{\color{red!80!black}\mathbf{1}}$};
      
      \draw[very thick] (empt) -- node[midway, xshift=-6pt] {\small$1$} (111) ;
      \draw[very thick] (111) -- node[midway, xshift=-10pt,yshift=-2pt] {\small$[3]$} (2111);
      \draw[very thick] (111) -- node[midway, xshift=8pt,yshift=-2pt] {\small$1$} (221);
      
      \draw[very thick] (2111) -- node[midway, xshift=-6pt,yshift=-1pt] {\small$1$} (3221);
      \draw[very thick] (221) -- node[midway, xshift=5pt,yshift=9pt] {\small$\frac{[2][4]}{[3]}$} (3221);
      \draw[very thick] (221) -- node[midway, xshift=10pt,yshift=-2pt] {\small$[2]$} (3311);
      \draw[very thick] (221) -- node[midway, xshift=7pt,yshift=2pt] {\small$1$} (332);
      
      \draw[very thick] (3221) -- node[midway, xshift=6pt,yshift=15pt] {\small$[2]$} (3321);
      \draw[very thick] (3311) -- node[midway, xshift=8pt,yshift=8pt] {\small$\frac{[4]}{[2]}$} (3321);
      \draw[very thick] (332) -- node[midway, xshift=12pt] {\small$\frac{[3][6]}{[2][4]}$} (3321);
      \draw[very thick] (3311) -- node[midway, xshift=4pt,yshift=-1pt] {\small$1$} (4311);
      \draw[very thick] (332) -- node[midway, xshift=-18pt,yshift=6pt] {\small$1$} (432);
      
      \draw[very thick] (3321) -- node[midway, xshift=4pt,yshift=0pt] {\small$1$} (4421);
      \draw[very thick] (4311) -- node[midway, xshift=8pt,yshift=8pt] {\small$\frac{[2][5]}{[4]}$}  (4421);
      \draw[very thick] (432) -- node[midway, xshift=-8pt,yshift=8pt] {\small$\frac{[3][7]}{[6]}$} (4421);
      
    \end{tikzpicture}
\end{center}
From this, the algorithm for computing the determinant is simple. Simply multiply together all the coefficients on the edges with multiplicity given by the product of the blue number at their bottom and red number at their top. 
In our example, only two edges are simultaneously labelled with a number other than 1 and appear multiple paths from the bottom to the top: $\colwei{(2,2,1,0)}\to \colwei{(3,3,1,1)}$ and $\colwei{(3,2,2,1)}\to \colwei{(3,3,2,1)}$. Both have associated intersection form $[2]$. We can then write
\begin{align}
    \det G(\underline{32312}, \underline{4232}) &=
    [3]\cdot
    \frac{[2][4]}{[3]}\cdot
    [2]^2\cdot
    [2]^2\cdot
    \frac{[3][6]}{[2][4]}\cdot
    \frac{[4]}{[2]}\cdot
    \frac{[3][7]}{[6]}\cdot
    \frac{[2][5]}{[4]}\\
    &=
    [2]^4[3]^2[5][7]\nonumber
\end{align}
We note that this is in fact a polynomial in $[2]$ despite being expressed as a product of rational functions.
\end{example}
\begin{lemma}
  The Gram determinant $\det G(\underline{x},\underline{y})$ is a polynomial in $[2]$.
\end{lemma}
\begin{proof}
  Consider calculating the Gram matrix in the $L_{\mathbb t}$ basis.
  The elements of this basis are diagrams and the composition of diagrams lies in the $\Z[\delta]$-span of other diagrams.
  Hence each inner product lies in $\Z[\delta]$ so the determinant of this matrix must too.
\end{proof}

\begin{remark}
  In \cref{eg:big} the result was expressible as a product of quantum integers.
  However, this is not always possible.
\end{remark}

%%\section{Other Cellular, Monoidal Categories}
%%\input{monoidal}

\section{Howe Duality}\label{sec:howe}
Consider the space $\C^n \otimes \C^m$.
This carries an action of $U_q(\Sl_n)$ by action on the first component, and of $U_q(\Gl_m)$ by action on the second.
These can be extended to commuting actions on $\bigwedge_q^t(\C^n \otimes \C^m)$ for any $k$.
Here the quantum symmetric square is defined as
\begin{equation}
    \bigwedge{}_q^{\hspace{-0.3em}t}(\C^n \otimes \C^m) = \left(T(\C^n \otimes \C^m) / \langle S^2_q(\C^n \otimes \C^m)\rangle \right)_t.
\end{equation}
That is to say, it is the $t$-th homogenous level of the tensor algebra of $\C^n \otimes \C^m$ quotiented out by the ideal generated by the quantum symmetric square.
Exact generators and relations for the external algebra $\bigwedge{}_q^{\hspace{-0.3em}\bullet}(\C^n \otimes \C^m)$ can be found in \cite[\S 4.2]{cautis_kamnitzer_morrison_2014}

%  It seems to be gl not sl  (I think this was pointed out by Kleshchev?)
% I think we just need surjectivity see CKM sec. 4.4
% No, hang on g, CKM14 certainly says they generate
\begin{theorem}\label{thm:howe_comm}\cite[Thm. 4.2.2]{cautis_kamnitzer_morrison_2014}
  The actions of $U_q(\Sl_n)$ and $U_q(\Gl_m)$ on $\bigwedge^k(\C^n \otimes \C^m)$ generate each other's commutants.
\end{theorem}

Now, 
\begin{align}
\bigwedge{}^{\hspace{-0.4em}t}_q \left(\C^n \otimes \C^m\right) &\simeq 
\bigwedge{}^{\hspace{-0.4em}t}_q (\C^n \oplus\cdots\oplus \C^n) \\&\simeq 
\bigoplus_{\substack{\underline{t}\,\vDash\, t\\\ell(\underline{t}) = m} }\left(\bigwedge{}^{\hspace{-0.4em}t_1}_q\C^n\otimes\bigwedge{}^{\hspace{-0.4em}t_2}_q\C^n\otimes\cdots\otimes\bigwedge{}^{\hspace{-0.4em}t_m}_q\C^n \right)=
\bigoplus_{\substack{\underline{t}\vDash t\\\ell(\underline{t}) = m} } V_{\underline t}\nonumber
\end{align}
as $U_q(\Sl_n)$ representations.
Here the direct sum runs over all compositions $\underline{t}$ of $t$ in $m$ parts\footnote{Strictly they should be in not more than $m$ parts, but we take them to be exactly $m$ parts by adding zeros.} with no part longer than $n$.
%footnote - seems OK to me but not to referee
Thus,
\renewcommand{\arraystretch}{1.5}
\begin{equation}\label{eq:bigmatrix}
    \End_{U_q(\Sl_n)}\left(\bigwedge{}^{\hspace{-0.4em}t}_q \left(\C^n \otimes \C^m\right)\right)
    \simeq
   \left(
   \begin{array}{c|c|c|c}
   \End(V_{\underline{a}}) & \Hom(V_{\underline a},V_{\underline{b}}) & \cdots & \Hom(V_{\underline a},V_{\underline c})\\
   \hline
   \Hom(V_{\underline{b}},V_{\underline{a}}) & \End(V_{\underline b}) & \cdots & \vdots \\
   \hline
   \vdots & \vdots & \ddots &\vdots \\
   \hline
   \Hom(V_{\underline c}, V_{\underline a}) & \Hom(V_{\underline c},V_{\underline{b}}) & \cdots & \End(V_{\underline c})
   \end{array}
   \right),
\end{equation}
\renewcommand{\arraystretch}{1} where $\{\underline a, \underline b, \ldots, \underline c\}$ are all the appropriate partitions of $t$.

\Cref{thm:howe_comm} states that there is a surjection $\phi$ of $U_q(\Gl_m)$ onto $ \End_{U_q(\Sl_n)}\left(\bigwedge^k \left(\C^n \otimes \C^m\right)\right)$ and its kernel is computed in~\cite[\S4.4]{cautis_kamnitzer_morrison_2014}.
%In fact, this surjection and its kernel are instrumental in deducing the correct relations for the web categories.

%  referee makes an bistorical adjustment here.  Something like:  Morrison in his thesis had found the full set of relations but was unable to show that they were sufficient.  He alos derived another set of relations known as the Kekule relations which were consequences of the other relations.  The surjection given here together with the kernel are instrumental in producing a minimal set of relations inside the ones found by Morrison.  

If we compose $\phi$ with projection onto the appropriate component in \cref{eq:bigmatrix} we obtain the families of surjections
\begin{equation*}
    \phi_{\underline a, \underline b}: U_q(\Gl_m) \to \Hom_{U_q(\Sl_n)}(V_{\underline a}, V_{\underline b})
\end{equation*}
and, composing with the quotient by $J_{<\lambda}$, the further surjection
\begin{equation*}
    \phi^{\not < \wt \underline b}_{\underline a, \underline b}: U_q(\Gl_m) \to \Hom^{\not < \wt \underline b}_{U_q(\Sl_n)}(V_{\underline a}, V_{\underline b}) = S(\underline a, \underline b)
\end{equation*}
where $\Hom^{\not < \wt \underline b}_{U_q(\Sl_n)}$ is the homomorphism space modulo $J_{<\lambda}$.

%Referee wants defn of "coefficient of identity" and of $f_{\lambda}$ earlier.
\subsection{Cellular forms}\label{subsec:cellular2}
Consider now a morphism $V_{\underline b}\to V_{\underline b}$ where $\underline b$ is of $U_q(\Sl_n)$ weight $\lambda$.
We know that there is a natural surjective map of algebras $f_\lambda: \End_{U_q(\Sl_n)}(V_{\underline b}) \twoheadrightarrow \C$.
This arises as the quotient of the algebra $\End_{U_q(\Sl_n)}(V_{\underline b})$ by its maximal ideal of co-dimension 1\footnote{We could, should we desire, extend this to a functor $f_\lambda$ on each $\Hom(V_{\underline a}, V_{\underline b})$ of the same weight $\lambda$ by choosing a family of neutral ladders and defining $f_\lambda$ to quotient out by all morphisms factoring through objects of weight less than $\lambda$ while still sending the neutral ladders to 1. However, we will not need this level of generality.}.
Recall that this gives rise to the cellular form on $S(\underline b, \lambda)$ by $\langle x, y\rangle = f_\lambda(\iota y \circ x)$.

On the other hand, consider the image of the torus of $U_q(\Gl_m)$, called $U_q(\mathfrak{h})$, in \cref{eq:bigmatrix}.
The element $h \in \mathfrak{h}$ gets sent to a diagonal matrix where the component on $\End(V_{\underline{b}})$ is simply $\underline{b}(h)\mathbb{1}_{V_{\underline{b}}}$, where $\underline b$ is considered a weight of $\Gl_m$.
Under the quotient by the maximal ideal, this is again sent to $\underline{b}(h)$ in $\C$.

Let $\pi:U_q(\Gl_m) \to U_q(\mathfrak{h})$ be the projection in the decomposition $$U_q(\Gl_m) = \big(\mathfrak{n}_-U_q(\Gl_m) + U_q(\Gl_m)\mathfrak{n}_+\big) \oplus U_q(\mathfrak h).$$
Suppose $\Psi \in \End(V_{\underline b})$ is the image of $g \in U_q(\Gl_m)$ under $\phi$ and the appropriate projection.
% I wonder if we should use \dot{U}_q(gl_n) or if we can try get away without idempotent completion still
Then the image of $\pi(g)$ in $\End(V_{\underline b})$ is $\underline b(\pi(g))\mathbb{1}_{V_{\underline b}} $.
%referee doesn't like Hence...
All other terms in $\Psi$, which is to say $\Psi - \pi(\Psi)$, factor through the maximal ideal of $\End(V_{\underline b})$, and so
$\underline b(\pi(g))\mathbb{1}_{V_{\underline b}} = f_\lambda(\Psi) \mathbb{1}_{V_{\underline b}}$ and thus $$\underline b(\pi(g)) = f_\lambda(\Psi).$$

There is an anti-automorphism, $\sigma$, of $U_q(\Gl_m)$ that fixes $U_q(\mathfrak{h})$ pointwise and sends $e_{i}$ to $f_{i}$ and vice-versa.
The $\underline b$-weighted bi-linear pairing is then the map
\begin{align*}
    (-,-)_{\underline b} : U_q(\Gl_m)\times U_q(\Gl_m) &\to \C\\
    (g_1, g_2) &\mapsto \underline b\left(\pi(\sigma(g_1) \cdot g_2)\right).
\end{align*}
This pairing is also invariant in that $(xg_1,g_2) = (g_1, \sigma(x) g_2)$ for any $x \in U_q(\Gl_m)$.
%define $x$.  Is it b?
The above discussion shows that
\begin{equation}
    (g_1, g_2)_{\underline b} = \left\langle \Psi_{\underline b, \underline b}(g_1), \Psi_{\underline b, \underline b}(g_2)\right\rangle.
\end{equation}
In particular, if we are able to find preimages of the clasp basis $\mathbb{L}_t$ under $\phi$, we can evaluate the form $(-,-)_{\underline b}$ on these pre-images to determine the form $\langle -, - \rangle$.
% notational confusion with \Psi?

%  Hereon in is where referee wants the Tolstoy stuff and a general beefing up of this section.*

\subsection{Extremal Projectors}
It turns out that it becomes somewhat simpler to use another algebra called the \textit{Taylor extension of $U_q(\Gl_m)$}.
\begin{definition}\cite[\S 6]{tolstoy_2001}
Let $\Phi^+$ be the positive roots as before, and for $\gamma \in \Phi^+$ let $\{c_i^{(\gamma)}\;:\;1 \le i< m\}$ be the set of coefficients obtained by decomposing $\gamma$  as a sum of simple roots.
The \emph{Taylor extension of $U_q(\Gl_m)$} which we write as $T_q(\Gl_m)$ is the vector space over the field of fractions of $U_q(\mathfrak h)$ consisting of formal sums of multiples of monomials
\begin{equation}
    e_{-\beta}^{n_\beta}\cdots e_{-\alpha}^{n_\alpha}e_{\alpha}^{m_\alpha}\cdots e_{\beta}^{m_\beta}.
\end{equation}
Here $e_\alpha, \ldots, e_\beta$ are the root vectors of the quantum Cartan-Weyl basis in the usual order, and, for each formal sum in $T_q(\Gl_m)$, there is a $N>0$ such that the non-negative integers $n_\alpha,\ldots, n_\beta$ and $m_\alpha, \ldots, m_\beta$ appearing in the sum satisfy the $m-1$ inequalities
\begin{equation}
    \left |\sum_{\gamma \in \Phi_+} (n_\gamma - m_\gamma) c_i^{(\gamma)} \right| \le N\quad\quad\quad i = 1,2,\ldots, m-1.
\end{equation}

\end{definition}
It is then the case that the Taylor extension of $U_q(\Gl_m)$ is an associative algebra, and in particular, acts on finite-dimensional $U_q(\Gl_m)$-modules in the natural way. One should note that even if the element of $T_q(\Gl_m)$ acting on finite-dimensional weight module $M$ has infinitely many terms, by weight arguments, all but a finite number of its monomials will act as zero on the module. Further it may  happen that an arbitrary element of $T_q(\Gl_m)$ involves a scalar coefficient that does not act naturally on the module. Indeed, the scalars are members of the field of fractions of $U_q(\mathfrak{h})$ which may have denominators that vanish when "acting" on certain weight spaces. However, it will turn out that the elements of $T_q(\Gl_m)$ we work with do act in well defined ways on the modules of interest.

We note that there is a natural embedding of $U_q(\Gl_m)$:
\begin{equation}
\begin{tikzcd}
U_q(\Gl_m) \arrow[rd, "\phi"] \arrow[r, hook] & T_q(\Gl_m) \arrow[d, "\phi"] \\
& \End_{U_q(\Sl_n)}(V_{\underline b}).
\end{tikzcd}
\end{equation}
Elements of the Taylor extension that act on $V_{\underline{b}}$ inherit a pairing from $\End_{U_q(\Sl_n)}(V_{\underline b})$ in the same way that $U_q(\Gl_m)$ does and in fact this is the natural extension of the pairing $(-,-)_{\underline b}$.

The following are important elements of $T_q(\Gl_m)$ that we will use to construct our clasped ladders.
\begin{definition}\cite[Thm. 1]{tolstoy_1990}\cite[Thm 5.14]{kalmykov_safronov_2022}
  The \emph{extremal projector} $P_q(\Gl_{m}) \in T_q(\Gl_m)$ is the unique idempotent satisfying
  $$ e_{\alpha_i} P_q(\Gl_m) = P_q(\Gl_m)e_{-\alpha_i} = 0.$$
  It satisfies $\sigma( P_q(\Gl_m)) = P_q(\Gl_m)$, where $\sigma$ is the natural extension of the anti-automorphism of $U_q(\Gl_m)$. It acts on all $U_q(\Gl_m)$-modules with generic weights on which the generators $e_{\alpha_i}$ act nilpotently, in particular $\End_{U_q(\Sl_n)}(V_{\underline b})$.
\end{definition}

\begin{lemma}
  The extremal projector $P_q(\Gl_m)$ has image under $\phi$ (and the appropriate projections) in $\End_{U_q(\Sl_n)}(V_{\underline b})$ equal to $\JW_{\underline b}$.
\end{lemma}
\begin{proof}
Under the action of $\phi$, we obtain an element of $ \End_{U_q(\Sl_n)}\left(\bigwedge{}^{\hspace{-0.4em}t}_q \left(\C^n \otimes \C^m\right)\right)$ that is idempotent, invariant under $\iota$, and that is killed on the left by all $e_{\alpha_i}$.

Further, from \cite[Eq. 7.3]{tolstoy_2001} we see that the extremal projectors are homogeneous. As such, $\phi(P_q(\Gl_m))$ is only non-zero on the diagonal when viewed as a matrix via \cref{eq:bigmatrix}. Namely, if $$f_{\underline b}:\End_{U_q(\Sl_n)}\left(\bigwedge{}^{\hspace{-0.4em}t}_q \left(\C^n \otimes \C^m\right)\right)\to\End_{U_q(\Sl_n)}\left(V_{\underline b}\right)$$ is the projection, then $\phi(P_q(\Gl_m)) = \sum_{\underline b}f_{\underline b}(\phi(P_q(\Gl_m)))$. Indeed, if $V_{\underline a}$ and $V_{\underline b}$ are of different weight, then the image of $\phi(P_q(\Gl_m))$ on $\End_{U_q(\Sl_n)}(V_{\underline a}, V_{\underline b})$ must be zero. Now, if $\underline a$ and $\underline b$ have the same weight (i.e. they are permutations of each other) and correspond to partition $\underline{\mu}$, we say they lie in the same block. Each such block is of the form $V_{\underline \mu}^{\oplus n_{\underline \mu}}$ for some multiplicity $n_{\underline \mu}$, and we can pick a basis such that the action of $\phi(P_q(\Gl_m))$ is block diagonal.

Now the relevant properties of the extremal projector carry over to these images. Each $f_{\underline b}(\phi(P_q(\Gl_m)))$ is invariant under $\iota$, is idempotent, and killed by the action of $e_{\alpha_i}$.  Moreover, in the construction of the extremal element in~\cite{tolstoy_1990} we see that $1 \in U_q(\Gl_m)$ appears with coefficient 1 and so by \cref{thm:jw_exists_unique} this is a clasp.

Thus the image under $\phi$ must be the sum of \emph{all} the clasps.
In particular, its projection onto some $\End(V_{\underline{b}})$ must be the clasp $\JW_{\underline{b}}$.
\end{proof}

Armed with extremal projectors, we next turn to distinguished elements of $T_q(\Gl_m)$ known as raising- and lowering operators.
They are so-named as they are an orthonormal set of operators that move weight vectors up and down between the weights of Weyl modules.

First however, we pause to set some notation for tableaux.

\subsection{Notation for Tableaux}
In what follows, we will be manipulating row- and column-semi-standard tableaux, and we set out the required notation here.

Let $T$ be a (row) semi-standard tableaux and 
recall that semi-standard tableaux describe paths in the restricted Young graph. Let $T^{(i)}$ be the partition described by all boxes in $T$ with entries at most $i$.
\begin{example}   % referee says this is misplaced.
If
\begin{equation} T = \vcenter{\hbox{\young(1113,224,33,55)}}\end{equation}
Then $T^{(1)} = (3)$, $T^{(2)} = (3,2)$, $T^{(3)} = (4,2,2)$, $T^{(4)} = (4,3,2)$ and $T^{(5)} = (4,3,2,2)$. We can also read out the partition pieces, such as $T^{(2)}_1 = 3$ and $T^{(4)}_3=2$.
\end{example}

%referee doesn't like the notation.....

\subsection{Lowering Operators}
%quesry about s,m.   Says first formula is incomprehensible.....  Queries ket notati Queries the \lambda_ij's
%We present the key lemma which that unifies the discussion about clasped light ladders and previous on raising and lowering operators.
These operators were introduced by Nagel and Moshinsky~\cite{nagel_moshinsky_1965} and exposited by Carter at the Arcata conference~\cite{carter_1986}.
In these papers they were defined for $U(\Gl_m)$, but in 1990, Tolstoy determined the appropriate elements for $U_q(\Gl_m)$~\cite{tolstoy_1990}, which we present below in the guise of a basis of Weyl modules.
This basis is of the Gelfand-Tsetlin form: obtained through the canonical inclusion of algebras
$$U_q(\Gl_1)\subset U_q(\Gl_1) \subset \cdots \subset U_q(\Gl_m)$$
and the corresponding decomposition of the restriction of Weyl modules for $U_q(\Gl_m)$ into the one-dimensional summands as $U_q(\Gl_1)$ modules.
Seemingly separately, extending the work of Carter, Brundan presented un-normalised lowering operators for $U_q(\Gl_n)$~\cite{brundan_1998} in 1998.

In this definition, following~\cite{tolstoy_1990} closely, the elements $e_{j,i}$ of $U_q(\Gl_m)$ for $j > i$ are inductively defined by
\begin{equation}\label{eq:eji}
  e_{j, i} = [e_{j,i+1}, e_{i+i, i}]_q
\end{equation}
where the $e_{i+1,i}$ are the usual generators of $U^-_q(\Gl_m)$ and $[a, b]_q = ab - (-1)^{\deg a \deg b} q^{(a,b)} ba$. Here the exponent of $q$ is the scalar product between roots $a$ and $b$.
\begin{definition}\label{lem:g-t}
  Fix $m$ and let $\lambda$ be a partition of $m$. In the Weyl module of $U_q(\Gl_m)$ of highest weight $\lambda$, let $|\lambda_m\rangle$ be a highest weight vector of unit paring with itself.
  Then, for each semi-standard tableau $T$ of weight $\lambda$, define elements of the Weyl module $|T\rangle$ as
  $$|T\rangle = F_-\left(T^{(1)}; T^{(2)}\right)F_-\left(T^{(2)}; T^{(3)}\right)\cdots F_-\left(T^{(m-1)}; T^{(m)}\right)|\lambda_m\rangle$$
  where 
  \begin{equation}\label{eq:fminus}
    F_-(T^{(j-1)}; T^{(j)}) = \mathcal{N}\left(T^{(j-1)}; T^{(j)}\right)^{-1} P_q(\Gl_{ j-1}) \prod_{i = 1}^{j-1}(e_{j,i})^{T^{(j)}_{i} - T^{(j-1)}_i},
  \end{equation}
  and the constants $\mathcal{N}$ are given by
  \begin{multline} \label{eq:tolstoy_heavy}
    \mathcal{N}\left(T^{(j-1)}; T^{(j)}\right) = 
    \left\{
    \prod_{k = 1}^{j-1} [T^{(j)}_{k}-T^{(j-1)}_{k}]!
    \prod_{1 \le i<k\le j-1}\frac{[T^{(j-1)}_{i}-i - T^{(j-1)}_{k}+k]!}{[T^{(j)}_{i}-i - T^{(j-1)}_{k}+k]!} \times \right. \\
    \left.\prod_{1 \le i < k \le j}\frac{[T^{(j)}_{i}-i - T^{(j)}_{k} + k-1]!}{[T^{(j-1)}_{i}-i - T^{(j)}_{k} + k-1]!}
    \right\}^{\frac 12}.
  \end{multline}
\end{definition}
\begin{lemma}\cite[Thm. 2]{tolstoy_1990}\cite[Eq. 4.11]{quesne_1993} The set $\{|T\rangle\}$ as given in \cref{lem:g-t} is a basis of the Weyl module which is orthonormal with respect to the pairing.

\end{lemma}
Note that these are \textit{not} divided powers $e_{ij}{}^{(T^{(j)}_{i} - T^{(j-1)})}$ but regular generators $e_{ij}{}^{T^{(j)}_{i} - T^{(j-1)}}$.
Additionally note that if $T$ is the unique semi-standard tableau with both weight and shape $\lambda$ then $|T\rangle = |\lambda_n\rangle$.

We've presented \cref{eq:tolstoy_heavy} in form given in~\cite{tolstoy_1990} (but with indices $j$ and $k$ swapped), to aide in reconciling the lemmata, but now use our notation to simplify the form of $\mathcal{N}$.
If $T$ has $n$ rows, then its maximal value, $s$, is at least  $n$.
Let
\begin{equation*}
N_{i,j}^T = \text{number of entries $j$ in row $i$ of $T$} = T^{(j)}_i - T^{(j-1)}_i.
\end{equation*}
We will also write   %referee asks why i\leqslant j.  Wants examples of M,N
\begin{align*}
     M_{i,k}^j &= (N_{k,j+1} + \cdots + N_{k,s}) - (N_{i,j+1} + \cdots + N_{i,s}).
\end{align*}
Note that $M_{i,k}^j = T^{(s)}_k - T^{(j)}_k - T^{(s)}_i + T^{(j)}_i$ and ovserve that it can also be expressed as a difference of two axial distances. Recall from \cref{def:axial} that $c^\lambda_{i,k} = \lambda_i - \lambda_k + i - k$ so  $M_{i,k}^j=c^{T^{(s)}}_{ki} -c^{T^{(j)}}_{ki} $.
\begin{example}
Let
\begin{equation} T = \vcenter{\hbox{\young(11236,2244,345,66)}}\end{equation}
so that $s = 6$.
Then, where omitted entries in the table denote a value of 0 and grey boxes denote entries that can never be non-zero for any $T$, the values of $N^T_{i,j}$ are
\begin{center}
\begin{tabular}{|l|llllll|} 
\hline
\diagbox{i}{j} & 1                                    & 2                                    & 3                                    & 4 & 5 & 6  \\ 
\hline
1              & 2                                    & 1                                    & 1                                    &   &   &  1  \\
2              & {\cellcolor[rgb]{0.871,0.867,0.855}} & 2                                    &                                      & 2 &   &    \\
3              & {\cellcolor[rgb]{0.871,0.867,0.855}} & {\cellcolor[rgb]{0.871,0.867,0.855}} & 1                                    & 1 & 1 &    \\
4              & {\cellcolor[rgb]{0.871,0.867,0.855}} & {\cellcolor[rgb]{0.871,0.867,0.855}} & {\cellcolor[rgb]{0.871,0.867,0.855}} &   &   & 2  \\
\hline
\end{tabular}.
\end{center}
The values of $(N_{i,j+1} + \cdots + N_{i,s})$ are
\begin{center}
\begin{tabular}{|l|llllll|} 
\hline
\diagbox{i}{j} & 0 & 1 & 2 & 3 & 4 & 5  \\ 
\hline
1              & 5 & 3 & 2 & 1 & 1 &  1 \\
2              & 4 & 4 & 2 & 2 &   &    \\
3              & 3 & 3 & 3 & 2 & 1 &    \\
4              & 2 & 2 & 2 & 2 & 2 & 2  \\
\hline
\end{tabular},
\end{center}
and the values of $M^{j}_{i,k}$ are (including $j=0$)
\begin{center}
\begin{tabular}{ccc}
\shortstack[c]{
$M^{0}_{i,k}$\\
\begin{tabular}{|l|llll|} 
\hline
\diagbox{i}{k} & 1 & 2 & 3 & 4 \\ 
\hline
1              & & -1& -2& -3 \\
2              & 1&  & -1& -2\\
3              & 2 & 1 &  & -1\\
4              & 3 & 2 & 1 &  \\
\hline
\end{tabular}\vspace{1em}}&
\shortstack[c]{
$M^{1}_{i,k}$\\
\begin{tabular}{|l|llll|} 
\hline
\diagbox{i}{k} & 1 & 2 & 3 & 4 \\ 
\hline
1              &  & 1& & -1 \\
2              & -1 &  & -1& -2\\
3              &  & 1 &  & -1\\
4              & 1 & 2 & 1 &  \\
\hline
\end{tabular}\vspace{1em}}&
\shortstack[c]{
$M^{2}_{i,k}$\\
\begin{tabular}{|l|llll|} 
\hline
\diagbox{i}{k} & 1 & 2 & 3 & 4 \\ 
\hline
1              &  & & 1&  \\
2              &  &  & 1& \\
3              & -1 & -1 &  & -1\\
4              &  &  & 1 &  \\
\hline
\end{tabular}\vspace{1em}}
\\
\shortstack[c]{
$M^{3}_{i,k}$\\
\begin{tabular}{|l|llll|} 
\hline
\diagbox{i}{k} & 1 & 2 & 3 & 4 \\ 
\hline
1              &  & 1& 1& 1 \\
2              & -1 &  &  & \\
3              & -1 &  &  & \\
4              & -1 &  &  &  \\
\hline
\end{tabular}}&
\shortstack[c]{
$M^{4}_{i,k}$\\
\begin{tabular}{|l|llll|} 
\hline
\diagbox{i}{k} & 1 & 2 & 3 & 4 \\ 
\hline
1              &  & -1 & & 1 \\
2              & 1 &  & 1 & 2\\
3              &  & -1 & &1\\
4              & -1 & -2 & -1 &  \\
\hline
\end{tabular}}&
\shortstack[c]{
$M^{5}_{i,k}$\\
\begin{tabular}{|l|llll|} 
\hline
\diagbox{i}{k} & 1 & 2 & 3 & 4 \\ 
\hline
1              &  & -1 & -1 & 1 \\
2              & 1 &  &  & 2\\
3              & 1 & &  & 2\\
4              & -1 & -2 & -2 &  \\
\hline
\end{tabular}}
\end{tabular}
\end{center}
Notice that $M^j_{i,k} = - M^j_{k,i}$.
% TODO prove?
\end{example}

Now $T^{(j)}_i = \lambda_i - (N^T_{i,j+1} + \cdots +N^T_{i,s})$ and $T^{(j)}_{k} - T^{(j-1)}_k = N_{k,j}$ so that $T^{(j)}_i - i - T^{(j)}_k + k = c_{ik} + M_{i,k}^j$. We can then write
\begin{align}\label{eq:tolstoy_rr}
&    \mathcal{N}\left(T^{(j-1)}; T^{(j)}\right)\\
    =&\nonumber\left\{
    \prod_{k = 1}^{j-1} [N_{k,j}]!
    \prod_{1 \le i<k\le j-1}\frac
    {[c_{i,k}+M_{i,k}^{j-1}]!}
    {[c_{i,k}+M_{i,k}^{j-1}+N_{i,j}]!}
    \prod_{1 \le i < k \le j}\frac
    {[c_{i,k} + M_{i,k}^{j-1} -1-N_{k,j}+N_{i,j}]!}
    {[c_{i,k} + M_{i,k}^{j-1}-1-N_{k,j}]!}
    \right\}^{\frac 12}\\
    =&\nonumber\left\{
    \prod_{k = 1}^{j-1} [N_{k,j}]!
    \prod_{1 \le i<k\le j-1}\frac
    {[c_{i,k}+ M_{i,k}^{j-1}+N_{i,j}-N_{i,j}]!}
    {[c_{i,k}+ M_{i,k}^{j-1}+N_{i,j}]!}
    \prod_{1 \le i < k \le j}\frac
    {[c_{i,k} + M_{i,k}^{j} -1]!}
    {[c_{i,k} + M_{i,k}^{j} -1-N_{i,j}]!}
    \right\}^{\frac 12}\\
    =&\nonumber \prod_{k = 1}^{j-1} [N_{k,j}]! \left\{
    {
      \prod_{1 \le i < k \le j}\qbinom{c_{i,k} + M_{i,k}^{j} - 1 }{N_{i,j}}
    }\Big/{
      \prod_{1 \le i<k < j}\qbinom{c_{i,k} + M_{i,k}^{j-1}+N_{i,j}}{N_{i,j}}
    }
    \right\}^{\frac 12}
  \end{align}
Note that as $c_{i,k} + M^j_{i,k} = c^{T^{(j)}}_{i,j}$ which is at least $1$ when $i<k$, all of the values in the manipulation above are strictly positive.

As such we can rewrite \cref{eq:fminus} as
\begin{equation}
F_-(T^{(j-1)}; T^{(j)}) = \mathcal{N}'(T^{(j-1)};T^{(j)})P_q(\Gl_{ j-1}) \prod_{i = 1}^{j-1}e_{j,i}^{(N_{ij})}
\end{equation}
where
\begin{equation}\label{eq:tolstoy_rr2}
\mathcal{N}'(T^{(j-1)};T^{(j)})=
\left\{
    \frac{
      \prod_{1 \le i < k \le j}\qbinom{c_{i,k} + M_{i,k}^{j} - 1 }{N_{i,j}}
    }{
      \prod_{1 \le i<k < j}\qbinom{c_{i,k} + M_{i,k}^{j-1}+N_{i,j}}{N_{i,j}}
    }
    \right\}^{-\frac 12}
\end{equation}
and
\begin{equation}
  e_{j,i}^{(N_{i,j})} = \frac{(e_{j,i})^{N_{i,j}}}{[N_{i,j}]!}
\end{equation}
We have at last replaced the normal powers with divided powers by absorbing the factors of $[N_{kj}]!$.

It is a routine calculation to show that the $P_q(\Gl_{ j-1})\prod_{i = 1}^{j-1}e_{j\,i}^{(N_{ij})}$ are exactly the pre-images of $(\JW \otimes \id) \circ (\id \otimes E_\mu)$ where $\mu$ is the weight described by the bitstring of 1s and 0s with a 1 in the $i$th position if there is a number $j$ in the $i$-th column of $T$. The key is repeated application of relations \cite[Eq 2.9, Eq 2.13]{cautis_kamnitzer_morrison_2014} and \cite[Lemma 2.9]{elias_2015}. The first states that
\begin{equation}
  \vcenter{\hbox{\begin{tikzpicture}[scale=0.5]
    \draw[very thick] (0, 1.5) -- (4, 1.5);
    \draw[very thick] (0, 0) -- (4, 0);
    \draw[very thick] (1, 1.5) -- (2, 0);
    \draw[very thick] (2, 1.5) -- (3, 0);
    \node at (1.1, 0.5) {\smaller$r$};
    \node at (2.9, 1.0) {\smaller$s$};
  \end{tikzpicture}}} 
  = \qbinom{r + s}{r}\;
  \vcenter{\hbox{\begin{tikzpicture}[scale=0.5]
    \draw[very thick] (0, 1.5) -- (3, 1.5);
    \draw[very thick] (0, 0) -- (3, 0);
    \draw[very thick] (1, 1.5) -- (2, 0);
    \node at (2.3, 1.0) {\smaller$r+s$};
  \end{tikzpicture}}}.
\end{equation}
A rung with label $r$ on the $i$-th string here is the image of the divided power $(e_{i+1,i})^{(r)} = \frac{(e_{i+1,i})^{r}}{[r]!}$ from which this follows immediately.
The second relation states that
\begin{equation}
  \vcenter{\hbox{\begin{tikzpicture}[scale=0.5]
    \draw[very thick] (0.5, 3) -- (4, 3);
    \draw[very thick] (0.5, 1.5) -- (4, 1.5);
    \draw[very thick] (0.5, 0) -- (4, 0);
    \draw[very thick] (2.5, 1.5) -- (3.5, 0);
    \draw[very thick] (1, 3) -- (2, 1.5);
    \draw[very thick] (2, 3) -- (3, 1.5);
    \node at (1.2, 2) {\tiny$1$};
    \node at (2.6, 0.5) {\tiny$1$};
    \node at (2.8, 2.5) {\tiny$1$};
  \end{tikzpicture}}} 
  =
  \vcenter{\hbox{\begin{tikzpicture}[scale=0.5]
    \draw[very thick] (0.5, 3) -- (4, 3);
    \draw[very thick] (0.5, 1.5) -- (4, 1.5);
    \draw[very thick] (0.5, 0) -- (4, 0);
    \draw[very thick] (2.5, 1.5) -- (3.5, 0);
    \draw[very thick] (1, 3) -- (2, 1.5);
    \node at (1.2, 2) {\tiny$2$};
    \node at (2.5, 0.5) {\tiny$1$};
  \end{tikzpicture}}} + 
  \vcenter{\hbox{\begin{tikzpicture}[scale=0.5]
    \draw[very thick] (0.5, 3) -- (4, 3);
    \draw[very thick] (0.5, 1.5) -- (4, 1.5);
    \draw[very thick] (0.5, 0) -- (4, 0);
    \draw[very thick] (2.5, 1.5) -- (3.5, 0);
    \draw[very thick] (2, 3) -- (3, 1.5);
    \node at (2.5, 0.5) {\tiny$1$};
    \node at (2.8, 2.5) {\tiny$2$};
  \end{tikzpicture}}}
  .
\end{equation}
The final relation claims
\begin{equation}\label{eq:run_swapping}
  \vcenter{\hbox{\begin{tikzpicture}[scale=0.5]
    \draw[very thick] (0.5, 3) -- (4, 3);
    \draw[very thick] (0.5, 1.5) -- (4, 1.5);
    \draw[very thick] (0.5, 0) -- (4, 0);
    \draw[very thick] (1, 3) -- (2, 1.5);
    \draw[very thick] (1.5, 1.5) -- (2.5, 0);
    \draw[very thick] (2.5, 1.5) -- (3.5, 0);
    \node at (2.2, 2.25) {\smaller$s$};
    \node at (1.7, 0.5) {\smaller$r$};
    \node at (3.3, 1) {\smaller$t$};
  \end{tikzpicture}}} 
  = \sum_{0 \le m \le \min\{s, t\}} \qbinom{r + t - s}{t-m}\;
  \vcenter{\hbox{\begin{tikzpicture}[scale=0.5]
    \draw[very thick] (0.5, 3) -- (4, 3);
    \draw[very thick] (0.5, 1.5) -- (4, 1.5);
    \draw[very thick] (0.5, 0) -- (4, 0);
    \draw[very thick] (2.5, 1.5) -- (3.5, 0);
    \draw[very thick] (1, 3) -- (2, 1.5);
    \draw[very thick] (2, 3) -- (3, 1.5);
    \node at (1.2, 2) {\smaller$m$};
    \node at (2.2, 0.5) {\smaller$r+t$};
    \node at (3.4, 2.5) {\smaller$s-m$};
  \end{tikzpicture}}}  .
\end{equation}
When pre-composed by the extremal projector, \cref{eq:run_swapping} becomes
\begin{equation}\label{eq:combinfork}
  \vcenter{\hbox{\begin{tikzpicture}[scale=0.5]
    \draw[very thick, fill=purple!30!white] (0,1) rectangle (0.5,3.5);
    \draw[very thick] (0.5, 3) -- (4, 3);
    \draw[very thick] (-.5, 3) -- (0, 3);
    \draw[very thick] (0.5, 1.5) -- (4, 1.5);
    \draw[very thick] (-.5, 1.5) -- (0, 1.5);
    \draw[very thick] (-.5, 0) -- (4, 0);
    \draw[very thick] (1, 3) -- (2, 1.5);
    \draw[very thick] (1.5, 1.5) -- (2.5, 0);
    \draw[very thick] (2.5, 1.5) -- (3.5, 0);
    \node at (2, 2.5) {\smaller$s$};
    \node at (1.7, 0.5) {\smaller$r$};
    \node at (3.3, 1) {\smaller$t$};
  \end{tikzpicture}}} 
  = \qbinom{r+t-s}{t}
  \vcenter{\hbox{\begin{tikzpicture}[scale=0.5]
    \draw[very thick, fill=purple!30!white] (0,1) rectangle (0.5,3.5);
    \draw[very thick] (0.5, 3) -- (3, 3);
    \draw[very thick] (-.5, 3) -- (0, 3);
    \draw[very thick] (0.5, 1.5) -- (3, 1.5);
    \draw[very thick] (-.5, 1.5) -- (0, 1.5);
    \draw[very thick] (-.5, 0) -- (3, 0);
    \draw[very thick] (1, 3) -- (2, 1.5);
    \draw[very thick] (1.5, 1.5) -- (2.5, 0);
    \node at (2, 2.5) {\smaller$s$};
    \node at (1.3, 0.5) {\smaller$r + t$};
  \end{tikzpicture}}}   .
\end{equation}
We can also see that if $s > t$ then
\begin{equation}\label{eq:no_disorder}
  \vcenter{\hbox{\begin{tikzpicture}[scale=0.5]
    \draw[very thick, fill=purple!30!white] (0,1) rectangle (0.5,3.5);
    \draw[very thick] (0.5, 3) -- (3, 3);
    \draw[very thick] (-.5, 3) -- (0, 3);
    \draw[very thick] (0.5, 1.5) -- (3, 1.5);
    \draw[very thick] (-.5, 1.5) -- (0, 1.5);
    \draw[very thick] (-.5, 0) -- (3, 0);
    \draw[very thick] (1, 3) -- (2, 1.5);
    \draw[very thick] (1.5, 1.5) -- (2.5, 0);
    \node at (2, 2.5) {\smaller$s$};
    \node at (1.7, 0.5) {\smaller$t$};
  \end{tikzpicture}}} 
  = 0,
\end{equation}
as it is the only term of \cref{eq:run_swapping} with $r=0$ left, after precomposition by the extemal projector.

From these relations, one can deduce, for example, that
\begin{equation}\label{eq:first_rung}
  P_q(\Gl_{j-1})e_{j\,i}^{(N_{ij})} =
     \vcenter{\hbox{\begin{tikzpicture}[scale=0.5]
    \draw[very thick, fill=purple!30!white] (0,-.5) rectangle (0.5,6.5);
    \draw[very thick] (0.5, 6) -- (3, 6);\draw[very thick] (-0.5, 6) -- (0, 6);
    \node at (1.5, 5.4) {\tiny$\vdots$};\node at (-0.2, 5.4) {\tiny$\vdots$};
    \draw[very thick] (0.5, 4.5) -- (3, 4.5);\draw[very thick] (-0.5, 4.5) -- (0, 4.5);
    \draw[very thick] (0.5, 3.5) -- (3, 3.5);\draw[very thick] (-0.5, 3.5) -- (0, 3.5);
    \draw[very thick] (0.5, 2.5) -- (3, 2.5);\draw[very thick] (-0.5, 2.5) -- (0, 2.5);
    \node at (1.5, 1.9) {\tiny$\vdots$};\node at (-0.2, 1.9) {\tiny$\vdots$};
    \draw[very thick] (0.5, 1) -- (3, 1);\draw[very thick] (-0.5, 1) -- (0, 1);
    \draw[very thick] (0.5, 0) -- (3, 0);\draw[very thick] (-0.5, 0) -- (0, 0);
    \draw[very thick] (-0.5, -1) -- (3, -1);

    \draw[very thick] (1, 3.5) -- (1.5, 2.5);
    \draw[very thick, path fading=south] (1, 2.5) -- (1.25, 2.0);
    \draw[very thick, path fading=north] (1, 1.5) -- (1.25, 1);
    \draw[very thick] (1, 1) -- (1.5, 0);
    \draw[very thick] (1, 0) -- (1.5, -1);

    \node at (2, 3) {\tiny$N_{ij}$};
    \node at (2, 0.5) {\tiny$N_{ij}$};
    \node at (2, -.5) {\tiny$N_{ij}$};
  \end{tikzpicture}}}.
\end{equation}
Here there are rungs between successive strands from the $i$-th to the $j$-th inclusive (the clasp is on strands $1$ through to $j-1$). We demonstrate this by example in \cref{eg:pq}.

The following example also demonstrates that repeated application of the above relations yield that the form of $P_q(\Gl_{ j-1})\prod_{i = 1}^{j-1}e_{j\,i}^{(N_{ij})}$ must be that of \cref{eq:e_mu} with a clasp. All that remains is the accounting to ensure the rungs have the correct multiplicities and that the combinatorial factors cancel. 

\begin{example}\label{eg:pq}
    Suppose that
    \begin{equation}
        T = \vcenter{\hbox{\young(1133466<10><12>,2344578<12>,4455699,55667<11>,66778<11>,77889<12>,8999<12>,<10><10><11><11><12>)}}
    \end{equation}
    and $j=9$. We will derive $P_q(\Gl_8) \prod_{i = 1}^{8}(e_{9\,i})^{(N_{i,9})}$, part of $F_-(T^{(8)};T^{(9)})$. The only $i$ for which $N_{9,i}$ are non-zero are those that index rows of $T$ with $9$ in them:
    \begin{align}
        N^T_{3,9} = T^{(9)}_3 - T^{(8)}_3 &=2\\\nonumber
        N^T_{6,9} = T^{(9)}_6 - T^{(8)}_6 &=1\\\nonumber
        N^T_{7,9} = T^{(9)}_7 - T^{(8)}_7 &=3.
    \end{align}
    The remainder are zero. Thus we are interested in
    \begin{equation}
        P_q(\Gl_8)(e_{9\,3})^{(2)}e_{9\,6}(e_{9\,7})^{(3)}.
    \end{equation}

    Now, $e_{9\,3} = [e_{9\,4}, e_{4\,3}]_q = e_{9\,4}e_{4\,3} + q^{-1}e_{4\,3}e_{9\,4}$ which translates into diagrams as
    \begin{align}
  P_q(\Gl_{8})e_{9\,3} &=
     \vcenter{\hbox{\begin{tikzpicture}[scale=0.5]
    \foreach \i in {0,...,8}{
      \draw[very thick] (-0.5,0.5*\i) -- (3.5,0.5*\i);
    }
    \draw[very thick, fill=purple!30!white] (0,.25) rectangle (0.5,4.25);
    \draw[very thick, fill=white] (1,-.25) rectangle (2.5,2.75);
    \node at (1.75,1.25) {$e_{9\,4}$};
    \draw[very thick] (2.75,3) -- (3,2.5);
  \end{tikzpicture}}}
  + q^{-1}\vcenter{\hbox{\begin{tikzpicture}[scale=0.5]
    \foreach \i in {0,...,8}{
      \draw[very thick] (-0.5,0.5*\i) -- (3.5,0.5*\i);
    }
    \draw[very thick, fill=purple!30!white] (0,.25) rectangle (0.5,4.25);
    \draw[very thick] (0.75,3) -- (1,2.5);
    \draw[very thick, fill=white] (1.5,-.25) rectangle (3,2.75);
    \node at (2.25,1.25) {$e_{9\,4}$};
  \end{tikzpicture}}} \\\nonumber&=
  \vcenter{\hbox{\begin{tikzpicture}[scale=0.5]
    \foreach \i in {0,...,8}{
      \draw[very thick] (-0.5,0.5*\i) -- (3.5,0.5*\i);
    }
    \draw[very thick, fill=purple!30!white] (0,.25) rectangle (0.5,4.25);
    \draw[very thick, fill=white] (1,-.25) rectangle (2.5,2.75);
    \node at (1.75,1.25) {$e_{9\,4}$};
    \draw[very thick] (2.75,3) -- (3,2.5);
  \end{tikzpicture}}}
  .
    \end{align}
    from which we can induct to see that
    \begin{equation}
  P_q(\Gl_{8})e_{9\,3} =
     \vcenter{\hbox{\begin{tikzpicture}[scale=0.5]
    \foreach \i in {0,...,8}{
      \draw[very thick] (-0.5,0.5*\i) -- (1.25,0.5*\i);
    }
    \draw[very thick, fill=purple!30!white] (0,.25) rectangle (0.5,4.25);
    \foreach \i in {0,...,5}{
      \draw[very thick] (0.75,0.5*\i+0.5) -- (1,0.5*\i);
    }
  \end{tikzpicture}}}
  \end{equation}
  in agreement with \cref{eq:first_rung}. Now, note that, for example,
  \begin{equation}
      \vcenter{\hbox{\begin{tikzpicture}[scale=0.5]
    \foreach \i in {0,...,8}{
      \draw[very thick] (-0.5,0.5*\i) -- (1.5,0.5*\i);
    }
    \draw[very thick, fill=purple!30!white] (0,.25) rectangle (0.5,4.25);
    \foreach \i in {0,...,5}{
      \draw[very thick] (0.75,0.5*\i+0.5) -- (1,0.5*\i);
    }
    \draw[very thick] (1.125,0.5*2+0.5) -- (1.375,0.5*2);
  \end{tikzpicture}}}=
  \vcenter{\hbox{\begin{tikzpicture}[scale=0.5]
    \foreach \i in {0,...,8}{
      \draw[very thick] (-0.5,0.5*\i) -- (1.5,0.5*\i);
    }
    \draw[very thick, fill=purple!30!white] (0,.25) rectangle (0.5,4.25);
    \foreach \i in {0,...,5}{
      \draw[very thick] (0.75,0.5*\i+0.5) -- (1,0.5*\i);
    }
    \node at (1.25,1.25) {\tiny$2$};
  \end{tikzpicture}}}
  +
  \vcenter{\hbox{\begin{tikzpicture}[scale=0.5]
    \foreach \i in {0,...,8}{
      \draw[very thick] (-0.5,0.5*\i) -- (1.5,0.5*\i);
    }
    \draw[very thick, fill=purple!30!white] (0,.25) rectangle (0.5,4.25);
    \foreach \i in {0,1,3,4,5}{
      \draw[very thick] (0.75,0.5*\i+0.5) -- (1,0.5*\i);
    }
    \draw[very thick] (1.125,0.5*2+0.5) -- (1.375,0.5*2);
    \node at (1.675,1.25) {\tiny$2$};
  \end{tikzpicture}}}=
  \vcenter{\hbox{\begin{tikzpicture}[scale=0.5]
    \foreach \i in {0,...,8}{
      \draw[very thick] (-0.5,0.5*\i) -- (1.5,0.5*\i);
    }
    \draw[very thick, fill=purple!30!white] (0,.25) rectangle (0.5,4.25);
    \foreach \i in {0,2,3,4,5}{
      \draw[very thick] (0.75,0.5*\i+0.5) -- (1,0.5*\i);
    }
    \draw[very thick] (1.125,0.5*1+0.5) -- (1.375,0.5*1);
    \node at (1.175,1.25) {\tiny$2$};
  \end{tikzpicture}}} - 
  \vcenter{\hbox{\begin{tikzpicture}[scale=0.5]
    \foreach \i in {0,...,8}{
      \draw[very thick] (-0.5,0.5*\i) -- (1.5,0.5*\i);
    }
    \draw[very thick, fill=purple!30!white] (0,.25) rectangle (0.5,4.25);
    \foreach \i in {0,...,5}{
      \draw[very thick] (1.125,0.5*\i+0.5) -- (1.375,0.5*\i);
    }
    \draw[very thick] (0.75,0.5*2+0.5) -- (1.,0.5*2);
  \end{tikzpicture}}}=0.
  \end{equation}
  Similarly, using \cref{eq:combinfork}, 
  \begin{equation}
  \vcenter{\hbox{\begin{tikzpicture}[scale=0.5]
    \foreach \i in {0,...,8}{
      \draw[very thick] (-0.5,0.5*\i) -- (1.5,0.5*\i);
    }
    \draw[very thick, fill=purple!30!white] (0,.25) rectangle (0.5,4.25);
    \foreach \i in {0,...,5}{
      \draw[very thick] (0.75,0.5*\i+0.5) -- (1,0.5*\i);
    }
    \draw[very thick] (1.125,0.5*4+0.5) -- (1.375,0.5*4);
    \node at (1.175,0.25) {\tiny$2$};
    \node at (1.175,0.75) {\tiny$2$};
    \node at (1.175,1.25) {\tiny$2$};
  \end{tikzpicture}}}=0\quad\quad\text{and}\quad\quad
  \vcenter{\hbox{\begin{tikzpicture}[scale=0.5]
    \foreach \i in {0,...,8}{
      \draw[very thick] (-0.5,0.5*\i) -- (1.5,0.5*\i);
    }
    \draw[very thick, fill=purple!30!white] (0,.25) rectangle (0.5,4.25);
    \foreach \i in {0,...,5}{
      \draw[very thick] (0.75,0.5*\i+0.5) -- (1,0.5*\i);
    }
    \draw[very thick] (1.125,0.5*0+0.5) -- (1.375,0.0*4);
  \end{tikzpicture}}}=\vcenter{\hbox{\begin{tikzpicture}[scale=0.5]
    \foreach \i in {0,...,8}{
      \draw[very thick] (-0.5,0.5*\i) -- (1.5,0.5*\i);
    }
    \draw[very thick, fill=purple!30!white] (0,.25) rectangle (0.5,4.25);
    \foreach \i in {0,...,5}{
      \draw[very thick] (0.75,0.5*\i+0.5) -- (1,0.5*\i);
    }
    \node at (1.175,0.25) {\tiny$2$};
  \end{tikzpicture}}}\quad\quad\text{and}\quad\quad
  \vcenter{\hbox{\begin{tikzpicture}[scale=0.5]
    \foreach \i in {0,...,8}{
      \draw[very thick] (-0.5,0.5*\i) -- (1.5,0.5*\i);
    }
    \draw[very thick, fill=purple!30!white] (0,.25) rectangle (0.5,4.25);
    \foreach \i in {0,...,5}{
      \draw[very thick] (0.75,0.5*\i+0.5) -- (1,0.5*\i);
    }
    \draw[very thick] (1.125,0.5*2+0.5) -- (1.375,0.5*2);
    \node at (1.175,0.25) {\tiny$2$};
    \node at (1.175,0.75) {\tiny$2$};
  \end{tikzpicture}}}=
  \vcenter{\hbox{\begin{tikzpicture}[scale=0.5]
    \foreach \i in {0,...,8}{
      \draw[very thick] (-0.5,0.5*\i) -- (1.5,0.5*\i);
    }
    \draw[very thick, fill=purple!30!white] (0,.25) rectangle (0.5,4.25);
    \foreach \i in {0,...,5}{
      \draw[very thick] (0.75,0.5*\i+0.5) -- (1,0.5*\i);
    }
    \node at (1.175,0.25) {\tiny$2$};
    \node at (1.175,0.75) {\tiny$2$};
    \node at (1.175,1.25) {\tiny$2$};
  \end{tikzpicture}}}.
  \end{equation}
  As such, the only term of $e_{9\,3} = [e_{9\,8},e_{8\,3}]_q$ that survives precomposition by $P_q(\Gl_{8})e_{9\,3}$ is $e_{9\,8}e_{8\,7}e_{7\,6}e_{6\,5}e_{5\,4}e_{4\,3}$.
  Finally,
  \begin{equation}
      \vcenter{\hbox{\begin{tikzpicture}[scale=0.5]
    \foreach \i in {0,...,8}{
      \draw[very thick] (-0.5,0.5*\i) -- (1.5,0.5*\i);
    }
    \draw[very thick, fill=purple!30!white] (0,.25) rectangle (0.5,4.25);
    \foreach \i in {0,...,5}{
      \draw[very thick] (0.75,0.5*\i+0.5) -- (1,0.5*\i);
    }
    \foreach \i in {0,...,4}{
      \node at (1.175,0.5*\i + 0.25) {\tiny$2$};
    }
    \draw[very thick] (1.125,0.5*5+0.5) -- (1.375,0.5*5);
  \end{tikzpicture}}}=[2]
  \vcenter{\hbox{\begin{tikzpicture}[scale=0.5]
    \foreach \i in {0,...,8}{
      \draw[very thick] (-0.5,0.5*\i) -- (1.5,0.5*\i);
    }
    \draw[very thick, fill=purple!30!white] (0,.25) rectangle (0.5,4.25);
    \foreach \i in {0,...,5}{
      \draw[very thick] (0.75,0.5*\i+0.5) -- (1,0.5*\i);
    }
    \foreach \i in {0,...,5}{
      \node at (1.175,0.5*\i + 0.25) {\tiny$2$};
    }
  \end{tikzpicture}}}.
  \end{equation}
  Thus
  \begin{equation}
  P_q(\Gl_{8})e_{9\,3}^2 = [2]
     \vcenter{\hbox{\begin{tikzpicture}[scale=0.5]
    \foreach \i in {0,...,8}{
      \draw[very thick] (-0.5,0.5*\i) -- (1.5,0.5*\i);
    }
    \draw[very thick, fill=purple!30!white] (0,.25) rectangle (0.5,4.25);
    \foreach \i in {0,...,5}{
      \draw[very thick] (0.75,0.5*\i+0.5) -- (1,0.5*\i);
    }
    \foreach \i in {0,...,5}{
      \node at (1.175,0.5*\i + 0.25) {\tiny$2$};
    }
  \end{tikzpicture}}},\quad\quad\text{or equivalently}\quad\quad
  P_q(\Gl_{8})e_{9\,3}^{(2)} = 
     \vcenter{\hbox{\begin{tikzpicture}[scale=0.5]
    \foreach \i in {0,...,8}{
      \draw[very thick] (-0.5,0.5*\i) -- (1.5,0.5*\i);
    }
    \draw[very thick, fill=purple!30!white] (0,.25) rectangle (0.5,4.25);
    \foreach \i in {0,...,5}{
      \draw[very thick] (0.75,0.5*\i+0.5) -- (1,0.5*\i);
    }
    \foreach \i in {0,...,5}{
      \node at (1.175,0.5*\i + 0.25) {\tiny$2$};
    }
  \end{tikzpicture}}},
  \end{equation}  
  recovering \cref{eq:first_rung} again. Using similar arguments, we see that the only term of $e_{9,6}$ that is left after precomposing by $P_q(\Gl_{8})e_{9\,3}^{(2)}$ is $e_{9,8}e_{8,7}e_{7,6}$. By the same reasoning then,
  \begin{equation}
  P_q(\Gl_{8})e_{9\,3}^{(2)}e_{9\,6} = 
     \vcenter{\hbox{\begin{tikzpicture}[scale=0.5]
    \foreach \i in {0,...,8}{
      \draw[very thick] (-0.5,0.5*\i) -- (1.5,0.5*\i);
    }
    \draw[very thick, fill=purple!30!white] (0,.25) rectangle (0.5,4.25);
    \foreach \i in {0,...,5}{
      \draw[very thick] (0.75,0.5*\i+0.5) -- (1,0.5*\i);
    }
    \foreach \i in {0,...,2}{
      \node at (1.175,0.5*\i + 0.25) {\tiny$3$};
    }
    \foreach \i in {3,...,5}{
      \node at (1.175,0.5*\i + 0.25) {\tiny$2$};
    }
  \end{tikzpicture}}}.
  \end{equation}
  Finally, again
  \begin{equation}
    \vcenter{\hbox{\begin{tikzpicture}[scale=0.5]
    \foreach \i in {0,...,8}{
      \draw[very thick] (-0.5,0.5*\i) -- (2,0.5*\i);
    }
    \draw[very thick, fill=purple!30!white] (0,.25) rectangle (0.5,4.25);
    \foreach \i in {0,...,5}{
      \draw[very thick] (0.75,0.5*\i+0.5) -- (1,0.5*\i);
    }
    \foreach \i in {0,...,2}{
      \node at (1.175,0.5*\i + 0.25) {\tiny$3$};
    }
    \foreach \i in {3,...,5}{
      \node at (1.175,0.5*\i + 0.25) {\tiny$2$};
    }
    \draw[very thick] (1.5,0.5*1+0.5) -- (1.75,0.5*1);
  \end{tikzpicture}}}=
  \vcenter{\hbox{\begin{tikzpicture}[scale=0.5]
    \foreach \i in {0,...,8}{
      \draw[very thick] (-0.5,0.5*\i) -- (2,0.5*\i);
    }
    \draw[very thick, fill=purple!30!white] (0,.25) rectangle (0.5,4.25);
    \foreach \i in {0,...,5}{
      \draw[very thick] (0.75,0.5*\i+0.5) -- (1,0.5*\i);
    }
    \foreach \i in {0,2}{
      \node at (1.175,0.5*\i + 0.25) {\tiny$3$};
    }
    \foreach \i in {1}{
      \node at (1.175,0.5*\i + 0.25) {\tiny$4$};
    }
    \foreach \i in {3,...,5}{
      \node at (1.175,0.5*\i + 0.25) {\tiny$2$};
    }
  \end{tikzpicture}}} = 0,
  \end{equation}
  by applying \cref{eq:combinfork} and \cref{eq:no_disorder} and so
  \begin{equation}
      P_q(\Gl_{8})e_{9\,3}^{(2)}e_{9\,6}e_{9\,7}^3 = 
      [2][3]\vcenter{\hbox{\begin{tikzpicture}[scale=0.5]
    \foreach \i in {0,...,8}{
      \draw[very thick] (-0.5,0.5*\i) -- (1.5,0.5*\i);
    }
    \draw[very thick, fill=purple!30!white] (0,.25) rectangle (0.5,4.25);
    \foreach \i in {0,...,5}{
      \draw[very thick] (0.75,0.5*\i+0.5) -- (1,0.5*\i);
    }
    \foreach \i in {0,...,1}{
      \node at (1.175,0.5*\i + 0.25) {\tiny$6$};
    }
    \foreach \i in {2}{
      \node at (1.175,0.5*\i + 0.25) {\tiny$3$};
    }
    \foreach \i in {3,...,5}{
      \node at (1.175,0.5*\i + 0.25) {\tiny$2$};
    }
  \end{tikzpicture}}}
  \end{equation}
  Here, the factors of $[2]$ and $[3]$ come from the factor in  \cref{eq:combinfork} as
  \begin{equation}
    \vcenter{\hbox{\begin{tikzpicture}[scale=0.5]
    \foreach \i in {0,...,8}{
      \draw[very thick] (-0.5,0.5*\i) -- (2,0.5*\i);
    }
    \draw[very thick, fill=purple!30!white] (0,.25) rectangle (0.5,4.25);
    \foreach \i in {0,...,5}{
      \draw[very thick] (0.75,0.5*\i+0.5) -- (1,0.5*\i);
    }
    \foreach \i in {0}{
      \node at (1.175,0.5*\i + 0.25) {\tiny$4$};
    }
    \foreach \i in {1,2}{
      \node at (1.175,0.5*\i + 0.25) {\tiny$3$};
    }
    \foreach \i in {3,...,5}{
      \node at (1.175,0.5*\i + 0.25) {\tiny$2$};
    }
    \draw[very thick] (1.5,0.5*1+0.5) -- (1.75,0.5*1);
  \end{tikzpicture}}}=
  \vcenter{\hbox{\begin{tikzpicture}[scale=0.5]
    \foreach \i in {0,...,8}{
      \draw[very thick] (-0.5,0.5*\i) -- (2,0.5*\i);
    }
    \draw[very thick, fill=purple!30!white] (0,.25) rectangle (0.5,4.25);
    \foreach \i in {0,...,5}{
      \draw[very thick] (0.75,0.5*\i+0.5) -- (1,0.5*\i);
    }
    \foreach \i in {0,...,1}{
      \node at (1.175,0.5*\i + 0.25) {\tiny$4$};
    }
    \foreach \i in {2}{
      \node at (1.175,0.5*\i + 0.25) {\tiny$3$};
    }
    \foreach \i in {3,...,5}{
      \node at (1.175,0.5*\i + 0.25) {\tiny$2$};
    }
  \end{tikzpicture}}}\quad\;\text{and}\quad\;
  \vcenter{\hbox{\begin{tikzpicture}[scale=0.5]
    \foreach \i in {0,...,8}{
      \draw[very thick] (-0.5,0.5*\i) -- (2,0.5*\i);
    }
    \draw[very thick, fill=purple!30!white] (0,.25) rectangle (0.5,4.25);
    \foreach \i in {0,...,5}{
      \draw[very thick] (0.75,0.5*\i+0.5) -- (1,0.5*\i);
    }
    \foreach \i in {0}{
      \node at (1.175,0.5*\i + 0.25) {\tiny$5$};
    }
    \foreach \i in {1}{
      \node at (1.175,0.5*\i + 0.25) {\tiny$4$};
    }
    \foreach \i in {2}{
      \node at (1.175,0.5*\i + 0.25) {\tiny$3$};
    }
    \foreach \i in {3,...,5}{
      \node at (1.175,0.5*\i + 0.25) {\tiny$2$};
    }
    \draw[very thick] (1.5,0.5*1+0.5) -- (1.75,0.5*1);
  \end{tikzpicture}}}=[2]
  \vcenter{\hbox{\begin{tikzpicture}[scale=0.5]
    \foreach \i in {0,...,8}{
      \draw[very thick] (-0.5,0.5*\i) -- (2,0.5*\i);
    }
    \draw[very thick, fill=purple!30!white] (0,.25) rectangle (0.5,4.25);
    \foreach \i in {0,...,5}{
      \draw[very thick] (0.75,0.5*\i+0.5) -- (1,0.5*\i);
    }
    \foreach \i in {0,...,1}{
      \node at (1.175,0.5*\i + 0.25) {\tiny$5$};
    }
    \foreach \i in {2}{
      \node at (1.175,0.5*\i + 0.25) {\tiny$3$};
    }
    \foreach \i in {3,...,5}{
      \node at (1.175,0.5*\i + 0.25) {\tiny$2$};
    }
  \end{tikzpicture}}}\quad\;\text{and}\quad\;
  \vcenter{\hbox{\begin{tikzpicture}[scale=0.5]
    \foreach \i in {0,...,8}{
      \draw[very thick] (-0.5,0.5*\i) -- (2,0.5*\i);
    }
    \draw[very thick, fill=purple!30!white] (0,.25) rectangle (0.5,4.25);
    \foreach \i in {0,...,5}{
      \draw[very thick] (0.75,0.5*\i+0.5) -- (1,0.5*\i);
    }
    \foreach \i in {0}{
      \node at (1.175,0.5*\i + 0.25) {\tiny$6$};
    }
    \foreach \i in {1}{
      \node at (1.175,0.5*\i + 0.25) {\tiny$5$};
    }
    \foreach \i in {2}{
      \node at (1.175,0.5*\i + 0.25) {\tiny$3$};
    }
    \foreach \i in {3,...,5}{
      \node at (1.175,0.5*\i + 0.25) {\tiny$2$};
    }
    \draw[very thick] (1.5,0.5*1+0.5) -- (1.75,0.5*1);
  \end{tikzpicture}}}=[3]
  \vcenter{\hbox{\begin{tikzpicture}[scale=0.5]
    \foreach \i in {0,...,8}{
      \draw[very thick] (-0.5,0.5*\i) -- (1.5,0.5*\i);
    }
    \draw[very thick, fill=purple!30!white] (0,.25) rectangle (0.5,4.25);
    \foreach \i in {0,...,5}{
      \draw[very thick] (0.75,0.5*\i+0.5) -- (1,0.5*\i);
    }
    \foreach \i in {0,...,1}{
      \node at (1.175,0.5*\i + 0.25) {\tiny$6$};
    }
    \foreach \i in {2}{
      \node at (1.175,0.5*\i + 0.25) {\tiny$3$};
    }
    \foreach \i in {3,...,5}{
      \node at (1.175,0.5*\i + 0.25) {\tiny$2$};
    }
  \end{tikzpicture}}}.
  \end{equation}
  Removing the $[3]!$ we see that
  \begin{equation}
      P_q(\Gl_{8})e_{9\,3}^{(2)}e_{9\,6}e_{9\,7}^{(3)} = 
      \vcenter{\hbox{\begin{tikzpicture}[scale=0.5]
    \foreach \i in {0,...,8}{
      \draw[very thick] (-0.5,0.5*\i) -- (1.5,0.5*\i);
    }
    \draw[very thick, fill=purple!30!white] (0,.25) rectangle (0.5,4.25);
    \foreach \i in {0,...,5}{
      \draw[very thick] (0.75,0.5*\i+0.5) -- (1,0.5*\i);
    }
    \foreach \i in {0,...,1}{
      \node at (1.175,0.5*\i + 0.25) {\tiny$6$};
    }
    \foreach \i in {2}{
      \node at (1.175,0.5*\i + 0.25) {\tiny$3$};
    }
    \foreach \i in {3,...,5}{
      \node at (1.175,0.5*\i + 0.25) {\tiny$2$};
    }
  \end{tikzpicture}}},
  \end{equation}
  as expected.
\end{example}

\subsection{Equivalence of Elias' and Normalising Constants}

Let $T'= T^{(s-1)}$ be the semistandard tableau which is the parent of $T$. That is, it contains the same boxes except those labelled $s$. These are the ``new'' boxes in $T$. Let $T'$ have shape $\lambda'\subset \lambda$ and the numbers $c'_{ij}$ and $N'_{ij}$ have their definition with respect to the primed variables.

Now, recall Elias' conjecture concerned the shape of successive \textit{column} semi-standard tableaux from \cref{conj:elias2}.  To bring this in line with the above calculation we will need to take transposes which we denote $\lambda^\perp$.

\begin{lemma}\label{lem:long}
  With the above notation,
  $\mathcal N'(T'; T)^2 = \kappa_{\lambda'^\perp}^{\lambda^\perp}$.
\end{lemma}
\begin{proof}
Elias' conjectural form for $\kappa_{\lambda'^\perp}^{\lambda^\perp}$ is given as
\begin{align}\label{eq:big_elias}
\kappa_{\lambda'^\perp}^{\lambda^\perp}
                          &=\prod_{\substack{1 \le a \\\text{no new box in}\\ \text{column $a$ of $\lambda$}}}\quad 
                            \prod_{\substack{a < b \\ \text{a new box in}\\ \text{column $b$ of $\lambda$}}} \;
                            \qfrac{\lambda'^\perp_a - \lambda'^\perp_b + b - a}{\lambda'^\perp_a - \lambda'^\perp_b + b - a-1}\nonumber\\
\intertext{Now, by the fact that column $a$ has no new boxes, $\lambda'^\perp_a = \lambda^\perp_a$. If we now group the columns labelled by $b$ by the row they appear in, which we will index by $i$, we get}
                          &=\prod_{\substack{1 \le i\le s\\ N_{is} \neq 0}}\quad
                            \prod_{\substack{1 \le a \le \lambda_i \\\text{no new box}\\ \text{in column $a$}}}\quad 
\prod_{b = \lambda_i - N_{is} + 1}^{\lambda_i}\quad 
                            \qfrac{\lambda^\perp_a - \lambda'^\perp_b + b - a}{\lambda^\perp_a - \lambda'^\perp_b + b - a-1}\nonumber\\
\intertext{Note that $N_{is} = 0$ for all $i > s$, hence the bounds on the first product. We have also swapped the order of product here.  Since the box indexed by $b$ is always in the $i$th row of $\lambda$, we have $\lambda'^\perp_b = i-1$, so}
                          &=\prod_{\substack{1 \le i\le s\\ N_{is} \neq 0}}\quad
                            \prod_{\substack{1 \le a \le \lambda_i \\\text{no new box}\\ \text{in column $a$}}}\quad 
                            \prod_{b = \lambda_i - N_{is} + 1}^{\lambda_i}\quad 
                            \qfrac{\lambda^\perp_a - i + b - a+1}{\lambda^\perp_a - i + b - a}\nonumber\\
\intertext{and we can collapse the telescoping product. Thus}
                          &=\prod_{\substack{1 \le i\le s\\ N_{is} \neq 0}}\quad
                            \prod_{\substack{1 \le a \le \lambda_i \\\text{no new box}\\ \text{in column $a$}}}\quad 
                            \qfrac{\lambda^\perp_a - i + \lambda_i - a + 1}{\lambda^\perp_a - i + \lambda_i - a + 1 - N_{is}}  .  \nonumber\\
\intertext{Now we repeat the row/column index swapping trick, by grouping the columns $a$ by the rows they appear in. For notational purposes, we will let $k$ be the row \textit{below} the last box in column $a$ so that $\lambda^\perp_a = k-1$.}
                          &=\prod_{\substack{1 \le i\le s\\ N_{is} \neq 0}}\quad
                            \prod_{i < k \le s}\quad
                            \prod_{a = \lambda_k + 1}^{\lambda_{k-1} - N_{k-1s}}\quad 
                            \qfrac{k - i + \lambda_i - a}{k - i + \lambda_i - a - N_{is}}\nonumber\\
\intertext{Finally, we adjust our index ranges and clean up the formula:}   
                          &=\prod_{\substack{1 \le i\le s\\ N_{is} \neq 0}}\quad
                            \prod_{i < k \le s}\quad
                            \prod_{\ell = 1}^{\lambda_{k-1} - \lambda_k - N_{k-1s}}\quad 
                            \qfrac{k - i + \lambda_i - \lambda_k - \ell }{k - i + \lambda_i - \lambda_k - \ell - N_{is}}\nonumber\\
                          &=\prod_{\substack{1 \le i\le s\\ N_{is} \neq 0}}\quad
                            \prod_{i < k \le s}\quad
                            \prod_{\ell = 1}^{\lambda_{k-1} - \lambda_k - N_{k-1s}}\quad 
                            \qfrac{c_{i,k} - \ell }{c_{i,k} - N_{i,s} - \ell}.
\end{align}
% query on hard brackets; query on order of the products
\vspace{3em}
\noindent
%referee claims sign error
On the other hand, recall from \cref{eq:tolstoy_rr2} that Tolstoy gives
\begin{equation}
    \mathcal N'(T'; T)^2  = \prod_{1\le i<k\le j\le n}\qbinom{c_{i,k} + M_{i,k}^j - 1 }{N_{i,j}} \;\big/ \prod_{1\le i<k< j\le n}\qbinom{c_{i,k} + M_{i,k}^{j-1}+N_{i,j}}{N_{i,j}} .
\end{equation}
We will manipulate this until it matches \cref{eq:big_elias} exactly.
Note that
\begin{align}
    c_{i,k} &= c'_{i,k} + N_{i,s} - N_{k,s}\\
    M^j_{i,k} &= \begin{cases}
    M'^j_{i,k} + N_{k,s} - N_{i,s} & j < s \\
    0 & j \ge s
    \end{cases}
\intertext{so that}
    c_{i,k} + M^j_{i,k} &= \begin{cases}
    c'_{i,k} + M'^j_{i,k} & j < s \\
    c'_{i,k} + N_{i,s} - N_{k,s} & j \ge s
    \end{cases}\\
    c_{i,k} + M^{j-1}_{i,k} +N_{i,j} &= \begin{cases}
    c'_{i,k} +  M'^{j-1}_{i,k} + N'_{i,j} & j < s \\
    c'_{i,k} + N_{i,s}  & j \ge s
    \end{cases}   .
\end{align}
As such, in $\mathcal N(T'; T)^2 $, all the terms with $j<s$ cancel and we are left with
\begin{align}\label{eq:big_carter}
\mathcal N(T'; T)^2  &= \prod_{1\le i<k\le s}\qbinom{c_{i,k} - 1 }{N_{i,s}} \big/
               \prod_{1\le i<k< s}\qbinom{c_{i,k} + N_{k,s}}{N_{i,s}}   .
\nonumber\\
\intertext{Putting in a $k=s$ case to the second product (noting $\lambda_s = N_{s,s}$), and splitting the indices $i$ and $k$, we obtain}
            &= \prod_{1\le i< s}\left(\qbinom{c_{i,s} + \lambda_s }{N_{i,s}}\times
               \prod_{i<k\le s} \qbinom{c_{i,k} -1}{N_{i,s}} \big/
               \qbinom{c_{i,k} + N_{k,s}}{N_{i,s}}\right)   .    \nonumber\\
\intertext{Clearly the only terms contributing anything are those where $N_{i,s} \neq 0$}
            &= \prod_{\substack{1 \le i\le s\\ N_{i,s} \neq 0}}
               \left(\qbinom{c_{i,s} + \lambda_s }{N_{i,s}}\times
               \prod_{i<k\le s} \frac{[c_{i,k}-1] \cdots [c_{i,k}-N_{i,s}]}{[c_{i,k}+N_{k,s}] \cdots [c_{i,k}+N_{k,s}-N_{i,s}]}\right)   .
\nonumber\\
\intertext{We introduce and remove a factor (note that each of $c_{i,k}-1$ down to $c_{i,k}-N_{i,s}$ are non-zero so that this factor is non-zero)}
            &= \prod_{\substack{1 \le i\le s\\ N_{i,s} \neq 0}}
               \left(\qbinom{c_{i,s} + \lambda_s }{N_{i,s}}\times
               \prod_{i<k\le s} \frac{[c_{i,k}-1] \cdots [c_{i,k}-N_{i,s}]}{[c_{i,k+1}-1] \cdots [c_{i,k+1}-N_{i,s}]}
               \frac{[c_{ik+1}-1] \cdots [c_{i,k+1}-N_{i,s}]}{[c_{i,k}+N_{k,s}] \cdots [c_{i,k}+N_{k,s}-N_{i,s}]}\right)
\nonumber\\
\intertext{to allow us to telescope the product}
            &= \prod_{\substack{1 \le i\le s\\ N_{i,s} \neq 0}}
               \left(\qbinom{c_{i,s} + \lambda_s }{N_{i,s}}\times
               \frac{[c_{i,i+1}-1] \cdots [c_{i,i+1}-N_{i,s}]}{[c_{i,s+1}-1] \cdots [c_{i,s+1}-N_{i,s}]}\times
               \prod_{i<k\le s}
               \frac{[c_{i,k+1}-1] \cdots [c_{i,k+1}-N_{i,s}]}{[c_{i,k}+N_{k,s}] \cdots [c_{i,k}+N_{k,s}-N_{i,s}]}\right)
\nonumber\\
\intertext{but $c_{i,s+1} = c_{i,s} + \lambda_s + 1$ and $c_{i,i+1}-1 = \lambda_{i} - \lambda_{i+1}$, so we can simplify}
            &= \prod_{\substack{1 \le i\le s\\ N_{i,s} \neq 0}}\left(\qbinom{\lambda_i - \lambda_{i+1} }{N_{i,s}}\times
               \prod_{i<k\le s}
               \frac{[c_{i,k+1}-1] \cdots [c_{i,k+1}-N_{i,s}]}{[c_{i,k}+N_{k,s}] \cdots [c_{i,k}+N_{k,s}-N_{i,s}]}\right)
\nonumber\\
\intertext{and writing $c_{i,k+1} = c_{i,k} + \lambda_k - \lambda_{k+1}+1$,}
            &= \prod_{\substack{1 \le i\le s\\ N_{i,s} \neq 0}}
               \left(\qbinom{\lambda_i - \lambda_{i+1} }{N_{i,s}}\times
               \prod_{i<k\le s}
               \frac{[c_{i,k}+\lambda_k-\lambda_{k+1}] \cdots [c_{i,k}+\lambda_k-\lambda_{k+1}-N_{i,s}+1]}{[c_{i,k}+N_{k,s}] \cdots [c_{i,k}+N_{k,s}-N_{i,s}]}\right)   .
\nonumber\\
\intertext{But now $N_{k,s} \le \lambda_k - \lambda_{k-1}$ so that $c_{i,k}+N_{k,s} \le c_{i,k}+\lambda_k-\lambda_{k+1}$ and so (introducing some more terms if necessary)}
            &= \prod_{\substack{1 \le i\le s\\ N_{i,s} \neq 0}}
               \left(\qbinom{\lambda_i - \lambda_{i+1} }{N_{i,s}}\times
               \prod_{i<k\le s}
               \frac{[c_{i,k}+\lambda_k-\lambda_{k+1}] \cdots [c_{i,k}+N_{k,s}+1]}{[c_{i,k}+\lambda_k-\lambda_{k+1}-N_{i,s}] \cdots [c_{i,k}+N_{k,s}-N_{i,s}]}\right)    .
\nonumber\\
\intertext{We can re-introduce the $c_{i,k+1}$}
            &= \prod_{\substack{1 \le i\le s\\ N_{i,s} \neq 0}}
               \left(\qbinom{\lambda_i - \lambda_{i+1} }{N_{i,s}}\times
               \prod_{i<k\le s}
               \frac{[c_{i,k+1}-1] \cdots [c_{i,k+1}-\lambda_k + \lambda_{k+1}+N_{k,s}]}{[c_{i,k+1}-N_{i,s}-1] \cdots [c_{i,k+1}-\lambda_k+\lambda_{k+1}+N_{k,s}-N_{i,s}-1]}\right)
\nonumber\\
\intertext{and write the product using a dummy index:}
            &= \prod_{\substack{1 \le i\le s\\ N_{i,s} \neq 0}}
               \left(\qbinom{\lambda_i - \lambda_{i+1} }{N_{i,s}}\times
               \prod_{i<k\le s}
               \prod_{\ell = 1}^{\lambda_{k} - \lambda_{k+1}-N_{k,s}}
               \frac{[c_{i,k+1}-\ell]}{[c_{i,k+1}-N_{i,s}-\ell]}\right)   .
\nonumber\\ 
\intertext{Now, if we replace all the $k+1$s with $k$s we pick up two terms on each side of the product range.}
            &= \prod_{\substack{1 \le i\le s\\ N_{i,s} \neq 0}}\left(
               \qbinom{\lambda_i - \lambda_{i+1} }{N_{i,s}}\times
               \prod_{i<k\le s}
               \prod_{\ell = 1}^{\lambda_{k-1} - \lambda_{k}-N_{k-1,s}}
               \frac{[c_{i,k}-\ell]}{[c_{i,k}-N_{i,s}-\ell]}
               \right.\nonumber\\
               &\left.\quad\quad\quad\quad\quad\quad
               \times\prod_{\ell = 1}^{\lambda_{s} - \lambda_{s+1}-N_{s,s}}
               \frac{[c_{i,s+1}-\ell]}{[c_{i,s+1}-N_{i,s}-\ell]}
               \times
               \prod_{\ell = 1}^{\lambda_{i} - \lambda_{i+1}-N_{i,s}}
               \frac{[c_{i,i+1}-N_{i,s}-\ell]}{[c_{i,i+1}-\ell]}
            \right)    .
\nonumber\\
\intertext{Note that $\lambda_{s+1} = 0$ and $\lambda_{s} = N_{s,s}$ so the second-to-last product is actually empty. Hence}
            &= \prod_{\substack{1 \le i\le s\\ N_{i,s} \neq 0}}\left(
               \qbinom{\lambda_i - \lambda_{i+1} }{N_{i,s}}\times
               \prod_{i<k\le s}
               \prod_{\ell = 1}^{\lambda_{k-1} - \lambda_{k}-N_{k-1,s}}
               \frac{[c_{i,k}-\ell]}{[c_{i,k}-N_{i,s}-\ell]}
               \times
               \prod_{\ell = 1}^{\lambda_{i} - \lambda_{i+1}-N_{i,s}}
               \frac{[c_{i,i+1}-N_{i,s}-\ell]}{[c_{,ii+1}-\ell]}
            \right)   .
\nonumber\\
\intertext{Finally, recall that $c_{i,i+1} = \lambda_i-\lambda_{i+1}+1$ and so}
            &= \prod_{\substack{1 \le i\le s\\ N_{i,s} \neq 0}}\left(
               \qbinom{\lambda_i - \lambda_{i+1} }{N_{i,s}}\times
               \prod_{i<k\le s}
               \prod_{\ell = 1}^{\lambda_{k-1} - \lambda_{k}-N_{k-1,s}}
               \frac{[c_{i,k}-\ell]}{[c_{i,k}-N_{i,s}-\ell]}
               \times
               \frac{[N_{i,s}]![\lambda_i - \lambda_{i+1}-N_{i,s}]!}{[\lambda_i - \lambda_{i+1}]!}
            \right)
\nonumber\\
            &= \prod_{\substack{1 \le i\le s\\ N_{i,s} \neq 0}}\quad
               \prod_{i<k\le s}
               \prod_{\ell = 1}^{\lambda_{k-1} - \lambda_{k}-N_{k-1,s}}
               \frac{[c_{i,k}-\ell]}{[c_{i,k}-N_{i,s}-\ell]}\nonumber .
\end{align}
This matches exactly 
% missing prime?  Full stops and state terms non-zero
with \cref{eq:big_elias} and so we deduce that
$$\mathcal N'(T'; T)^2  = \kappa_{\lambda'^\perp}^{\lambda^\perp}$$
as desired.
\end{proof}

\begin{theorem}
    Elias' conjecture (\cref{conj:elias}) holds.
\end{theorem}
\begin{proof}
Note the construction of $|T\rangle$ as a product of successive $F_-(T^{(j-1)}; T^{(j)})$ acting on the highest weight vector.
Each $F_-(T^{(j-1)}; T^{(j)})$ is in turn (up to a normalising factor) simply a extremal projector composed with $E_\mu$. Since we know that extremal projectors are sent to clasps, we thus recover the construction in \cref{eq:inductive_diagram_2}.

We know that these vectors are orthonormal. The clasped light ladders are certainly orthogonal, but they must be normalised by sucessive $\kappa$ as per \cref{conj:elias2} to have unit inner product with themselves. As per the discussion in \cref{subsec:cellular2}, this inner product is identified between the Weyl module and webs and so the normalisation factor $\mathcal{N}'(T';T)$ is the required value.

However, we have seen in \cref{lem:long} that this normalising factor is exactly the conjectured intersection form.
\end{proof}

\section{Acknowledgments}
The authors are indebted to the reviewer for their detailed and substantial comments.
The second author was supported by a DTP studentship from the Department of Pure Mathematics and Mathematical Sciences of the University of Cambridge funded by the Engineering and Physical Sciences Research Council under grant EP/N509620/1.

\bibliographystyle{alphaabbr}
\bibliography{all_cite}

\end{document}